\documentclass[12pt,reqno]{amsart}
\usepackage{amsfonts,amssymb,amsmath}
\usepackage{enumitem}
\usepackage{datetime}

\usepackage{xcolor}

\renewenvironment{proof}{{\bfseries Proof.}}{\qed}

\numberwithin{equation}{section} 
 
\newtheorem{theorem}{Theorem}[section] 
\newtheorem{proposition}[theorem]{Proposition} 
 
\newtheorem{lemma}[theorem]{Lemma} 

\theoremstyle{definition}
 
\newtheorem{remark}[theorem]{Remark} 

\usepackage[colorlinks=true,citecolor=cyan,urlcolor=blue,linkcolor=blue]{hyperref}

\newcommand*{\bigchi}{\mbox{\Large$\chi$}}
 
\def\D{\mathbb {D}}
\def\R{\mathbb {R}}
\def\C{\mathbb {C}}
\def\N{\mathbb {N}}
\def\H{\mathbb {H}}

\def\O{\mathbb {O}}
\def\E{\mathbb {E}}

\def\sgn{\mathbf {sgn}}
\def\Db{\mathbf {D}}
\def\Ab{\mathbf {A}}
\def\ib{\mathbf {i}}
\def\jb{\mathbf {j}}
\def\kb{\mathbf {k}}

\def\d{\mathbf{ d}}

\def\ZC{\mathcal {Z}}
\def\CC{\mathcal {C}}
\def\PC{\mathcal {P}}
\def\NC{\mathcal {N}}
\def\OC{\mathcal {O}}
\def\BC{\mathcal {B}}
\def\YC{\mathcal {Y}}
\def\EC{\mathcal {E}}

\def\HC{\mathcal {H}}
\def\SC{\mathcal {S}}
\def\DC{\mathcal {D}}
\def\AC{\mathcal {A}}

\def\g{\mathfrak {g}}
\def\h{\mathfrak {h}}
\def\p{\mathfrak {p}}
\def\k{\mathfrak {k}}
\def\s{\mathfrak {s}}

\def\u{\mathfrak {u}}
\def\o{\mathfrak {o}}
\def\l{\mathfrak {l}}
\def\z{\mathfrak {z}}

\def\m{\mathfrak {m}}

\def\<>{\langle \cdot ,\, \cdot \rangle}

\def\lto{\longrightarrow}

\begin{document} 
 
\title[Homotopy type of nilpotent orbits]{Homotopy type of the nilpotent orbits in classical 
Lie algebras}

\author[I. Biswas]{Indranil Biswas} 

\address{School of Mathematics, Tata Institute of 
Fundamental Research, Homi Bhabha Road, Mumbai 400005, India}

\email{indranil@math.tifr.res.in} 

\author[P. Chatterjee]{Pralay Chatterjee}

\address{The Institute of Mathematical Sciences, HBNI, CIT Campus, 
Tharamani, Chennai 600113, India}

\email{pralay@imsc.res.in} 

\author[C. Maity]{Chandan Maity}

\address{Indian Institute of Science Education and Research (IISER) Mohali,
 Knowledge City, Sector 81, S.A.S. Nagar 140306, Punjab, India}

\email{cmaity@iisermohali.ac.in, maity.chandan1@gmail.com} 

\subjclass[2010]{57T15, 17B08}

\keywords{Nilpotent orbit, classical groups, homogeneous spaces.}

\begin{abstract}
In \cite{BCM} homotopy types of nilpotent orbits are explicitly 
described in the case of real simple classical Lie algebras for which any maximal compact 
subgroup in the associated adjoint group is not semisimple. In this paper we extend the above 
description of homotopy type of nilpotent orbits to the remaining cases of real simple 
classical Lie algebras for which any maximal compact subgroup in the associated adjoint 
group is semisimple.
\end{abstract}

\maketitle

\tableofcontents

\section{Introduction}\label{sec-introduction}

Let $\g$ be a real simple Lie algebra, and let $G$ be the associated adjoint Lie group. An 
element $X \in \g$ is called {\it nilpotent} if ${\rm ad}(X): \g \longrightarrow \g$ is a 
nilpotent operator. For any nilpotent $X$, let $\OC_X \,:=\, \{{\rm Ad}(g)X \,\big\vert\,\, g \,\in 
\,G \}$ be the corresponding {\it nilpotent orbit} under the adjoint action of $G$ on $\g$. 
Such nilpotent orbits form a rich class of homogeneous spaces. In fact they lie on the interface 
of several disciplines in mathematics such as Lie theory, symplectic geometry, representation 
theory, algebraic geometry; see \cite{CoM}, \cite{M}. Nevertheless, surprisingly for a very 
long period there seems to have been hardly any literature on the topological invariants of 
these orbits, other than the description of fundamental group in the case of simple Lie 
algebras. The computation of the fundamental groups of such orbits is attributed to T. 
Springer and R. Steinberg \cite{SS} for the classical case and A. Alexeevski \cite{A} for the 
complex exceptional case; see \cite[Corollary 6.1.6, p. 91, pp. 128--134]{CoM}, \cite[pp. 
229--230]{M}. We also refer the reader to the works of D. King \cite{Ki} and E. Sommers 
\cite{So} in this regard.

It is only recent that attentions have been drawn to topological invariants of such orbits 
other than the fundamental group; see \cite{Ju} and \cite{Cr}. The first two authors
began their study on this topic in \cite{BC} where the second cohomology groups of the 
nilpotent orbits in all the complex simple Lie algebras were computed; see \cite[Theorems 5.4, 
5.5, 5.6, 5.11, 5.12]{BC}. Each adjoint orbit in a semisimple Lie algebra is equipped with the 
Kostant-Kirillov two form. The motivation for studying the second cohomology groups stemmed 
from, besides computing some invariant other than the fundamental group, the exactness 
criterion obtained in \cite[Proposition 1.2]{BC} for the Kostant-Kirillov form on adjoint 
orbits of arbitrary elements in the Lie algebra of a real semisimple Lie group with a 
semisimple maximal compact subgroup. It may be mentioned that \cite[Proposition 1.2]{BC} 
generalizes \cite[Theorem 1.2]{ABB} where the exactness criterion is obtained for the 
Kostant-Kirillov form on an adjoint orbit of a semisimple element in a complex semisimple Lie 
algebra. In \cite{CM} the second cohomology groups of nilpotent orbits are computed for most 
of the nilpotent orbits in non-compact non-complex exceptional Lie algebras. For the rest of 
cases of nilpotent orbits, which are not covered in the above computations, an upper bound for 
the dimension of the second cohomology group is obtained; see \cite[Theorems 3.2--3.13]{CM}.

To describe the results of this paper we need to recall our recent work in \cite{BCM}. In 
\cite{BCM}, considering all the non-complex and non-compact real classical Lie algebras, we 
have given a complete description of the second and the first cohomology groups of all the 
nilpotent orbits in terms of their standard parametrizations involving the (signed) Young 
diagram. In doing so, first generalizing \cite[Theorem 3.3]{BC} we obtained a computable 
description of the second and first cohomologies of a general connected homogeneous space in 
terms of the ambient Lie group and the associated quotienting (closed) subgroup. In the setting of 
\cite{BCM} it was convenient to assume that the simple real Lie group $G$ is the connected 
component of the $\R$-points of a $\R$-simple algebraic groups defined over $\R$.

Setting 
$\g\,:=\, {\rm Lie}\, G$, let $X \,\in \,\g$ be a non-zero nilpotent element, and let $\OC_X$ 
be its adjoint $G$-orbit. A Lie theoretic reformulation of the second and the first cohomology 
groups of $\OC_X$ was obtained in \cite[Theorem 4.2]{BCM}, incorporating a $\s\l_2 
(\R)$-triple containing $X$; the computations in \cite{CM} also use this result crucially. Let 
$\{X,\,H,\,Y\}$ be a $\s\l_2(\R)$-triple in $\g$, containing $X$, and let $\ZC_G (X,H,Y)$ be 
the centralizer of the triple $\{X,\,H,\,Y\}$ in $G$. Let $K$ be a maximal compact subgroup in 
$\ZC_G (X,H,Y)$, and $M$ be a maximal compact subgroup in $G$ containing $K$. Let $\m$ and 
$\z(\k)$ be the Lie algebras of $M$ and the center of $K$, respectively. Then \cite[Theorem 
4.2]{BCM} says that the computation of the second cohomology of the nilpotent orbits boils 
down to understanding the action of the component group $K/K^\circ$ on the subalgebra 
$\z(\k)\cap [\m,\,\m]$. Thus, when $M$ is semi-simple, this amounts to describing the action 
of $K/K^\circ$ on $\z(\k)$, and hence knowing the isomorphism class of $K$ is enough to 
compute the second cohomology group in this case.

However, when $M$ is not semisimple, it does not suffice
to know the isomorphism classes of $K$ and $M$, rather what is needed is an understand of
the embedding of $K$ in $M$. The case of ${\s\u}(p,q)$ dealt in \cite[\S~4.4]{BCM} constitutes 
a typical example of such a situation. Although the isomorphism class of $M$ is a standard 
known object in the case of a non-compact classical simple real Lie group, and the isomorphism 
class of $K$ can be obtained immediately using either the work of Springer and Steinberg 
\cite{SS} (see also \cite[Lemma 4.4]{BCM}), hardly anything can be concluded, from these 
isomorphism classes, on how $K$ is embedded in $M$. Consequently, one of the major objectives in 
\cite{BCM} was to compute this embedding explicitly for all the nilpotent orbits in the 
classical real simple Lie algebras $\g$ for which maximal compact subgroup of $G$ is not 
semisimple. This situation occurs in the cases when $\g$ is either $\s\u(p,q)$ or $\s\o^*(2n)$ 
or $ \s\p(n,\R)$.

It follows from a minor variation of a general fact due to Mostow (see Theorem \ref{mostow}, 
\cite[Theorem 3.2]{BC}) that $M/K$ naturally embeds in $G/\ZC_G (X)$ as a deformation retract. 
In particular, the compact submanifold $M/K$ of $G/\ZC_G (X)$ is in the same homotopy class as 
that of $G/\ZC_G (X)$; see Theorem \ref{mostow-corollary} for details. Thus computations in 
\cite{BCM} in fact yield compact sub-homogeneous spaces as convenient and optimal homotopy 
types of nilpotent orbits in the case when maximal compact subgroups of $G$ are not semisimple 
(we refer to \cite[Propositions 4.14, 4.30~and~ 4.36]{BCM}). It should be mentioned here that, for 
a certain technical reason, the homotopy types of the nilpotent orbits in $\s\o(p,q)$ were 
also described in \cite[Proposition 4.22]{BCM} under the following assumption on the partition 
associated to the parametrization: $\N_\d\,=\,\O_\d$; see \eqref{Nd-Ed-Od} for the definitions 
of $\N_\d$ and $\O_\d$.

The objective of this paper is to extend the computations and complete the project, initiated 
in \cite{BCM}, of describing optimal homotopy types of nilpotent orbits by giving explicit 
embedding of maximal compact subgroups $K$ of $\ZC_G (X)$ in $M$ for the remaining cases of 
simple Lie algebras $\g$ for which $M$ is semisimple; see Theorems \ref{homototy-type-sl-nc}, 
\ref{homototy-type-sl-nr}, \ref{homototy-type-sl-nh} \ref{max-cpt-so-nc-wrt-onb}, 
\ref{max-cpt-0-wrt-onb-sopq}, \ref{max-cpt-sp-nC-wrt-basis} and \ref{max-cpt-sp-pq-wrt-onb}.

The description of a suitable reductive part of the centralizer in $G$ of a nilpotent element 
in $\g$, when $\g$ is isomorphic to a complex simple Lie algebra or
it is isomorphic to one of the Lie algebras 
$\g\l_n(\R)$, $\u(p,q)$, $\o(p,q)$ and $\s\p(n,\R)$, is due to Springer and Steinberg 
\cite{SS}. Since we are unable to find such descriptions in the literature when the classical 
simple Lie algebras are matrix subalgebras with entries from $\H$, we 
record them here in the final section as an appendix.

The paper is organized as follows. In Section \ref{sec-notation} we fix some notation and terminology and we recall some necessary 
background. The explicit homotopy types of the nilpotent orbits are described in Section \ref{sec-homotpy-type-nil-orbts}; they are spread 
across Theorems \ref{homototy-type-sl-nc}, \ref{homototy-type-sl-nr}, \ref{homototy-type-sl-nh} \ref{max-cpt-so-nc-wrt-onb}, 
\ref{max-cpt-0-wrt-onb-sopq}, \ref{max-cpt-sp-nC-wrt-basis} and \ref{max-cpt-sp-pq-wrt-onb}.

\section{Notation and background}\label{sec-notation}

In this section we fix the notation, and recall some background material which will be used throughout. 
Subsequently, a few specialized notation are mentioned as and when they are needed. The notation and the 
background introduced in this section overlap with those in \cite[\S~2]{BCM} to some extent. However, for the sake 
of completeness and clarity of the exposition we also include them here.

Once and for all fix a square root of $-1$ and call it $\sqrt{-1}$. The Lie groups will be denoted by the capital 
letters, while the Lie algebra of a Lie group will be denoted by the corresponding lower case German letter, unless 
a different notation is explicitly mentioned. Sometimes, for notational convenience, the Lie algebra of a Lie group 
$G$ is also denoted by ${\rm Lie} (G)$. The connected component of $G$ containing the identity element is denoted 
by $G^{\circ}$, and the commutator subgroup of $G$ is denoted by $(G,G)$. For a subgroup $H$ of $G$, and a subset 
$S$ of $\g$, the {\it centralizer} of $S$ in $H$ is $$\ZC_{H} (S)\,:=\, \{h\,\in\, H \,\,\big\vert\,\,
{\rm Ad}(h)Y \,=\, Y\, \text{ for all 
}\ Y \,\in\, S \}\,.$$ Similarly, for a Lie subalgebra $\h \,\subset\, \g$, by $\z_\h (S)$ we denote the subalgebra $$\{X 
\,\in\, \h \,\,\big\vert\,\, [X,\,Y]\,=\,0\ \text{ for all }\ Y\,\in\, S \}\,.$$

Let $G$ be a semisimple Lie group. An element $X \,\in\, \g$ is called {\it nilpotent} if ${\rm ad}(X) \,\colon\, \g \,
\longrightarrow\, \g$ is a nilpotent operator. The {\it set of nilpotent elements} in $\g$ is denoted by ${\NC}_{\g}$.
For any $X \,\in\, 
\NC_\g$, define $$\OC_X \,:= \,\{{\rm Ad} (g)X\,\big\vert\,\, g \,\in\, G \}$$ to be the {\it nilpotent orbit} of $X$ in $\g$. The 
{\it set of all nilpotent orbits in $\g$} is denoted by ${\NC}(G)$.

\subsection{Classical Lie groups and their Lie algebras}
The notation $\D$ will stand for either $\R$ or $\C$ or $ \H$, unless mentioned otherwise. Let $V$ be a right 
vector space over $\D$. Let ${\rm End}_\D (V)$ be the right $\R$-algebra of {\it $\D$-linear endomorphisms} of $V$, 
and let ${\rm GL}(V)$ be the {\it group of invertible elements} of ${\rm End}_\D (V)$. For a $\D$-linear 
endomorphism $T \,\in\, {\rm End}_\D (V)$, and an ordered $\D$-basis $\BC$ of $V$, the {\it matrix of $T$ with 
respect to $\BC$} is denoted by $[T]_{\BC}$. When $\D$ is either $\R$ or $\C$, let
$$
{\rm tr} \,:\, {\rm End}_\D (V) \,\longrightarrow\, \D\, \ \ \text{ and }\ \
{\rm det} \,:\, {\rm End}_\D V \,\longrightarrow\, \D
$$
respectively be the usual {\it trace} and {\it determinant} maps.
When $\D \,=\, \R$ or $ \C$, define 
$$
{\rm SL}(V)\,:=\, \{ z \,\in \,{\rm GL}(V)\,\big\vert\,\, \text{det}(z)\,= \,1\} \ \ \text{ and } \ \ \s\l (V ) \,:=\,
\{ y \,\in\, {\rm End}_\D (V) \,\big\vert\,\, \text{tr}(y)\,=\, 0 \}\, .
$$ 
If $\D \,= \,\H$, then define
$$
{\rm SL}(V)\,:=\, \big({\rm GL} (V),\, \, {\rm GL}(V)\big) \quad 
\text{ and } \quad
\s\l (V )\,:=\, \big[ {\rm End}_\D(V),\, {\rm End}_\D(V)\big] \,.
$$

Let $\D$ be $\R$, $\C$ or $\H$, as above. Let $\sigma$ be either the identity map $\text{Id}$ or an {\it 
involution} of $\D$; i.e., $\sigma$ is a $\R$-linear map with $\sigma^2 \,=\, \text{Id}$ and
$\sigma (xy) \,=\, \sigma (y) \sigma (x)$ for all $x,\, y \,\in\, \D$. Let $\epsilon \,=\,
\pm 1$. Following \cite[\S~23.8, p. 264]{Bo} we call a map
$$\langle \cdot,\, \cdot \rangle \,\colon\, V \times V \,\longrightarrow\, \D$$
a $\epsilon$-$\sigma$ {\it Hermitian form} if
\begin{itemize}
\item $ \langle \cdot,\, \cdot \rangle $ is additive in each argument, 
\item $ \langle v,\, u \rangle \,= \, \epsilon \sigma( \langle u, v \rangle)$, and
\item $ \langle v \alpha,\, u \rangle \,=\, \sigma (\alpha) \langle v,\, u \rangle$ for all
$v,\,u \,\in\, V$ and for all $\alpha \,\in\, \D$.
\end{itemize}

A $\epsilon$-$\sigma$ Hermitian form $ \langle \cdot, \, \cdot \rangle $ is called {\it
non-degenerate} if $ \langle v,\, u \rangle \,=\,0 $ for all $v$ if and only if $u \,=\, 0$.
All $\epsilon$-$\sigma$ Hermitian forms considered here will be assumed to be non-degenerate.
For a $\epsilon$-$\sigma$ Hermitian form $\<>$, define 
$$
{\rm U} (V,\, \langle \cdot,\, \cdot \rangle ) \,:=\,
\{T \,\in\, {\rm GL}(V) \,\big\vert\,\, \langle Tv ,\, Tu \rangle \,= \, \langle v ,
\, u \rangle\ ~~ \forall~~ v,\,u \,\in \,V \}$$ and
$$\u (V,\, \langle \cdot, \,\cdot \rangle ) \,:=\, \{T\,\in\, {\rm End}_\D(V)\,\big\vert\,\, 
\langle Tv ,\,u \rangle + \langle v ,\, Tu \rangle \,= \,0\ ~~\forall~~ v,\,u \,\in \,V \}\, .$$ 
We next define 
$$
{\rm SU} (V, \langle \cdot, \cdot \rangle )\,: =\,
{\rm U} (V, \langle \cdot, \cdot \rangle ) \cap {\rm SL}(V)\ \ \text{ and } \ \
\s\u (V, \langle \cdot, \cdot \rangle ) \,:=\,
\u (V, \langle \cdot, \cdot \rangle ) \cap \s\l (V)\, .
$$ 

We define the {\it usual conjugations} $\sigma_c$ on $\C$ and on $\H$, respectively, by 
$\sigma_c\,\colon\,\C\,\lto\,\C$, $ (x_1 + \sqrt{-1}x_2 )\,\longmapsto\, x_1-\sqrt{-1}x_2$, and 
$$\sigma_c\,\colon\, \H\,\lto\, \H,\ \ (x_1 + \ib x_2 + \jb x_3 + \kb x_4 )\,\longmapsto\, x_1 - \ib x_2 - \jb x_3 
- \kb x_4\, ,$$ $x_i \,\in\, \R$ for $i\,=\, 1,\, \cdots ,\, 4$. As discussed in \cite[\S~2]{BCM}, without loss of 
any generality, we may consider the involution $\sigma_c$ instead of arbitrary involution $\sigma$ in a 
$\epsilon$-$\sigma$ Hermitian form $\<>$. Therefore, from now on we will restrict ourselves to the involution $\sigma_c$ on $\D$.

We next introduce certain standard nomenclature associated to the specific values of $\epsilon$ and $\sigma$. If
$\sigma\,=\, \sigma_c$ and $\epsilon \,=\,1$, then $\langle \cdot ,\, \cdot \rangle$
is called a {\it Hermitian} form.
When $\sigma \,= \,\sigma_c$ and $\epsilon \,=\,-1$, then $\langle \cdot ,\, \cdot \rangle$
is called a {\it skew-Hermitian} form. 
The form $\langle \cdot, \, \cdot \rangle$ is called {\it symmetric} if
$\sigma \,=\, \text{Id}$ and $\epsilon \,=\,1 $. Lastly, if
$\sigma \,=\, \text{Id}$ and $\epsilon \,=\, -1 $, then $\langle \cdot, \, \cdot \rangle$ 
is called a {\it symplectic} form.
If $\langle \cdot, \, \cdot \rangle$ is a symmetric form on $V$, define
$${\rm O} (V,\, \langle \cdot,\, \cdot \rangle)\,:=\, {\rm U} (V, \,\langle \cdot,\, \cdot \rangle )
\ \ \text{ and }\ \ 
\o (V,\, \langle \cdot,\, \cdot \rangle )\,:=\, \u (V,\, \langle \cdot,\, \cdot \rangle)\, .$$
Similarly, if $\langle \cdot, \, \cdot \rangle$ is a symplectic form on $V$, then define
$${\rm Sp} (V,\, \langle \cdot,\, \cdot \rangle)\,:=\, {\rm SU} (V,\, \langle \cdot,\, \cdot \rangle )
\ \ \text{ and }\ \ \s\p (V, \,\langle \cdot, \,\cdot \rangle) \,:=\, \s\u (V,\,
\langle \cdot,\, \cdot \rangle )\, .$$
When $\D\,=\,\H$ and $\<>$ is a skew-Hermitian form on $V$, define
$${\rm SO}^*(V,\, \langle \cdot,\, \cdot \rangle)\,:=\,{\rm SU}(V,\, \langle \cdot,\, \cdot \rangle )
\ \ \text{ and } \ \ \s\o^* (V,\, \langle \cdot,\, \cdot \rangle )\,:=\, \s\u (V,\, \langle \cdot,\,
\cdot \rangle )\, .$$

We next introduce some terminology associated to certain types of $\D$-basis of $V$.
For a symmetric or Hermitian form $\<>$ on $V$, an orthogonal basis $\AC$ of $V$
is called {\it standard orthogonal} if $\langle v,\, v \rangle \,=\, \pm 1$ for all $v \,\in \,\AC$. For a standard orthogonal basis ${\AC}$ of $V$,
set $$p \,:=\, \# \{ v \,\in\, \AC \,\,\big\vert\,\, \langle v, \,v\rangle \,> \, 0\} \ \ \text{ and }\ \
q \,:=\, \# \{ v \,\in \,\AC \,\,\big\vert\,\, \langle v, \,v\rangle \, <\, 0\}\, .$$ 
The pair $(p,\,q)$, which is independent of the choice
of the standard orthogonal basis ${\AC}$, is called the {\it signature} of $\<>$.

When $\D \,=\,\R$ or $\C$ and $\<>$ is symplectic, the dimension $\dim_\D V$ is an even. Let $2n \,=\, \dim_\D V$. In
this case an ordered basis $\BC \,:=\, (v_1,\, \cdots, \,v_n;\, v_{n+1}, \,
\cdots ,\, v_{2n})$ of $V$ is said to be {\it symplectic} if $\langle v_i,\, v_{n+i} \rangle
\,=\, 1$ for all $1\,\leq\, i \,\leq\, n$ and $\langle v_i,\, v_j \rangle \,=\, 0$ for all $j \,\neq\, n +i$. 
The ordered set $(v_1,\, \cdots ,\, v_n)$ is called the {\it positive part} of $\BC$ and it is denoted by
$\BC_+$. Similarly, the ordered set $(v_{n+1},\, \cdots ,\, v_{2n})$ is called the {\it negative part} of $\BC$, and it is denoted by $\BC_-$. 
The {\it complex structure on $V$ associated to the above
symplectic basis $\BC$} is defined to be the $\R$-linear map $$J_{\BC} \,:\, V
\,\longrightarrow\, V\, , \ \ v_i\,\longmapsto\, v_{n+i}\, ,\ \ v_{n+i}\,\longmapsto\, -v_{i}
\quad \forall~ \ 1\,\leq\, i \,\leq\, n\, .$$
If $\D \,=\,\H$, and $\<>$ is a skew-Hermitian form on $V$,
an orthogonal $\H$-basis $$\BC\,:=\, (v_1,\, \cdots,\, v_m)$$
of $V$ ($m \,:=\, \dim_\H V$) is said to be {\it standard orthogonal} if
$\langle v_r,\, v_r \rangle \,=\, \jb$ for all $1 \,\leq\, r \,\leq\, m$ and
$\langle v_r, \, v_s \rangle \,=\, 0$ for all $r \,\neq \,s$.

For $P\,=\, (p_{ij}) \,\in\, {\rm M}_{r\times s}(\D)$, let $P^t$ denote the {\it transpose} of $P$. If
$\D \,=\, \C$ or $\H$, then define $\overline{P} \,:=\, (\sigma_c(p_{ij}))$. Let
\begin{equation}\label{defn-I-pq-J-n}
{\rm I}_{p,q} \,:=\, \begin{pmatrix}
{\rm I}_p \\
& -{\rm I}_q
\end{pmatrix}\, , \quad
{\rm J}_n \,:=\, \begin{pmatrix}
& -{\rm I}_n \\
{\rm I}_n & 
\end{pmatrix}\,.
\end{equation}
The classical groups and Lie algebras that we will be working with are:
\begin{align*}
&{\rm SL}_n (\C)\,:=\, \{g\,\in\, {\rm GL}_n (\C)\,\big\vert\, \det (g) \,=\,1 \}, \qquad \,\,\,
{\s\l}_n (\C)\,:=\, \{ Y\,\in\, {\rm M}_n (\C)\,\big\vert\, \text{tr} ( Y)\, =\,0 \};
\\
&{\rm SL}_n (\R)\,:=\, \{g\,\in\, {\rm GL}_n (\R)\,\big\vert\, \det (g) \,=\,1\}, \qquad \,\,\,
{\s\l}_n (\R)\,:=\, \{ Y \,\in\, {\rm M}_n (\R)\,\big\vert\, \text{tr} ( Y) \,=\,0 \};
\\
&{\rm SL}_n (\H)\,:=\, \{g\,\in\, {\rm GL}_n (\H)\,\big\vert\, \text{Nrd}_{{\rm M}_n (\H)} (g)\,=\,1 \}, 
{\s\l}_n (\H)\,:=\, \{ Y\,\in\, {\rm M}_n (\H)\,\big\vert\, \text{Trd}_{{\rm M}_n (\H)} ( Y) \,=\,0 \};
\\
&{\rm SO} (n,\C)\,:=\, \{g\,\in\, {\rm SL}_{n}(\C)\,\big\vert\, g^t g\, =\, {\rm I}_{n} \},\qquad \, \,
{\s\o} (n,\C)\,:=\, \{ Y \,\in\, \s\l_{n}(\C)\,\big\vert\, Y^t + Y \,=\,0 \};
\\
&{\rm SO}(p,q)\,:=\, \{g \,\in\, {\rm SL}_{p+q}(\R)\,\big\vert\, g^t{\rm I}_{p,q} g\,=\, {\rm I}_{p,q} \},
{\s\o} (p,q)\,:=\, \{ Y \,\in\, \s\l_{p+q}(\R) \,\big\vert\, Y^t {\rm I}_{p,q} + {\rm I}_{p,q} Y \,=\,0 \};
\\
&{\rm Sp}(n,\C)\,:= \,\{g \,\in\, {\rm SL}_{2n}(\C)\,\big\vert\, g^t {\rm J}_{n} g \,=\,{\rm J}_{n} \}, \quad
{\s\p} (n,\C)\,:=\, \{ Y\,\in\, \s\l_{2n}(\C)\,\big\vert\, Y^t {\rm J}_{n} + {\rm J}_{n} Y \,=\,0 \};
\\
&{\rm Sp} (p,q)\,:=\, \{g\,\in\, {\rm SL}_{p+q}(\H)\,\big\vert\, \overline{g}^t{\rm I}_{p,q} g\,=\, {\rm I}_{p,q} \}, 
{\s\p} (p,q)\,:=\, \{Y\,\in\, \s\l_{p+q}(\H)\,\big\vert\, \overline{ Y}^t{\rm I}_{p,q} + {\rm I}_{p,q} Y \,=\,0 \}.
\end{align*}

For any group $H$, let $H^n_\Delta$ denote the diagonally embedded copy of $H$ 
in the $n$-fold direct product $H^n$. 
Let $V$ be a vector space over $\D$. Define $\mathfrak{d}_V \,: \,{\rm End}_\D (V) \,\longrightarrow\,
\D^*$ to be $\mathfrak{d}_V \,:=\, \det$ if $\D\,=\, \C$ or $\R$, and $\mathfrak{d}_V \,:=\,
{\rm Nrd}_{{\rm End}_\D V}$ if
$\D \,= \,\H$. Let now $V_i$, $1 \,\leq\, i \,\leq\, m$, be right vector spaces over $\D$.
As before, $\D$ is either $\R$ or $\C$ or $\H$. 
For every $1 \,\leq\, i \,\leq\, m$, let $H_i \,\subset \,{\rm GL} (V_i) $ be a 
matrix subgroup. Define the subgroup
$$
S\big( \prod_i H_i \big) \, :=\, \Big\{ (h_1,\, \cdots,\, h_m) \,\in\, \prod_{i=1}^m H_i \ \Bigm| \ 
\prod_i \mathfrak{d}_{V_i} (h_i) \,=\,1 \Big\}\, \subset\, \prod_{i=1}^m H_i\, .
$$

The following notation will allow us to write block-diagonal square matrices with many blocks in a convenient way.
For $r$-many square matrices $A_i \,\in\, {\rm M}_{m_i} (\D)$, $1 \,\leq\, i \,\leq\, r$, the block diagonal square
matrix of size $\sum m_i \times \sum m_i$, with $A_i$ as the $i$-th 
block in the diagonal, is denoted by $A_1 \oplus \cdots \oplus A_r$. This is also abbreviated as
$\oplus_{i =1}^r A_i$. Furthermore, if $B \,\in\, {\rm M}_m (\D)$ and $s$ is a positive integer,
then denote
$ B_\blacktriangle^s \,:=\, \underset{s\text{-many}}{\underbrace{B \oplus \cdots \oplus B}}$.

\subsection{Jacobson-Morozov Theorem}

For a Lie algebra $\g$ over $\R$, a subset $\{X,\,H,\,Y\} \,\subset\, \g$ is 
said to be a {\it $\s\l_2(\R)$-triple} if $X \,\neq\, 0$, $[H,\, X] \,=\, 2X$,\, $[H,\, Y] \,= \, -2Y$ and $[X,\, Y] 
\,=\,H$. We now recall a well-known result due to Jacobson and Morozov.

\begin{theorem}[{Jacobson-Morozov, cf.~\cite[Theorem~9.2.1]{CoM}}]\label{Jacobson-Morozov-alg}
Let $X\,\in\, \g$ be a non-zero nilpotent element in a real semisimple Lie algebra $\g$. Then there exist
$H,\,Y\, \in\, \g$ such that $\{X,\,H,\,Y\} \,\subset\, \g$ is a $\s\l_2(\R)$-triple.
\end{theorem}

\subsection{Finite dimensional \texorpdfstring{$\s\l_2(\R)$}{Lg}-modules}\label{notions-sl2-modules}

Given an endomorphism $T \,\in\, {\rm End}_\R (W)$, where $W$ is a $\R$-vector space, and any $\lambda \,\in\, \R$, set
$$W_{T, \lambda} \,:=\, \{ w \,\in\, W \,\,\big\vert\,\, T w \,= \,w \lambda \}\, .$$ 
Let $V$ be a right vector space of dimension $n$ over $\D$, where $\D$ is, as before, $\R$ or $ \C $ or $ \H$. 
Let $\{X,\,H,\,Y\} \,\subset\, \s\l (V )$ be a $\s\l_2(\R)$-triple. Note that $V$ is also
a $\R$-vector space using the inclusion $\R \,\hookrightarrow\, \D$. Hence $V$ is a module
over $ \text{ Span}_\R \{ X,\,H,\,Y\} \,\simeq\, \s\l_2(\R)$. 
For any positive integer $d$, let $M(d-1)$ denote the sum of all the $\R$-subspaces $A$ of $V$ such that
\begin{itemize}
\item $\dim_\R A \,= \,d$, and

\item $A$ is an irreducible $\text{Span}_\R \{ X,\,H,\,Y\}$-submodule of $V$.
\end{itemize} 
Then $M(d-1)$ is the {\it isotypical component} of $V$ containing all the irreducible submodules
of $V$ with highest weight $d-1$. Let
\begin{equation}\label{definition-L-d-1}
L(d-1)\,:= \,V_{Y,0} \cap M(d-1)\, .
\end{equation}
As the endomorphisms $X,\,H,\,Y$ of $V$ are $\D$-linear, 
the $\R$-subspaces $M(d-1)$, $V_{Y,0}$ and $L(d-1)$ of $V$ are also $\D$-subspaces. Let
\begin{equation}\label{dim-L-d-1}
t_{d} \,:=\, \dim_\D L(d-1)\, .
\end{equation}

\subsection{Preliminary results on topology of homogeneous spaces}\label{sub-sec-Prelm-rslt}

In this section we will recall some well-known results. The following one gives an equivalence of homotopy types 
between a non-compact homogeneous space and certain compact homogeneous space.

\begin{theorem}[\cite{Mo}]\label{mostow} 
Let $G$ be a connected Lie group, and let $H\,\subset\,G$ be a closed subgroup with finitely many connected 
components. Let $M$ be a maximal compact subgroup of $G$ such that $M \cap H$ is a maximal compact subgroup of $H$. 
Then the image of the natural embedding $M/ (M \cap H)\,\hookrightarrow\, G/H$ is a deformation retract of 
$G/H$. 
\end{theorem} 

Theorem \ref{mostow} is proved in \cite[p. 260, Theorem 3.1]{Mo} for connected $H$.
However, as mentioned in \cite{BC}, using \cite[p. 180, Theorem 3.1]{H}, the proof as in \cite{Mo}
goes through when $H$ has finitely many connected components.

The following result explains why one needs to identify a maximal compact subgroup of the centralizer of a $\s\l_2 
(\R)$-triple as a subgroup of a maximal compact subgroup of the ambient simple real Lie group.
 
\begin{theorem}\label{mostow-corollary} 
Let $\g$ be a real simple Lie algebra and $G$ be the real (connected) adjoint group of $\g$.
Let $X \in \g$ be a nilpotent element, and $\{ X, \, H, \, Y \}$ be a $\s\l_2 (\R)$-triple in $\g$ containing $X$.
Let $K$ be a maximal compact subgroup of $\ZC_G (X,\,H,\,Y)$. Then $K$ is a maximal compact
subgroup of $\ZC_G (X)$. Moreover, if $M$ is a maximal compact subgroup of $G$ containing $K$, then
the compact homogeneous space $M/K$ embeds in $G/\ZC_G (X)$ as a deformation retract. 
\end{theorem}

\begin{proof}
It follows from the proof of \cite[Theorem 4.2]{BCM}. The last part follows from Theorem \ref{mostow}.
\end{proof}

\subsection{Partitions and (signed) Young diagrams}\label{sec-partition-Young-diagram}

An {\it ordered set of order $n$} is a $n$-tuple $(v_1,\, \cdots ,\,v_n)$, where $v_1,\, \cdots,\, v_n$ are elements of some set, such that $v_i \,\neq\, v_j$ if $i \,\neq\, j$.
If $w \,\in\, \{ v_1,\, \cdots ,\,v_n\}$, then write $w \,\in \,(v_1,\, \cdots ,\,v_n)$. 
A pair of ordered sets $(v_1,\, \cdots ,\,v_n)$ and $(w_1,\, \cdots ,\,w_m)$ is said to be 
{\it disjoint } if  $\{v_1,\, \cdots ,\,v_n\} \cap \{w_1,\, \cdots ,\,w_m\}=\emptyset$. 
For a pair of ordered sets
$(v_1,\, \cdots ,\,v_n)$ and $(w_1,\, \cdots ,\,w_m)$ which is disjoint, 
the ordered set $(v_1,\, \cdots ,\,v_n,\, w_1,\, \cdots ,\, w_m)$ will be denoted by
$$(v_1, \,\cdots ,\,v_n) \vee (w_1,\, \cdots ,\,w_m)\, .$$
Furthermore, for $k$-many ordered sets $(v^i_1, \,\cdots,\, v^i_{n_i})$, $1 \,\leq\, i\, \leq\, k$, which are pairwise disjoint, we
define the ordered set $(v^1_1, \cdots, v^1_{n_1})  \vee \cdots \vee (v^k_1,\, \cdots,\, v^k_{n_k})$ 
to be the following juxtaposition of ordered sets $(v^i_1,\, \cdots, \,v^i_{n_i})$ with increasing $i$:
$$
(v^1_1,\, \cdots,\, v^1_{n_1}) \vee \cdots \vee (v^k_1,\, \cdots,\, v^k_{n_k})
\,:=\, (v^1_1,\, \cdots,\, v^1_{n_1}, \,\cdots,\, v^k_1,\, \cdots,\, v^k_{n_k})\, .
$$

A {\it partition} of a positive integer $n$ is an object of the form $[ d_1^{t_{d_1}},\, \cdots ,\, d_s^{t_{d_s}} ]$ where $t_{d_i},\, d_i \,\in\, \N$, $ 1 \,\leq\, i
\,\leq\, s$, such that
$\sum_{i=1}^{s} t_{d_i} d_i \,=\, n$, $t_{d_i} \,\geq\, 1$ and $d_{i+1}\,> \,d_i \,>\, 0$ for all $i$; see \cite[\S~3.1, p.~30]{CoM}. Let $\PC(n)$ denote the {\it set of all partitions of $n$}. For a partition $\d\, =\, [ d_1^{t_{d_1}},\, \cdots ,\, d_s^{t_{d_s}} ]$ of $n$,
define 
\begin{equation}\label{Nd-Ed-Od}
{\N}_{\d} \,:=\, \{ d_i \,\,\big\vert\,\, 1 \,\leq\, i \,\leq\, s \}\, ,\ \
{\E}_{\d} \,:=\, {\N}_{\d}\cap (2\N)\, , \ \ 
{\O}_{\d} \,:=\, {\N}_{\d}\setminus {\E}_{\d} .
\end{equation}
Further define 
\begin{equation}\label{Od1-Od3}
{\O}^1_{\d} \,:=\, \{ d \,\big\vert\, d \,\in\, \O_\d, \ d \,\equiv\, 1 \,\pmod{4} \}\ \
\text{ and }\ \ {\O}^3_{\d} \,:=\, \{ d \,\big\vert\, d \,\in\, \O_\d, \ d \,\equiv\, 3\, \pmod{4} \} .
\end{equation}
Following \cite[Theorem 9.3.3]{CoM}, a partition
$\d$ of $n$ will be called {\it even} if ${\N}_{\d} \,=\, {\E}_{\d}$. Let $\PC_{\rm even} (n)$
be the subset of $\PC(n)$ consisting of all even partitions of $n$. We call a partition 
$\d$ of $n$ to be {\it very even} if
\begin{itemize}
\item $\d$ is even, and

\item $t_\eta$ is even for all $\eta \in {\N}_{\d}$.
\end{itemize}
Let $\PC_{\rm v. even} (n)$ be the subset of $\PC(n)$ consisting of all very even partitions of
$n$. Now define 
$$\PC_1(n) \,:=\, \{ \d \,\in\, \PC(n) \,\,\big\vert\,\,t_\eta ~\text{ is even for all }~ 
\eta \,\in\, \E_\d \}$$
and
$$
\PC_{-1}(n) \,:=\, \{\d \,\in\, \PC(n) \,\,\big\vert\,\,
 t_\theta ~\text{ is even for all }~ \theta \,\in\, \O_\d \}\, .$$
Clearly, we have $\PC_{\rm v. even} (n) \,\subset\, \PC_1(n)$.

We next define certain sets using collections of matrices with entries comprising of signs $\pm 1$, which are
easily seen to be in bijection with sets of equivalence classes of various types of signed Young diagrams.
These sets will be used in parametrizing the nilpotent orbits in the classical Lie algebras.

For a partition $\d \,\in\, \PC (n)$ and $d \,\in\, \N_\d$, we define the subset 
$\Ab_d \,\subset\, {\rm M }_{t_d \times d} (\C)$ of matrices $(m^d_{ij})$ with entries in the set $\{ \pm 1 \}$ as follows : 
\begin{equation}\label{A-d}
\Ab_d \,:=\, \{ (m^d_{ij}) \,\in\, {\rm M }_{t_d \times d} (\C)\,\big\vert\, (m^d_{ij})\
\text{ satisfies } \eqref{yd-def1} \text{ and } \eqref{yd-def2}~ \text{ below} \} \,.
\end{equation}

\begin{enumerate}[label = {{\bf Yd}.\arabic*}]
\item~ \label{yd-def1} There is an integer $ 0 \,\leq\, p_d \,\leq\, t_d$ such that
$$
 m^d_{i1} \,:=\, \begin{cases}
+1 & \text{ if } \ 1 \,\leq\, i \,\leq\, p_d\\
-1 & \text{ if } \ p_d \,<\, i \,\leq\, t_d.
\end{cases}
$$
\item~ \label{yd-def2} 
\begin{align*}
m^d_{ij} &:=\, (-1)^{j+1}m^d_{i1} \qquad \text{if } \ 1\,<\,j \,\leq \,d , \ d\,\in\, \E_\d \cup \O^1_\d; \\
m^d_{ij} &:= \begin{cases}
(-1)^{j+1}m^d_{i1} & \text{ if }\ 1<j \,\leq\, d-1 \\
-m^d_{i1} & \text{ if }\ j \,=\, d 
\end{cases},\, \, d \,\in\, \O^3_\d\, .
\end{align*}
\end{enumerate}

For any $(m^d_{ij}) \,\in\, \Ab_d$ set $${\rm sgn}_+ (m^d_{ij})\,:=\, \# \{ (i,\,j) \,\,\big\vert\,\, 1 \,\leq\, i \,\leq\, t_d,\
1 \,\leq\, j \,\leq \, d ,\ m^d_{ij} \,=\, +1 \}$$ and
$${\rm sgn}_-(m^d_{ij})\,:=\, \# \{ (i,\,j) \,\,\big\vert\,\, 1\,\leq\, i \,\leq\, t_d,\ 1 \,\leq \,j \,\leq \,d ,\
m^d_{ij} \,=\, -1 \}\, .$$
Let $\SC_\d(n) \,:= \, \Ab_{d_1} \times \cdots \times \Ab_{d_s}$.
For a pair of non-negative integers $(p,q)$ with $p+q=n$
we now define the subset $\SC_\d(p, q) \subset \SC_\d(n) $ by 
\begin{equation}\label{S-d-pq}
\SC_\d(p, q)\, :=\, \big\{ (M_{d_1},\, \dotsc ,\, M_{d_s})\,\in\, \SC_\d(n) \,\big\vert\,\, \sum_{i=1}^s
{\rm sgn}_+ M_{d_i}\,= \, p, \, \sum_{i=1}^s { \rm sgn}_- M_{d_i}\,=\, q \big\}.
\end{equation}
We also define
\begin{equation}\label{yd-Y-pq}
\YC(p,\,q) \,:=\, \{ (\d, \,\sgn ) \,\big\vert\,\, \d \,\in\, \PC (n),\ \sgn \,\in \,\SC_\d(p,\,q) \} \, .
\end{equation}

It is easy to see that there is a natural bijection between the set $\YC(p,\,q)$ and
the equivalence classes of signed Young diagrams of size $p+q$ with signature $(p,\,q)$.
Hence, we will call $\YC(p,\,q)$ the {\it set of equivalence classes of signed Young diagrams of size $p+q$
with signature $(p,\,q)$}.

For any $\d \,\in\, \PC(n)$ and $d \,\in\, \N_\d$, define the subset $\Ab_{d, 1} $ of $\Ab_d$ by
$$\Ab_{d, 1} \,:=\, \{(m^d_{ij}) \,\in \,\Ab_d \,\big\vert\,\, m^d_{i\, 1} = +1~ \ \forall\ 1 \,\leq\, i \,\leq\, t_d \}\, .$$
Further define $\SC^{\rm even}_{\d}(p,q) \,\subset\, \SC_\d (p,q)$ and $\SC^{\rm odd}_{\d}(n) \,\subset \,\SC_\d (n)$ by
\begin{equation}\label{S-d-pq-even}
\SC^{\rm even}_{\d}(p,\,q) \,:=\, \{ (M_{d_1},\, \dotsc ,\,M_{d_s}) \,\in \,\SC_\d (p,\,q) \,\big\vert\,\, 
M_\eta \in \Ab_{\eta, 1} ~\ \forall~\ \eta \,\in\, \E_\d \}
\end{equation}
and
\begin{equation}\label{S-d-pq-odd}
\SC^{\rm odd}_{\d}(n) \,:=\, \{ (M_{d_1},\, \dotsc ,\,M_{d_s}) \,\in\, \SC_\d (n) \,\,\big\vert\,\,
M_\theta \,\in\, \Ab_{\theta, 1} ~\ \forall~\ \theta \,\in\, \O_\d \}\, .
\end{equation}

For a pair $(p,q)$ of non-negative integers we define the sets $\YC^{\rm even} (p,q)$ and $\YC^{\rm even}_1 (p,q)$ by
\begin{equation}\label{yd-even-Y-pq}
\YC^{\rm even} (p,\,q)\,:=\, \{ (\d, \,\sgn)\,\big\vert\,\, \d \,\in\, \PC (n),~ \sgn \,\in\, \SC^{\rm even}_{\d}(p,\,q) \},
\end{equation}
\begin{equation}\label{yd-1-Y-pq}
\YC^{\rm even}_1(p,\,q)\,:=\, \{ (\d,\, \sgn) \,\big\vert\,\, \d \,\in\, \PC_1 (n),~ \sgn
\,\in\, \SC^{\rm even}_{\d}(p,\,q) \}.
\end{equation}
Similarly, for a non-negative integer $n$, set 
\begin{equation}\label{yd-odd-Y-pq}
\YC^{\rm odd}(n) \,:= \,\{ (\d, \,\sgn) \,\,\big\vert\,\, \d \,\in\, \PC(n),~\ \sgn \,\in\, \SC^{\rm odd}_{\d}(n) \}, 
\end{equation}
\begin{equation}\label{yd-odd-1-Y-pq}
\YC^{\rm odd}_{-1}(2n)\,:=\,\{ (\d,\,\sgn)\,\,\big\vert\,\, \d \,\in\, \PC_{-1}(2n),~\ \sgn \,\in\,\SC^{\rm odd}_{\d}(2n)\}\, .
\end{equation}

Let $\d \,\in\, \PC(n)$. 
For $\theta \,\in\, \O_\d$ and $M_\theta \,:= \,(m^\theta_{rs})\,\in\,\Ab_\theta $,
define $$l^+_{\theta , i} (M_\theta ) \,:= \,\# \{ j \,\big\vert\, m^\theta_{ij} \,=\, +1 \}\ \ \text{ and }\ \
l^-_{\theta , i} (M_\theta ) \,:=\, \# \{j \,\big\vert\, m^\theta_{ij} \,= \,-1 \}$$
for all $1 \,\leq\, i \,\leq\, t_\theta$; set
$$
\SC'_\d(p,\, q) :=
\Bigg\{ (M_{d_1}, \dotsc ,M_{d_s}) \in \SC^{\rm even}_\d (p,q) \biggm| \! \!
\begin{array}{cc}
\quad l^+_{\theta , i} (M_\theta ) \text{ is even }\forall \, \theta \in \O_\d,\ 1 \leq i \leq t_\theta \\
 \text{or } l^-_{\theta , i} (M_\theta ) \text{ is even } \forall \theta \in \O_\d,\ 1 \leq i \leq t_\theta \!
\end{array} 
\Bigg\}.
$$

\section{Homotopy types of the nilpotent orbits}\label{sec-homotpy-type-nil-orbts}

In this section we will explicitly write down homotopy types of the nilpotent orbits in classical Lie algebras as compact homogeneous spaces.

Let $V$ be a right $ \D $-vector space of dimension $n$, where $\D$ is, as before, $\R$ or $ \C $ or $ \H$. Let $\{X,\,H,\,Y\} \subset 
\s\l (V )$ be a $\s\l_2(\R)$-triple. Consider the non-zero irreducible $\text{Span}_\R \{ X,\, H, \,Y\}$-submodules of $V$. Let $\{d_1,\, 
\cdots, \, d_s\}$, with $d_1 \,<\, \cdots \,<\, d_s$, be the integers that occur as $\R$-dimension of such $\text{Span}_\R \{ 
X,\,H,\,Y\}$-modules. Then using \eqref{dim-L-d-1} and \cite[Lemma A.1 (2)]{BCM}, it follows that
$$ \sum_{i=1}^s t_{d_{i}} d_i \,=\,{\dim}_\D V\,=\, n\, .$$ 
Thus 
\begin{equation}\label{partition-symbol}
\d \,:=\, \big[d_1^{t_{d_1} },\, \cdots ,\, d_s^{t_{d_s} }\big] \,\in\, \PC(n)\, .
\end{equation}
 Consider $\N_\d$, $\E_\d $ and $\O_\d $ as defined in \eqref{Nd-Ed-Od}. We have
\begin{align}\label{isotypicalcomp}
V &=\,\bigoplus_{d \in \N_\d} M(d-1) \ \ \ \text{ and } \ \ L(d-1) \,=\, V_{Y,0} \cap V_{H, 1-d} \ \ \text{ for }
\ d \,\geq \,1\, .
\end{align} 

Let $(v^d_1,\, \dotsc ,\, v^d_{t_d})$
be the ordered $\D$-basis of $ L(d-1) $ as in \cite[Proposition A.2]{BCM} for $d \,\in\, \N_\d$. Then it follows that
\begin{equation}\label{old-ordered-basis-part}
\BC^l (d) \,:=\, (X^l v^d_1,\, \dotsc ,\,X^l v^d_{t_d})
\end{equation}
is an ordered $\D$-basis of $X^l L(d-1)$ for $0\,\leq\, l \,\leq\, d-1$ with $d\,\in\, \N_\d$; see \cite[Proposition A.2(2)]{BCM}. Define
\begin{equation}\label{old-ordered-basis}
\BC(d) \,:=\, \BC^0 (d) \vee \cdots \vee \BC^{d-1} (d) \ \forall\ d \,\in\, \N_\d\, ,\ \text{ and }\ \BC\,:=\,
\BC(d_1) \vee \cdots \vee \BC(d_s)\, .
\end{equation}

Let 
\begin{equation}\label{map-D-SL}
\begin{split}
 \Db_{{\rm SL}(V)} \colon {\rm M}_{t_{d_1}}(\D) \times \cdots \times {\rm M}_{t_{d_s}}(\D) & \lto {\rm M}_{n}(\D)\\
 \big(A_{d_1}\, , \dotsc\, , A_{d_s} \big) \, &\longmapsto \, \bigoplus_{j=1}^s \big( A_{d_j} \big)_\blacktriangle ^{d_j} \, 
\end{split}
 \end{equation}
be the $\R$-algebra embedding. Let
\begin{equation}\label{algebra-isom}
 \Lambda_\BC \,: \, {\rm End} (V) \,\longrightarrow\, {\rm M}_n(\D)
\end{equation}
be the isomorphism of $\R$-algebras with respect to the ordered basis $\BC$.
Next define the character when $\D= \R$ or $\C$ 
\begin{equation}\label{map-bigchi}
\begin{split}
\bigchi_\d \,\colon\, \prod_{d \in \N_\d} {\rm GL} (L(d-1))\,& \, \longrightarrow \, \,\D^*\\
 \big(g_{t_{d_1}}\,,\, \dotsc \,,\, g_{t_{d_s}}\big)\, & \longmapsto \,
\prod_{i=1}^s \big(\det g_{t_{d_i}}\big)^{d_i}.
\end{split}
\end{equation}

Let $\langle \cdot,\, \cdot \rangle \,:\, V \times V \,\longrightarrow\, \D \ $ be a $\epsilon$-$\sigma$ Hermitian form. Assume that
$\{X,\,H,\,Y\}$ be a $\s\l_2(\R)$-triple in $\s\u (V, \,\langle \cdot,\, \cdot \rangle )$. Define the form
\begin{equation}\label{new-form}
(\cdot ,\,\cdot)_{d} \,:\, L(d-1) \times L(d-1 )\,\longrightarrow\, \D\, ,\ \ \
(v,\, u)_d \,:=\, \langle v \,,\, X^{d-1} u \rangle
\end{equation}
as in \cite[p.~139]{CoM}.

\subsection {Homotopy types of the nilpotent orbits in \texorpdfstring{${\s\l}_n(\D)$}{Lg}}\label{sec-sl-nc}

Let $n$ be a positive integer. In this subsection we write down the homotopy types of the nilpotent orbits in ${\s\l}_n(\D)$ for $ \D\,=\, 
\C,\,\R,\,\H $.

First recall a standard parametrization of $\NC({\rm SL}_n (\D))$, the set of all nilpotent orbits in $\s\l_n(\D)$. Let $X
\,\in\, {\NC}_ {\s\l_n(\D)}$ be a nilpotent element. First assume that $X \,\neq\, 0$, and
let $\{X,\, H,\, Y\} \,\subset\, \s\l_n(\D)$ be a $\s\l_2(\R)$-triple. 
Let $V := \D^n$ be the right $\D$-vector space of column vectors. 
Let $\{c_1,\, \cdots,\, c_l\}$, with $c_1 \,<\, \cdots \,<\, c_l$,
be the finitely many ordered integers that occur as $\R$-dimension of non-zero irreducible
$\text{Span}_\R \{ X,H,Y\}$-submodules of $V$.
Recall that $M(c-1)$ is defined to be the isotypical component of $V$ containing all irreducible $\text{Span}_\R \{ X,H ,Y \}$-submodules of $V$ with highest weight $(c-1)$, and as in \eqref{definition-L-d-1}, we set $L(c-1)\,:= \,V_{Y,0} \cap M(c-1)$. Recall that the space $L(c_r-1)$ is a $\D$-subspace for $1\leq r \leq l$.
Let $$t_{c_r} \,:=\, \dim_\D L(c_r-1)$$ for $ 1\,\leq\, r \,\leq\, l$. Then as $\sum_{r=1}^l t_{c_r} c_r \,=\,n$, we have
$[c_1^{t_{c_1}}, \,\cdots,\,c_l^{t_{c_l}}]\,\in\, \PC(n)$. Define $\Psi_{{\rm SL}_n (\D)} ( \OC_{X})\,:=\, [c_1^{t_{c_1}},\, \cdots ,\,c_l^{t_{c_l}}]$. It is easy to see that $\Psi_{{\rm SL}_n (\D)} (\OC_X) \,\neq\, [ 1^n ]$ as $X \,\neq\, 0$.
By declaring $\Psi_{{\rm SL}_n (\D)} (\OC_0) \,= \,[ 1^n ]$ we obtain a map 
\begin{align}\label{parametrizing-map-SL-D}
\Psi_{{\rm SL}_n (\D)} \,\colon\, \NC ({\rm SL}_n (\D)) \,\longrightarrow\, \PC (n).
\end{align}
Now we will consider three cases $\D\,=\,\C,\, \R,\,\H$ separately.

\subsubsection{The case of $\s\l_n(\C)$}

First we assume that $ \D\,=\,\C$. Recall that the map as in \eqref{parametrizing-map-SL-D}
parametrizes the nilpotent orbits in $\s\l_n(\C)$.

\begin{theorem}[{\cite[Theorem 5.1.1]{CoM}}]\label{sl-C-parametrization}
The map $\Psi_{{\rm SL}_n (\C)}$ in \eqref{parametrizing-map-SL-D} is a bijection.
\end{theorem}

\begin{theorem}\label{homototy-type-sl-nc}
Let $X \,\in\, \NC_{{\s\l_n}(\C)}$,\, $ X\,\neq\, 0$, and $\Psi_{{\rm SL}_n(\C)}(\OC_X) \,=\, \d$. Let $\{X,\,H,\,Y\}$ be a
$\s\l_2(\R)$-triple in $\s\l_n(\C)$. Let $K$ be a maximal compact subgroup of $\ZC_{{\rm SL}_n(\C)}(X,\,H,\,Y)$. 
Let the maps $\Lambda_\BC,\, \Db_{{\rm SL}(V)}$ and $\bigchi_\d$ be defined as in \eqref{algebra-isom}, \eqref{map-D-SL} and \eqref{map-bigchi}, respectively.
Then $\Lambda_\BC(K)$ is given by
$$
\Lambda_\BC(K)\, = \, \big\{\, \Db_{{\rm SL}(V)}(g) \, \bigm| \, g \in \prod_{i=1}^s {\rm U}(t_{d_i}) \,, \, \bigchi_\d(g) =1 \, \big\}.
$$ 
Moreover, the nilpotent orbit $ \OC_X$ in $ \s\l_n(\C) $ is homotopic to $ {\rm SU}(n)/ \Lambda_\BC (K)$.
\end{theorem}

\begin{proof}
Let $V$ be a right $\C$-vector space of column vectors such that $ \dim_\C V=n$, and $\BC$ be an ordered basis of $V$ as in 
\eqref{old-ordered-basis}. The proof follows from \cite[Lemma 4.4(1)]{BCM} and by writing the matrices of the elements of the maximal 
compact subgroup $ K $ with respect to the ordered basis $\BC$ as in \eqref{old-ordered-basis}.

The second part follows from Theorem \ref{mostow-corollary} and the well-known fact that any maximal compact subgroup of ${\rm SL}_n(\C)$ 
is isomorphic to ${\rm SU}(n)$.
\end{proof}

\subsubsection{The case of $\s\l_n(\R)$}

Next we will consider $ \D\,=\,\R $. The following known result says that the map $\Psi_{{\rm SL}_n (\R)} $ as in 
\eqref{parametrizing-map-SL-D} is ``almost'' a parametrization of the nilpotent orbits in $\s\l_n(\R)$.

\begin{theorem}[{\cite[Theorem 9.3.3]{CoM}}]\label{sl-R-parametrization}
For the map $\Psi_{{\rm SL}_n (\R)}$ in \eqref{parametrizing-map-SL-D},
$$\# \Psi_{{\rm SL}_n (\R)}^{-1} (\d) \,=\,
\begin{cases}
1 & \text{ for all } \ \d \,\in\, \PC(n) \setminus \PC_{\rm even} (n)\\
2 & \text{ for all } \ \d \,\in\, \PC_{\rm even} (n).
\end{cases}$$ 
\end{theorem}

\begin{theorem}\label{homototy-type-sl-nr}
Let $X\,\in\,\NC_{{\s\l_n}(\R)}$,\, $X\,\neq\, 0$ and $\Psi_{{\rm SL}_n(\R)}(\OC_X) \,=\, \d$. Let $\{X,\,H,\,Y\}$ be a $\s\l_2(\R)$-triple in 
$\s\l_n(\R)$. Let $K$ be a maximal compact subgroup of $\ZC_{{\rm SL}_n(\R)}(X,\,H,\,Y)$. Let the maps $\Lambda_\BC,\, \Db_{{\rm SL}(V)}$ and 
$\bigchi_\d$ be defined as in \eqref{algebra-isom}, \eqref{map-D-SL} and \eqref{map-bigchi}, respectively. Then $\Lambda_\BC(K)$ is given 
by
$$
\Lambda_\BC(K)\, = \, \big\{\, \Db_{{\rm SL}(V)}(g) \, \bigm| \, g \in \prod_{i=1}^s {\rm O}_{t_{{d_i}}} \,, \, \bigchi_\d(g) =1 \, \big\}.
$$ 
The nilpotent orbit $ \OC_X$ in $ \s\l_n(\R) $ is homotopic to $ {\rm SO}_n/ \Lambda_\BC (K)$.
\end{theorem}

\begin{proof}
Let $V$ be a right $\R$-vector space of column vectors such that $ \dim_\R V\,=\,n$, and $\BC$ be an ordered basis of $V$ as in 
\eqref{old-ordered-basis}. The proof follows from {\cite[Lemma 4.4(1)]{BCM}} and by writing the matrices of the elements of the maximal 
compact subgroup $K$ with respect to the ordered basis $ \BC$ as in \eqref{old-ordered-basis}.
	
The second part follows from Theorem \ref{mostow-corollary} and the well-known fact that any maximal compact subgroup of ${\rm SL}_n(\R)$ 
is isomorphic to ${\rm SO}_n$.
\end{proof}

\subsubsection{The case of $\s\l_n(\H)$}
Finally we will consider $\D\,=\,\H $.

\begin{theorem}[{\cite[Theorem 9.3.3]{CoM}}]\label{sl-H-parametrization}
The map $\Psi_{{\rm SL}_n (\H)}$ in \eqref{parametrizing-map-SL-D} is a bijection.
\end{theorem}

\begin{theorem}\label{homototy-type-sl-nh}
Let $X \,\in\, \NC_{{\s\l_n}(\H)}$,\, $ X\,\ne\, 0 $, and $\Psi_{{{\rm SL}_n(\H)}}(\OC_X) \,=\, \d$. Let $\{X,\,H,\,Y\}$ be a
$\s\l_2(\R)$-triple in $\s\l_n(\H)$. Let $K$ be a maximal compact subgroup of $\ZC_{{\rm SL}_n(\H)}(X,\,H,\,Y)$.
Let the maps $\Lambda_\BC$ and $ \Db_{{\rm SL}(V)}$ be defined as in \eqref{algebra-isom} and \eqref{map-D-SL}, respectively.
Then $\Lambda_\BC(K)$ is given by
$$
\Lambda_\BC(K)\, = \, \big\{\, \Db_{{\rm SL}(V)}(g) \, \bigm| \, g \in \prod_{i=1}^s {\rm Sp}(t_{d_i}) \, \big\}.
$$ 
Moreover, the nilpotent orbit $ \OC_X$ in $ \s\l_n(\H) $ is homotopic to $ {\rm Sp}(n)/ \Lambda_\BC (K)$.
\end{theorem}

\begin{proof}
Let $V$ be a right $\H$-vector space of column vectors such that $ \dim_\H V = n $, and $\BC$ be an ordered basis of $V$ as in 
\eqref{old-ordered-basis}. Now the proof follows from \cite[Lemma 4.4(1)]{BCM} and by writing the matrices of the elements of the maximal 
compact subgroup $ K $ with respect to the ordered basis $\BC$ as in \eqref{old-ordered-basis}.
	
The second part follows from Theorem \ref{mostow-corollary} and the well-known fact that any maximal compact subgroup of ${\rm SL}_n(\H)$ 
is isomorphic to ${\rm Sp}(n)$.
\end{proof}

\subsection{Homotopy types of the nilpotent orbits in \texorpdfstring{${\s\o}(n, \C)$}{Lg}}\label{sec-so-nc}

Let $n$ be a positive integer such that $n\ge 5$. The aim in this subsection is to write down the homotopy types of the nilpotent orbits 
in the simple Lie algebra ${\s\o}(n,\C)$ as compact homogeneous spaces. Throughout this subsection $\<>$ denotes the symmetric form on 
$\C^n$ defined by $\langle x,\, y \rangle \,:=\, x^t y$ for $x,\, y \,\in\, \C^n$.

We first recall a suitable parametrization of $\NC({\rm SO}(n,\C))$. 
Let $$\Psi_{{\rm SL}_n (\C)}\, :\,\NC ({\rm SL}_n(\C)) \,\longrightarrow\, \PC (n)$$ be the parametrization of
$\NC ({\rm SL}_n (\C))$; see Section \ref{sec-sl-nc} for details. 
Since ${\rm SO} (n,\C) \,\subset\, {\rm SL}_n (\C)$ (consequently as, the set of nilpotent elements
$\NC_{{\s\o}(n,\C)} \,\subset\, \NC_{\s\l_n(\C)}$)
we have the inclusion map
$\Theta_{{\rm SO}(n,\C)}\,:\, \NC ({\rm SO} (n,\C)) \,\longrightarrow\,\NC ( {\rm SL}_n (\C))$. Let
$$
\Psi_{{\rm SO}(n,\C)}\,:= \,\Psi_{{\rm SL}_n (\C)}\circ
\Theta_{{\rm SO}(n,\C)}\,\colon\, \NC ({\rm SO} (n,\C))
\,\longrightarrow\, \PC (n)
$$ be the composition.
Recall that $\Psi_{{\rm SO}(n,\C)} ( \NC ({\rm SO} (n,\C)))\,\subset \, \PC_1 (n)$, see \cite[Proposition A.6]{BCM}. Hence we have the following map:
$$\Psi_{{\rm SO}(n,\C)} \,\colon\, \NC ({\rm SO} (n,\C))
\,\longrightarrow\, \PC_1 (n) \,.
$$ 
The following known result says that the map $\Psi_{{\rm SO}(n,\C)} $ is ``almost'' a parametrization of the nilpotent orbits in $\s\o(n,\C)$. 

\begin{theorem}[{\cite[Theorem 5.1.2, Theorem 5.1.4]{CoM}}] \label{so-nC-parametrization}
For the above map 
$\Psi_{{\rm SO}(n,\C)}$,
$$ 
\# \Psi_{{\rm SO}(n,\C)}^{-1} (\d)\,=
\begin{cases}
2 & \text{ for all } \, \d \,\in \,\PC_{\rm v.even} (n) \\
1 & \text{ for all } \,\d \,\in\, \PC_1(n) \setminus\PC_{\rm v.even} (n) . 
\end{cases}$$ 
\end{theorem}

Let $0 \neq X \,\in\, \NC_{\s\o(n, \C)}$ and $\{X,\,H,\, Y \}$ a $\s\l_2(\R)$-triple in $\s\o(n,\C)$. 
Let $V$ be a right $\C$-vector space of column vectors such that $ \dim_\C V\,=\,n$.
Recall that the form $(\cdot ,\, \cdot )_d$ on $L(d-1)$ (see \eqref{new-form} for the definition) is symmetric when $d\,\in\, \O_\d$ and symplectic when $d \in \E_\d$. Let $(v^d_1,\, \dotsc \,, v^d_{t_d})$ be a $\C$-basis of $L(d-1)$ as in \cite[Proposition A.6]{BCM}. We may further assume that
$(v^\theta_1,\, \dotsc \,, v^\theta_{t_\theta})$ is an orthonormal basis of $L(\theta - 1)$ for the form $(\cdot, \,\cdot)_\theta$ for all $\theta \in \O_\d$, i.e., 
\begin{equation} \label{onb-theta-so-nc}
(v_j^\theta,\, v_j^\theta)_\theta \, = \, 1 \ \text{ for } 1 \leq j \leq t_\theta \ \ \text{ and } \ \ (v_j^\theta,\, v_i^\theta)_\theta \, = \, 0 \ \text{ for } j \,\neq\, i. 
 \end{equation}
Similarly, for all $\eta \,\in \,\E_{\d}$ we may assume that $(v^\eta_1,\, \dotsc \,, v^\eta_{t_\eta})$ is a symplectic basis of $L(\eta -1)$. This is equivalent to say that 
\begin{equation}\label{onb-eta-so-nc}
(v_j^\eta,\, v_{t_\eta/2 + j}^\eta)_\eta \, = \, 1 \ \ \text{ for } 1 \,\leq\, j \,\leq\, t_{\eta}/2 \ \ \text{ and } \ \
(v_j^\eta,\, v_{i}^\eta)_\eta \, = \, 0 \ \ \text{ for } i \,\neq\, j + t_{\eta}/2 \,.
\end{equation}

Next we will construct an orthonormal basis of $M(d-1)$. 
Following \cite[Lemma A.9]{BCM}, for $\theta \,\in\, \O_\d$ and $l$ even, $0\,\leq\, l \,\leq\, \theta-1$, we define
\begin{equation}\label{onb-theta-l-even-so-nc}
{ w}^{\theta}_{jl}\,:=\,
\begin{cases}
\big( X^l v^{\theta}_j + X^{\theta-1-l} v^{\theta}_j \big)/{\sqrt{2}} & \text{ if }\ 0\,\leq\,l\,<\, (\theta-1)/2\\
X^l v^{\theta}_j & \text{ if }\ l \,= \,(\theta-1)/2 \\
\sqrt{-1} \big(X^{\theta-1-l} v^{\theta}_j - X^l v^{\theta}_j \big) /{\sqrt{2}} & \text{ if }\ (\theta-1)/2 \,<\,
l \,\leq \,\theta -1.
\end{cases}
\end{equation}
Similarly, for $l$ odd, $0\,\leq\, l \,\leq\, \theta-1$, let
\begin{equation}\label{onb-theta-l-odd-so-nc}
{ w}^{\theta}_{jl}\,:=\,
\begin{cases}
\sqrt{-1}\big( X^l v^{\theta}_j + X^{\theta-1-l} v^{\theta}_j \big) / {\sqrt{2}} & \text{ if }\ 0\,\leq\,l\,<\, (\theta-1)/2\\
\sqrt{-1} X^l v^{\theta}_j & \text{ if }\ l \,= \,(\theta-1)/2\\
\big(X^{\theta-1-l} v^{\theta}_j - X^l v^{\theta}_j \big) / {\sqrt{2}}& \text{ if }\ (\theta-1)/2 \,<\, l \,\leq \,\theta -1.
\end{cases}
\end{equation}

Using \cite[Lemma A.9(2)]{BCM} and the relation in \eqref{onb-theta-so-nc} we conclude that for all
$\theta \,\in\, \O_\d$, 
$$\{w^\theta_{jl} \,\,\big\vert\,\, 1\,\leq\, j \,\leq\, t_\theta,\ \, 0\,\leq\, l
\,\leq\, \theta -1 \}$$
is an orthonormal basis of $M(\theta -1)$ with respect to $\<>$. For each $0 \,\leq\, l \,\leq\, \theta -1$, set
\begin{equation}\label{V-l-theta-so-nc}
V^l(\theta) \,:=\, \text{Span}_\C\{w^\theta_{1l},\, \cdots \,, w^\theta_{t_\theta l} \} .
\end{equation}
The orthonormal ordered basis $(w^\theta_{1 l},\, \cdots ,\, w^\theta_{t_\theta l} )$ of $V^l(\theta ) $ with 
respect to $\<>$ is denoted by $\CC^l(\theta)$. Recall that $\BC^l(d)\,=\, (X^lv^d_1,\, \cdots ,\, X^lv^d_{t_d})$ 
is the ordered basis of $X^l L(d-1)$ for $0\,\leq\, l \,\leq\, d -1,\, d \,\in\, \N_\d$ as in 
\eqref{old-ordered-basis-part}.

\begin{lemma}\label{Z-so-nC-XHY} Let $ X $ be a nilpotent element in $ \s\o(n,\C) $ and $ \{X,H,Y\} $ be a $ \s\l_2 (\R)$-triple in $\s\o(n,\C) $ containing $X$.
The following holds:
$$
\ZC_{ {\rm SO} (n,\C) }(X,\,H,\,Y)\! =\! \left\{ g \in { {\rm SO} (n,\C)} \middle\vert \! 
\begin{array}{ccc}
g (V^l (\theta)) \subset \ V^l (\theta) \ \text{and } \vspace{.14cm}\\
\big[g|_{ V^l(\theta)}\big]_{{\CC}^l(\theta)} = \big[g |_{ V^0 (\theta)}\big]_{{\CC}^0 (\theta)} \ { \forall } \, \theta \in \O_\d, 0 \leq l < \theta ~; \vspace{.14cm}\\
g(X^l L(\eta-1)) \subset X^l L(\eta-1) \text{ and } \vspace{.14cm}\\
\! \big[g |_{ X^l L(\eta-1)} \big]_{{\BC}^l (\eta)}\! \!=\! \big[g |_{L(\eta-1)}\big]_{{\BC}^0 (\eta)}\ {\forall } \, \eta \in \E_\d,
0 \,\leq\, l\,<\, \eta \!
\end{array}
\!\! \right\}\!.
$$
\end{lemma}

\begin{proof}
The proof of this lemma follows from the observation that if $\theta \,\in\, \O_\d$ is fixed, then for
$0 \,\leq\, l \,\leq \,\theta -1$ and $g \,\in\,\ZC_{ {\rm SO} (n,\C) }(X,\,H,\,Y)$ the following holds:
$$
g(X^l L(\theta-1))\,\subset \, X^l L(\theta-1) \text{ if and only if } g (V^l (\theta))\,\subset\, \ V^l (\theta) \,,
$$
and moreover, 
$$
\big[g |_{ X^l L(\theta-1)} \big]_{{\BC}^l (\theta)}\! \!=\! \big[g |_{L(\theta-1)}\big]_{{\BC}^0 (\theta)} \text{ if and only if } \big[g|_{ V^l(\theta)}\big]_{{\CC}^l(\theta)} = \big[g |_{ V^0 (\theta)}\big]_{{\CC}^0 (\theta)} \,.
$$
In fact, for any such $g$ as above, $[g |_{L(\theta-1)}\big]_{{\BC}^0 (\theta)}\, = \, [g |_{ V^0 (\theta)}]_{{\CC}^0 (\theta)} $.
\end{proof}

\begin{remark}\label{structure-of-centralizer-sonc}
{\rm 
We follow the notation as in Lemma \ref{Z-so-nC-XHY}.
Let $g \,\in\, \ZC_{ {\rm SO} (n, \C)}(X,\, H,\, Y)$. Let $\theta \,\in\, \O_\d$ and $\eta \,\in\, \E_\d$.
Then it follows from Lemma \ref{Z-so-nC-XHY} that $ g$ keeps the subspaces $V^0(\theta)$ and $L(\eta -1)$ invariant. 
Since the restriction of $\<> $ is a symmetric form on $ V^0(\theta) $ we have $g|_{ V^0(\theta)} \,\in\,
{\rm O}(V^0(\theta), \<>)$. 
Further recall that the form $ (\cdot,\cdot)_\eta $, as defined in \eqref{new-form}, is symplectic on $L(\eta-1)$, and $(gx,\,gy)_\eta
\,=\, (x,\,y)_\eta$ for all $x,\,y\,\in\, L(\eta-1)$, see \cite[Remark A.7]{BCM}. 
Thus $g |_{L(\eta-1)} \,\in\, {\rm Sp}(L(\eta-1),\, (\cdot,\, \cdot)_\eta)$.
}
\end{remark}

For $\eta\,\in \,\E_\d$, $0\,\leq\, l \,\leq\, \eta/2 -1$, set
\begin{equation}\label{W-eta-so-nc}
W^l(\eta)\,:=\, X^l L (\eta-1) + X^{\eta-1-l} L(\eta-1)\, .
\end{equation}
Next we will construct an orthonormal basis of $W^l$.
For $l$ even, $0\,\leq\, l \,\leq\, \eta/2-1$, let
 \begin{align}\label{definition-w-basis-sonC-even1}
 {w}^{\eta}_{jl} \,:=\,
\begin{cases}
\big( X^l v^{\eta}_j + X^{\eta-1-l} v^{\eta}_{t_\eta/2 + j} \big) / {\sqrt{2}} & \text{ if } \ 0\,<\,j \,\leq \, t_\eta/2\\
\big( X^l v^{\eta}_j - X^{\eta-1-l} v^{\eta}_{j- t_\eta/2 } \big) / {\sqrt{2}} & \text{ if } \ t_\eta/2\, < \,j \,\leq \, t_\eta\\
\sqrt{-1}\big( X^l v^{\eta}_{j- t_\eta} - X^{\eta-1-l} v^{\eta}_{j - t_\eta/2} \big) / {\sqrt{2}} & \text{ if } \ t_\eta \, < \,j \,\leq \, 3t_\eta/2\\
\sqrt{-1}\big( X^l v^{\eta}_{j- t_\eta}+ X^{\eta-1-l} v^{\eta}_{j - 3t_{\eta}/2} \big) / {\sqrt{2}} & \text{ if } \ 3t_\eta/2 \, < \,j \,\leq \, 2t_\eta \,.
\end{cases}
\end{align}
For $l$ odd, $0\leq l \leq \eta/2-1$, let
\begin{align}\label{definition-w-basis-sonC-even2}
{w}^{\eta}_{jl} \,:=\,
\begin{cases}
\big( X^l v^{\eta}_j - X^{\eta-1-l} v^{\eta}_{j + t_\eta/2 } \big) / {\sqrt{2}}& \text{ if } \ 0\,<\,j \,\leq \, t_\eta/2\\
\big( X^l v^{\eta}_j + X^{\eta-1-l} v^{\eta}_{j - t_\eta/2} \big) / {\sqrt{2}} & \text{ if } \ t_\eta/2\, < \,j \,\leq \, t_\eta\\
\sqrt{-1}\big( X^l v^{\eta}_{j- t_\eta}+X^{\eta-1-l} v^{\eta}_{j- t_\eta/2} \big) / {\sqrt{2}} & \text{ if } \ t_\eta \, < \,j \,\leq \, 3t_\eta/2 \\
\sqrt{-1}\big( X^l v^{\eta}_{j- t_\eta} - X^{\eta-1-l} v^{\eta}_{j - 3t_{\eta}/2} \big) / {\sqrt{2}} & \text{ if } \ 3t_\eta/2 \, < \,j \,\leq \, 2t_\eta \,.
\end{cases}
\end{align}
Using \eqref{onb-eta-so-nc} it follows that for all $\eta\,\in\, \E_\d$, $$\{w^\eta_{jl} \, \mid \, 1\leq j \leq 2t_\eta, \, 0\leq l \leq 
\eta/2 -1 \}$$ is an orthonormal basis of $M(\eta -1)$ with respect to $\<>$. The orthonormal ordered basis $(w^\eta_{1 l} \, ,\, \cdots \, 
,\, w^\eta_{2t_\eta l} )$ of $W^l(\eta)$ with respect to $\<>$ is denoted by $\DC^l(\eta)$.

The next two lemmas are standard fact where we recall, without proofs, explicit descriptions of a maximal compact 
subgroup in an orthogonal group and a maximal compact subgroup in a symplectic group. Let $V'$ be a $\C$-vector 
space, $\<>'$ be a non-degenerate symmetric form on $V'$ and $\BC'$ be an orthonormal basis of $V'$. We set
$$
K_{\BC'} \, := \, \{g \,\in\, {\rm SU}(V' ,\, \<>') \,\,\big\vert\,\, [g]_{\BC'} \,=\, \overline{[g]}_{\BC'} \}\,.
$$ 
\begin{lemma}\label{max-cpt-orthogonal-gps}
 Let $V',\, \<>'\, ,\, \BC'$ be as above. Then $K_{\BC'}$ is a maximal compact subgroup in $ {\rm SO}(V,\,\<>') $.
\end{lemma}
 
Let $\widetilde V$ be a $\C$-vector space, $\widetilde {\<>}$ be a non-degenerate symplectic form on $\widetilde V$ and $\widetilde \BC$ be a symplectic basis of $\widetilde V$. Let $J_{\widetilde \BC}$ be the complex structure on $\widetilde V$ associated to $\widetilde \BC$. Let $2m := \dim_\C \widetilde V$. We set 
$$
K_{\widetilde \BC} \, := \, \{g \in {\rm Sp}(\widetilde V, \widetilde {\<>}) \,\,\big\vert\,\, \overline{g} \, J_{\widetilde \BC}
\,= \, J_{\widetilde \BC} \, g \}\,.
$$ 
\begin{lemma}\label{max-cpt-symplectic-gps}
Let $\widetilde V,\, \widetilde{\<>} ,\,\widetilde \BC$ be as above. Then
\begin{enumerate}
\item $K_{\widetilde \BC}$ is a maximal compact subgroup in ${\rm Sp}(\widetilde V,\, \widetilde {\<>})$.

\item $K_{\widetilde \BC} \,=\, \big\{g \in {\rm Sp}(\widetilde V,\,\widetilde{ \<>}) \, \big| \, [g]_{\widetilde \BC} = \begin{pmatrix}
A & - \overline{B}\\
B & \overline {A}
\end{pmatrix}
\text{ where } A + \jb B \,\in\, {\rm Sp}(m) , \,\,A,\, B \,\in\, {\rm M}_m(\C) \big\}$.
 \end{enumerate}
\end{lemma}

We next give the description of a suitable maximal compact subgroup of the group $\ZC_{ {\rm SO} (n, \, \C) } (X,\, H,\, Y)$ in terms of 
the subspaces $V^l(\theta), \, W^l(\eta)$ defined as in \eqref{V-l-theta-so-nc}, \eqref{W-eta-so-nc}, respectively.

\begin{lemma}\label{max-cpt-reductive-part-so-nC}
Let $K$ be the subgroup of $\ZC_{ {\rm SO} (n,\, \C) } (X,\, H,\, Y)$ consisting of elements $g$ in 
$\ZC_{ {\rm SO} (n,\, \C) } (X,\, H,\, Y)$ satisfying the following conditions:
\begin{enumerate}
\item For all $\theta \,\in\, \O_\d $ and $0 \,\leq\, l \,\leq\, \theta -1$ the inclusion $g ( V^l (\theta))
\,\subset\, V^l (\theta)$ holds.
\item For all $\theta \,\in\, \O_\d$, there exist $A(\theta) \,\in\, {\rm O}_{t_\theta}$ such that 
$$
\big[g|_{V^l(\theta)}\big]_{\CC^l(\theta)}\, = \, A(\theta) \quad \text{for } 0 \,\leq\, l \,\leq\, \theta -1 \,. 
$$ 
\item For all $\eta\,\in\, \E_\d$ and $ 0 \,\leq\, l \,\leq\, \eta/2 -1$ the inclusion \, $ g(W^l (\eta)) \,\subset\, W^l (\eta)$ holds.
\item For all $\eta \,\in \,\E_\d$, $0 \,\leq\, l \,\leq\, \eta/2-1$ there exist $A_1(\eta), \, A_2(\eta),\, B_1(\eta),\, B_2(\eta)
\,\in\, {\rm M}_{t_\eta /2}(\R)$ with $(A_1(\eta) + \ib A_2(\eta)) + \jb (B_1(\eta) + \ib B_2(\eta) ) \,\in\, {\rm Sp}(t_\eta/2) $ such that
$$
\big[ g|_{W^l (\eta)} \big]_{\DC^l(\eta)} \, = \,
\begin{pmatrix}
A_1(\eta) & - B_1(\eta) &- A_2(\eta) & -B_2(\eta)\\
B_1(\eta) & A_1(\eta) & -B_2(\eta)& A_2(\eta)\\
A_2(\eta) & B_2(\eta) & A_1(\eta) & - B_1(\eta)\\
B_2(\eta) & - A_2(\eta) & B_1(\eta)& A_1(\eta)
\end{pmatrix} \,.
$$
\end{enumerate}
Then $K$ is a maximal compact subgroup of $\ZC_{ {\rm SO} (n,\, \C)}(X,\, H,\, Y)$.
\end{lemma}

\begin{proof}
Let $K'$ be the subgroup consisting of all $g \,\in\, \ZC_{{\rm SO} (n, \, \C)}(X,\,H,\,Y)$ satisfying the conditions 
\eqref{max-cpt-odd-subspace-so-nc} and \eqref{max-cpt-even-subspace-so-nc} below:
\begin{align}
\big[g|_{ V^0(\theta)}\big]_{{\CC}^0(\theta)}\, &=\, \overline{\big[g|_{ V^0(\theta)}\big]}_{{\CC}^0(\theta)}\,, \quad\, \text{for all }\theta \in \O_\d\,; \label{max-cpt-odd-subspace-so-nc}\\
\overline{ g }|_{ L(\eta-1)} { J}_{\BC^0(\eta)} \, &= \, { J}_{\BC^0(\eta)} g |_{ L(\eta-1)}\,, \quad \text{for all }\eta \in \E_\d \,
\label{max-cpt-even-subspace-so-nc}.
\end{align}
In view of Lemma \ref{Z-so-nC-XHY}, Lemma \ref{max-cpt-orthogonal-gps} and Lemma \ref{max-cpt-symplectic-gps} it is clear that $K'$ is a 
maximal compact subgroup of $\ZC_{{\rm SO} (n, \, \C)}(X,\,H,\,Y)$. Thus to prove the lemma it is enough to show that 
$K \,=\, K'$. Let $g \in \ZC_{{\rm SO} (n, \, \C)}(X,\,H,\,Y)$. In view of Lemma \ref{max-cpt-orthogonal-gps} and Remark 
\ref{structure-of-centralizer-sonc} it is clear that $g$ satisfies (1), (2) in the statement of the lemma if and 
only if $g$ satisfies \eqref{max-cpt-odd-subspace-so-nc}. Clearly, $g$ satisfies (3) of the Lemma
\ref{max-cpt-reductive-part-so-nC}. Let $[g|_{L(\eta -1)}]_{\BC^0(\eta)}\,:= \,\begin{pmatrix}
A(\eta) & C(\eta)\\
B(\eta) & D(\eta)
\end{pmatrix}$.
Now suppose that $g$ satisfies \eqref{max-cpt-even-subspace-so-nc}. Then it follows that
$C(\eta)\,=\, - \overline{B(\eta)}$, $D(\eta)\,=\,\overline{A(\eta)}$ and $$A(\eta) + \jb B(\eta)
\,\in\, {\rm Sp}(t_\eta/2)\, .$$ Set $$A(\eta)\,:=\, A_1 (\eta) + \sqrt{-1} A_2(\eta)\ \ \text{ and }\ \ B(\eta)\,:=\,
B_1 (\eta) + \sqrt{-1} B_2(\eta)$$ for $A_1 (\eta),\, A_2 (\eta) ,\, B_1 (\eta),\, B_2 (\eta)
\,\in\, {\rm M}_{t_\eta /2}(\R)$. Then $g$ satisfies (4) of the Lemma \ref{max-cpt-reductive-part-so-nC}. Next
we assume that $g$ satisfies (3), (4) of Lemma \ref{max-cpt-reductive-part-so-nC}. We observe that 
$$
[g|_{L(\eta -1)}]_{\BC^0(\eta)}\,=\, \begin{pmatrix}
A_1 (\eta) + \sqrt{-1} A_2(\eta) & -B_1 (\eta) + \sqrt{-1} B_2(\eta) \\
B_1 (\eta) + \sqrt{-1} B_2(\eta) & A_1 (\eta) - \sqrt{-1} A_2(\eta)
\end{pmatrix}
$$
which proves that \eqref{max-cpt-even-subspace-so-nc} holds. This completes the proof. 
\end{proof}
 
Now we introduce some notation which will be required to state Theorem \ref{max-cpt-so-nc-wrt-onb}. For $\eta\,\in\, 
\E_\d$, set
$$
\DC(\eta) \,:=\, \DC^0 (\eta) \vee \cdots \vee \DC^{\eta/2-1} (\eta)\, ,
$$
and for $\theta\,\in\, \O_\d$, set 
$$
\CC(\theta)\,:=\, \CC^0 (\theta) \vee \cdots \vee \CC^{\theta-1} (\theta).
$$
Let $\alpha\,:=\, \# \E_\d $ and $\beta \,:=\, \# \O_\d$.
We enumerate
$\E_\d \,=\,\{ \eta_i\,\big\vert\,\, 1 \,\leq\, i \,\leq\, \alpha \}$ such that $\eta_i \,<\, \eta_{i+1}$, 
and similarly
$\O_\d \,=\,\{ \theta_j\,\big\vert\,\, 1 \,\leq\, j \,\leq\, \beta \}$ such that $\theta_j \,<\, \theta_{j+1}$.
Now define
$$
\EC \,:=\, \DC(\eta_1) \vee \cdots \vee \DC(\eta_{\alpha}) ; \ 
\text{ and } \OC \,:=\, \CC (\theta_1) \vee \cdots \vee \CC (\theta_{\beta}).
$$
Finally, define
\begin{equation}\label{orthogonal-basis-so-nc-final}
\HC \,:= \, \EC \vee \OC. 
\end{equation}

Define the $\R$-algebra embedding 
\begin{align}\label{embedding-C2R}
\wp_{m, \C} \,: \,{\rm M}_m(\C) \,\longrightarrow\, {\rm M}_{2m}(\R)\, ,\ \
S+ \sqrt{-1} T \,\longmapsto\, \begin{pmatrix}
	S & -T\\
	T & S
\end{pmatrix}	
\end{align}
where $S,\,T \,\in\, {\rm M}_m(\R)$.
Similarly, for $P,\,Q \,\in\, {\rm M}_m(\C)$ define the $\R$-algebra embedding 
\begin{align}\label{embedding-H2C}
\wp_{m, \H} \,: \,{\rm M}_m(\H) \,\longrightarrow\, {\rm M}_{2m}(\C)\, ,\, ~ \,
P + \jb Q \,\longmapsto\, \begin{pmatrix}
P & - \overline{Q}\\
Q & \overline{P}
\end{pmatrix} \,.
\end{align}
The $\R$-algebra $\prod_{i=1}^{\alpha} {\rm M}_{t_{\eta_i}/2}(\H) \times\prod_{k=1}^\beta {\rm M}_{t_{\theta_k}} (\R)$ is embedded into $ {\rm M}_n(\R)$ in the following way:
\begin{align}\label{map-D-sonC}
\Db_{{\rm SO}(n,\C)} \colon \prod_{i=1}^{\alpha} & {\rm M}_{t_{\eta_i}/2}(\H) \, \, \times \,\, \prod_{k=1}^\beta {\rm M}_{t_{\theta_k}} (\R) \,\, \longrightarrow \, \,{\rm M}_n(\R)\\
\big(A_{\eta_1},\, \dotsc, \, A_{\eta_\alpha}\, ; ~ C_{\theta_1} & \,,\, \dotsc \, , \, C_{\theta_\beta} \big) 
\longmapsto \, ~ \bigoplus_{i=1}^\alpha \Big(\wp_{t_{\eta_i} , \C} \big(\wp_{t_{\eta_i}/2,\H}(A_{\eta_i})\big)\Big)_\blacktriangle^{\eta_i/2} 
~\oplus ~ \bigoplus_{k=1}^\beta \big( {C_{\theta_k}}\big)_\blacktriangle^{\theta_k}\,. \nonumber
\end{align}

Note that the basis $ \HC $ as in $\eqref{orthogonal-basis-so-nc-final}$ is an orthonormal basis of $V$ with respect to $ \<> $. Let $\Lambda_\HC \,\colon\, {\rm End}_\C \C^{n} \,\longrightarrow\, {\rm M}_n (\C)$ be the isomorphism of $\C$-algebras induced by the ordered basis $\HC$. 

\begin{theorem}\label{max-cpt-so-nc-wrt-onb} 
Let $X \,\in\, \NC _{\s\o (n,\C)}$, $ X\ne 0 $, and $\Psi_{{\rm SO} (n,\C)}(\OC_X) \,=\, \d $.
Let $\alpha \,:=\, \# \E_\d $ and $\beta \,:=\, \# \O_\d$.
Let $\{X,\,H,\,Y\}$ be a $\s\l_2(\R)$-triple in $\s\o (n,\C)$.
Let $K$ be the maximal compact subgroup of $\ZC_{{\rm SO} (n,\C)} ( X,\, H,\, Y)$ as in Lemma
\ref{max-cpt-reductive-part-so-nC}. Let the map $\Db_{{\rm SO}(n,\C)} $ be defined as in \eqref{map-D-sonC}.
Then $\Lambda_\HC (K) \,\subset \,{\rm SO}_n$ is given by 
\begin{align*} 
\Lambda_\HC(K)\,=\,
 \Big\{ \Db_{{\rm SO}(n,\C)} (g) \, \bigm| \, g \in \prod_{i=1}^\alpha {\rm Sp}(t_{\eta_i}/2) \, \times\, {S( \prod_{j=1}^\beta {\rm O}(t_{\theta_j})) } \Big\}.
\end{align*}
 The nilpotent orbit $ \OC_X $ in $ \s\o(n,\C) $ is homotopic to $ {\rm SO}_n/ \Lambda_\HC(K) $.
\end{theorem}
 
\begin{proof}
This follows by writing the matrices of the elements of the maximal compact subgroup $ K $ in Lemma 
\ref{max-cpt-reductive-part-so-nC} with respect to the ordered basis $\HC$ as in 
\eqref{orthogonal-basis-so-nc-final}.
 	
The second part follows from Theorem \ref{mostow-corollary} and the fact that any maximal compact subgroup of 
${\rm SO}(n,\C)$ is isomorphic to ${\rm SO}_n$.
\end{proof}
 
\subsection{Homotopy types of the nilpotent orbits in \texorpdfstring{${\s\o}(p,q)$}{Lg}}\label{sec-so-pq}

Let $n$ be a positive integer and $(p, q)$ be a pair of non-negative integers such that $p + q \,=\,n\geq 5$. In 
this subsection we write down the homotopy types of the nilpotent orbits in ${\s\o}(p,q)$ under the adjoint action 
of $ {\rm SO}(p,q) ^\circ$. Throughout this subsection $\<>$ denotes the symmetric form on $\R^n$ defined by 
$\langle x,\, y \rangle \,:=\, x^t{\rm I}_{p,q} y$ where $x,\, y \,\in\, \R^n$ and ${\rm I}_{p,q}$ is as in 
\eqref{defn-I-pq-J-n}.

We first need to describe a suitable parametrization of $\NC ({\rm SO} (p,q)^\circ)$, the set of all nilpotent orbits in ${\s\o}(p,q)$ under the adjoint action of ${\rm SO} (p,q)^\circ$, see \cite[\S 4.5]{BCM}.
Let $$\Psi_{{\rm SL}_n (\R)} \,:\,\NC ({\rm SL}_n(\R)) \,\longrightarrow\,\PC (n)$$ be the parametrization of
$\NC ({\rm SL}_n (\R))$ as in Theorem \ref{sl-R-parametrization}.
Since ${\rm SO} (p,q) \,\subset\, {\rm SL}_n (\R)$ (consequently as, the set of nilpotent elements
$\NC_{{\s\o}(p,q)} \,\subset\, \NC_{\s\l_n(\R)}$)
we have the inclusion map
$$\Theta_{{\rm SO}(p,q)^\circ}\,:\, \NC ({\rm SO} (p,q)^\circ) \,\longrightarrow\,\NC ( {\rm SL}_n (\R))\, .$$ Let
$$
\Psi'_{{\rm SO}(p,q)^\circ}\,:= \,\Psi_{{\rm SL}_n (\R)}\circ
\Theta_{{\rm SO}(p,q)^\circ}\,\colon\, \NC ({\rm SO} (p,q)^\circ)
\,\longrightarrow\, \PC (n)
$$ be the composition.
Recall that $\Psi'_{{\rm SO}(p,q)^\circ} ( \NC ({\rm SO} (p,q)^\circ))\,\subset \, \PC_1 (n)$. 
Let $X \,\in\, {\s\o}(p,q)$ be a non-zero nilpotent element and $\OC_X$
the corresponding nilpotent orbit in $\s\o(p,q)$ under the adjoint action of 
${\rm SO} (p,q)^\circ$. Let $\{X,\, H,\, Y\} \,\subset\, {\s\o}(p,q)$ be a $\s\l_2(\R)$-triple.
Let $V:= \R^n$ be the right $\R$-vector space of column vectors.
Let $\{d_1,\, \cdots,\, d_s\}$, with $d_1 \,<\, \cdots \,<\, d_s$, be ordered finite set of natural numbers
that occur as dimension of non-zero irreducible 
$\text{Span}_\R \{ X,\,H,\,Y\}$-submodules of $V$. Recall that $M(d-1)$ is defined to be the isotypical component of $V$ containing all irreducible Span$_\R \{X,\,H,\,Y\}$-submodules of $V$ with highest weight $d-1$ and as in \eqref{definition-L-d-1} we set
$$L(d-1)\,:= \,V_{Y,0} \cap M(d-1).$$ Let $t_{d_r} \,:=\, \dim_\R L(d_r-1)$,
$1 \,\leq \,r \,\leq\, s$. Then 
$$\d\,:=\, [d_1^{t_{d_1}}, \,\cdots ,\, d_s^{t_{d_s}}]\,\in\, \PC_1(n),$$ and moreover,
$\Psi'_{{\rm SO}(p,q)^\circ} (\OC_X) \,=\, {\d}$.

We next assign $\sgn_{\OC_X} \in \SC^{\rm even}_{\d}(p,q)$ to each $\OC_X \,\in\,
\NC({\rm SO}(p,q)^\circ)$; see \eqref{S-d-pq-even} for the definition of
$\SC^{\rm even}_{\d}(p,q)$.
For each $d\,\in\, \N_\d$ (see \eqref{Nd-Ed-Od} for the definition of $\N_\d$) we will define a $t_d \times d$ matrix $(m^d_{ij})$ in
$ \Ab_{d}$ that depends only on the orbit $\OC_X$; see \eqref{A-d} for the definition of $\Ab_d$.
For this, recall that the form $$(\cdot,\,\cdot)_{d} \,\colon\, L(d-1) \times L(d-1) \,
\longrightarrow\, \R\, ,$$ defined in \eqref{new-form}, is symmetric or symplectic
according as $d$ is odd or even. 
Denoting the signature of $(\cdot,\,\cdot)_{d}$ by $(p_{d},\, q_{d})$ when $d\,\in\, \O_\d$, we now define
\begin{align*}
m^\eta_{i1} &:= +1 \qquad \text{if } \ 1 \,\leq\, i \,\leq\, t_{\eta}, \quad \eta \,\in\, \E_\d\,; \\
m^\theta_{i1} &:= \begin{cases}
+1 & \text{ if } \ 1 \,\leq\, i \,\leq\, p_{\theta} \\
-1 & \text{ if } \ p_\theta \,<\, i \leq t_\theta
\end{cases}\, , \theta \,\in\, \O_\d\,;
\end{align*}
and for $j \,>\,1$, define $(m^d_{ij})$ as follows:
\begin{align}\label{def-sign-alternate}
m^d_{ij} \, &:=\, (-1)^{j+1}m^d_{i1} \qquad \text{if $ 1<j \leq d , \ d\in \E_\d \cup \O^1_\d$};\\
m^\theta_{ij}\, &:= \begin{cases}
(-1)^{j+1}m^\theta_{i1} & \text{ if } \ 1\,<\,j \,\leq\, \theta-1\\
-m^\theta_{i1} & \text{ if } j \,=\, \theta
\end{cases} , \, \theta\,\in\, \O^3_\d.\label{def-sign-alternate-1}
\end{align}
Then the matrices $(m^d_{ij})$ clearly verify \eqref{yd-def2}. Set $$\sgn_{\OC_X} \,:=\,
((m^{d_1}_{ij}),\, \cdots,\, (m^{d_s}_{ij})).$$ It now follows from the last paragraph of \cite[Remark A.13]{BCM} and the above definition of $m^\eta_{i1}$ for $\eta \,\in\, \E_\d$ that $\sgn_{\OC_X}\,\in\, \SC^{\rm even}_{\d}(p,q)$.
Thus we have the map 
$$
\Psi_{{\rm SO} (p,q)^\circ} \,\colon\, \NC({\rm SO}(p,q)^\circ)\,
\longrightarrow\, \YC^{\rm even}_1(p,q)\, , \ \
\OC_X \,\longmapsto\, \big(\Psi'_{{\rm SO} (p,q)^\circ} (\OC_X),\, \sgn_{\OC_X}\big) \,;
$$
where $\YC^{\rm even}_1(p,q)$ is as in \eqref{yd-1-Y-pq}.
The map $\Psi_{{\rm SO} (p,q)^\circ}$ is surjective.
The following theorem is standard, see \cite[Theorem 9.3.4]{CoM}, \cite[Theorem 4.16]{BCM}.

\begin{theorem}\label{so-pq-parametrization}
For the above map $\Psi_{{\rm SO}(p,q)^\circ}$,
$$ \# \Psi_{{\rm SO}(p,q)^\circ}^{-1} (\d, \sgn) \,=
\begin{cases}
4 & \text{ for all } \, \d \,\in\,\PC_{\rm v.even} (n) \\
2 & \text{ for all } \, \d \,\in\, \PC_1(n) \setminus \PC_{\rm v.even} (n) , \, \sgn \in \SC'_{\d}(p,q)\\
1 & \text{ otherwise}. 
\end{cases}$$ 
\end{theorem}

Let $0\,\not=\, X \,\in\, \NC _{\s\o (p,q)}$ and $\{X,\,H,\,Y\}\,\subset\, \s\o (p,q)$ a $\s\l_2(\R)$-triple. Let $$\Psi_{{\rm SO} 
(p,q)^\circ} (\OC_X) \,=\, \big( \d,\, \sgn_{\OC_X} \big).$$ Then $\Psi'_{{\rm SO} (p,q)^\circ} (\OC_X) \,=\, \d$. Let $V$ be a right 
$\R$-vector space of column vectors such that $ \dim_\R V\,=\,n$. Recall that the form $( \cdot ,\, \cdot )_d$ on $L(d-1)$ (see 
\eqref{new-form} for the definition) is symmetric for $d \in \O_\d$ and symplectic for $d \,\in\, \E_\d$. Let $(v^d_1,\, \dotsc \,, 
v^d_{t_d})$ be a $\R$-basis of $L(d-1)$ as in \cite[Proposition A.6]{BCM}. Recall that $\sgn_{\OC_X}$ determines the signature of $(\cdot, 
\,\cdot)_\theta$ on $L(\theta -1)$, $\theta \,\in\, \O_\d$; let $(p_\theta, q_\theta)$ be the signature of $(\cdot,\, \cdot)_\theta$. We 
may further assume that $(v^\theta_1,\, \dotsc \,, v^\theta_{t_\theta})$ is a standard orthogonal basis of $L(\theta - 1)$ for the form 
$(\cdot ,\, \cdot)_\theta$ for all $\theta \,\in\, \O_\d$, i.e.,
\begin{equation}\label{orthonormal-basis-odd-sopq}
( v^{\theta}_j, v^\theta_j)_{\theta} \,=\,
\begin{cases}
+1 & \text{ if } 1 \,\leq\, j \,\leq\, p_{\theta}\\
-1 & \text{ if } p_{\theta} \,<\, j \,\leq\, t_{\theta}. 
\end{cases}
\end{equation}
Similarly, we may assume that $(v^\eta_1,\, \dotsc \,, v^\eta_{t_\eta})$ is a symplectic basis of $L(\eta -1)$ for all $\eta
\,\in\, \E_{\d}$. This is equivalent to say that 
\begin{equation}\label{onb-eta-so-pq}
(v_j^\eta,\, v_{t_\eta/2 + j}^\eta)_\eta \, = \, 1 \ \ \text{ for } 1 \leq j \leq t_{\eta}/2 \ \ \text{ and } \ \ (v_j^\eta,\,
v_{i}^\eta)_\eta \, = \, 0 \ \ \text{ for } i \,\neq\, j + t_{\eta}/2 \,.
\end{equation}

For $\theta \,\in\, \O_\d$, let $\big\{{w}^\theta_{jl} \,\mid\, 1 \,\leq\, j \,\leq\, t_{\theta}, \, 0 \,\leq\, l \, \leq\, \theta-1 \big\}$ be the $\R$-basis of $M (\theta-1)$ as in \cite[Lemma A.9(2)]{BCM}. For each $0 \,\leq\, l \,\leq\, \theta-1$, define
$$
V^l (\theta) \,:=\, \text{ Span}_\R \{{w}^\theta_{1l}, \,\dots ,\, {w}^\theta_{t_{\theta}l} \}\, .
$$
The ordered basis $\big( {w}^\theta_{1l},\, \dots ,\, {w}^\theta_{t_{\theta}l} \big)$ of
$ V^l (\theta) $ is denoted by $\CC^l (\theta)$.

Next we will write down a general version of \cite[Lemma 4.18]{BCM} which will give a suitable description of reductive part of the centralizer of a nilpotent element in ${\s\o}(p,q) $. Recall that $\BC^l(d) = (X^lv^d_1 \,,\, \dotsc\, ,\, X^lv^d_{t_d})$ is the ordered basis of $X^l L(d-1)$ for $0\,\leq\, l \,\leq\, d -1,\, d
\,\in\, \N_\d$ as in \eqref{old-ordered-basis-part}.

\begin{lemma}\label{reductive-part-comp-so-pq}
Let $ X $ be a nilpotent element in $ \s\o(p,q) $ and $ \{X,\,H,\,Y\} $ be a $ \s\l_2 (\R)$-triple in $\s\o(p,q)$ containing $X$.
Then the following holds:
 $$
 \ZC_{ {\rm SO} (p,q) }(X,\,H,\,Y)\! =\! \left\{ g \in { {\rm SO} (p,q)} \middle\vert \! 
\begin{array}{ccc}
g (V^l (\theta)) \subset \ V^l (\theta) \ \text{and } \vspace{.14cm}\\
\big[g|_{ V^l(\theta)}\big]_{{\CC}^l(\theta)} = \big[g |_{ V^0 (\theta)}\big]_{{\CC}^0 (\theta)} \ { \forall } \, \theta \in \O_\d, 0 \leq l < \theta ~; \vspace{.14cm}\\
g(X^l L(\eta-1))\,\subset\, X^l L(\eta-1) \text{ and } \vspace{.14cm}\\
\! \big[g |_{ X^l L(\eta-1)} \big]_{{\BC}^l (\eta)}\! \!=\! \big[g |_{L(\eta-1)}\big]_{{\BC}^0 (\eta)}\ {\forall } \, \eta \,\in\, \E_\d,\,
0\, \leq\, l\,< \,\eta \!
\end{array}
\!\!   \right\}\!.
$$
\end{lemma}

\begin{proof}
We omit the proof as it is very similar to that of Lemma \ref{Z-so-nC-XHY}.
\end{proof}
 
We next impose orderings on the sets $\{v\,\in\, \CC^l (\theta)\,\mid\, \langle v,\, v \rangle > 0 \}$, $\{ v \,\in\, \CC^l (\theta)
\,\mid\, \langle v,\, v \rangle < 0 \}$. Define the ordered sets by $\CC^l_+ (\theta)\,$, $\CC^l_- (\theta)\,$, $\CC^l_+ (\zeta)$ and
$\CC^l_- (\zeta)$ as in \cite[(4.19), (4.20), (4.21), (4.22)]{BCM}, respectively according as $\theta \in \O^1_\d$ or $\zeta \in \O^3_\d$. 
For all $\theta \,\in\, \O_\d$ and $0 \,\leq\, l \,\leq\, \theta-1$, set
\begin{gather}
V^l_+ (\theta) \,:=\, \text{ Span}_\R \{ v \,\mid\, v \,\in\,\CC^l (\theta),\,\langle v,\,\,v \rangle
\,>\, 0 \},\nonumber\\
V^l_- (\theta) \,:=\, \text{ Span}_\R \{ v\,\mid\, v\, \in\, \CC^l (\theta),\,\,\langle v,\, v \rangle\, <\, 0 \}.\nonumber
\end{gather}
It is straightforward from \eqref{orthonormal-basis-odd-sopq}, and the orthogonality relations in \cite[Lemma A.9]{BCM}, that $\CC^l_+ 
(\theta)$ and $ \CC^l_- (\theta)$ are indeed ordered bases of $V^l_+ (\theta)$ and $V^l_- (\theta)$, respectively. For $\eta\,\in\, \E_\d, \, 
0 \,\leq\, l \,\le\, \eta/2 -1$, set
$$
W^l(\eta) \, : =\, X^lL(\eta -1) + X^{\eta -1-l} L(\eta -1).
$$
Now we will construct a standard orthogonal basis of $W^l$ as done in \eqref{definition-w-basis-sonC-even1}, 
\eqref{definition-w-basis-sonC-even2}. For $l$ even, $0\,\leq\, l \,\leq\, \eta/2-1$, let
$$
 {w}^{\eta}_{jl} \,:=\,
\begin{cases}
 \big( X^l v^{\eta}_j + X^{\eta-1-l} v^{\eta}_{t_\eta/2 + j} \big) / {\sqrt{2}}    & \text{ if } \ 1\,<\,j \,\leq \, t_\eta/2  \\
 \big( X^l v^{\eta}_j  -  X^{\eta-1-l} v^{\eta}_{j- t_\eta/2 } \big) / {\sqrt{2}}    & \text{ if } \ t_\eta/2\, < \,j \,\leq \, t_\eta  \\
 \big( X^l v^{\eta}_{j- t_\eta}  -  X^{\eta-1-l} v^{\eta}_{j - t_\eta/2} \big) / {\sqrt{2}}    & \text{ if } \  t_\eta \, < \,j \,\leq \, 3t_\eta/2  \\
 \big( X^l v^{\eta}_{j- t_\eta}  +  X^{\eta-1-l} v^{\eta}_{j - 3t_{\eta}/2} \big) / {\sqrt{2}}  & \text{ if } \  3t_\eta/2 \, < \,j \,\leq \, 2t_\eta \,.
 \end{cases}
$$
For $l$ odd, $0\,\leq\, l \,\leq\, \eta/2-1$, let
$$
{w}^{\eta}_{jl} \,:=\,
\begin{cases}
\big( X^l v^{\eta}_j  -  X^{\eta-1-l} v^{\eta}_{j + t_\eta/2 } \big) / {\sqrt{2}}     & \text{ if } \ 1\,<\,j \,\leq \, t_\eta/2  \\
\big( X^l v^{\eta}_j  +  X^{\eta-1-l} v^{\eta}_{j - t_\eta/2} \big) / {\sqrt{2}}  & \text{ if } \ t_\eta/2\, < \,j \,\leq \, t_\eta  \\
\big( X^l v^{\eta}_{j- t_\eta}+X^{\eta-1-l} v^{\eta}_{j- t_\eta/2} \big) / {\sqrt{2}} & \text{ if } \  t_\eta \, < \,j \,\leq \, 3t_\eta/2 \\
\big( X^l v^{\eta}_{j- t_\eta}-  X^{\eta-1-l} v^{\eta}_{j - 3t_{\eta}/2} \big) / {\sqrt{2}} & \text{ if } \  3t_\eta/2 \, < \,j \,\leq \, 2t_\eta \,.
\end{cases}
$$
Using \eqref{onb-eta-so-pq} it follows that  $\{w^\eta_{jl} \, \mid \, 1\,\leq\, j \,\leq\, 
2t_\eta, \, 0\,\leq\, l \,\leq\, \eta/2 -1 \}$ is a standard orthogonal basis of $M(\eta -1)$ with respect to $\<>$ for all $\eta \,\in \,\E_\d$. The 
ordered basis $(w^\eta_{1 l} \, ,\, \dotsc \, ,\, w^\eta_{2t_\eta l} )$ of $W^l(\eta)$ is denoted by 
$\DC^l(\eta)$.

\begin{remark}\label{structure-of-centralizer-sopq}
{\rm 
We follow the notation of Lemma \ref{reductive-part-comp-so-pq}.
Let $g \,\in\, \ZC_{ {\rm SO} (p,q)}(X,\, H,\, Y)$.  Let $\theta \,\in\, \O_\d$ and $\eta \,\in\, \E_\d$.   
Then it follows from Lemma \ref{reductive-part-comp-so-pq} that $ g$ keeps the subspaces $V^0(\theta)$ and $L(\eta -1)$  invariant. 
Since the restriction of $\<> $ is a symmetric form on $V^0(\theta)$ we have $g|_{ V^0(\theta)} \,\in\,
{\rm O}(V^0(\theta), \,\<>)$. 
Further recall that the form $ (\cdot,\,\cdot)_\eta $, as defined in \eqref{new-form}, is symplectic on $L(\eta-1)$, and
$(gx,\,gy)_\eta \,=\, (x,\,y)_\eta$ for all $x,\,y\,\in\, L(\eta-1)$; see \cite[Remark A.7]{BCM}. 
Thus $g |_{  L(\eta-1)} \,\in\, {\rm Sp}(L(\eta-1),\, (\cdot,\, \cdot)_\eta)$.
}
\end{remark}

In the next lemma, which generalizes \cite[Lemma 4.19]{BCM}, we specify a maximal compact subgroup of $\ZC_{ {\rm 
SO} (p,q) }(X,\,H,\,Y)$ which will be used in Theorem \ref{max-cpt-wrt-onb-sop}. The notation $(-1)^l$ stands for 
the sign `$+$' or the sign `$-$' according as $l$ is an even or odd integer. Recall that the $\R$-algebra 
embedding $\wp_{m, \C}$ is defined in \eqref{embedding-C2R}.
 
\begin{lemma}\label{max-cpt-centralizer-sopq}
Let $K$ be the subgroup of $\ZC_{ {\rm SO} (p,q) }(X,\,H,\,Y)$ consisting of elements $g$ in
$\ZC_{ {\rm SO} (p,q) }(X,\,H,\,Y)$ satisfying the following conditions:
\begin{enumerate}

\item $g ( V^l_+ (\theta)) \,\subset \, V^l_+ (\theta)$ and $g ( V^l_- (\theta)  ) \,
\subset \, V^l_-(\theta)$, for all  $ \theta \,\in\, \O_\d$ and $ 0 \,\leq\, l \,\leq\, \theta -1$.

\item When $ \theta \,\in\, \O^1_\d$,
$$
\Big[g |_{ V_+^0 (\theta) }\Big]_{{\CC}^0_+ (\theta)} \,= 
\begin{cases}                                                            
\Big[g |_{  V^l_{(-1)^{l}} (\theta)}\Big]_{{\CC}^l_{(-1)^{l}} (\theta)} & \text{ for all } 0 \leq l < (\theta-1)/2 \vspace{.15cm}\\
\Big[g |_{  V^{(\theta-1)/2}_{+} (\theta)}\Big]_{{\CC}^{(\theta-1)/2}_{+} (\theta)} \vspace{.15cm}  \\
\Big[g |_{  V^l_{(-1)^{l+1}} (\theta)}\Big]_{{\CC}^l_{(-1)^{l+1} } (\theta)} &  \text{ for all } (\theta-1)/2 < l \leq \theta-1,
\end{cases}                 
$$
$$
\Big[g |_{ V_-^0 (\theta) }\Big]_{{\CC}^0_- (\theta)} \,=
\begin{cases}
\Big[g |_{  V^l_{(-1)^{l+1}} (\theta)}\Big]_{{\CC}^l_{(-1)^{l+1} } (\theta)} &  \text{ for all } 0 \leq l < (\theta-1)/2 \vspace{.15cm}\\
\Big[g |_{  V^{(\theta-1)/2}_{-} (\theta)}\Big]_{{\CC}^{(\theta-1)/2}_{-} (\theta)} \vspace{.15cm}  \\
\Big[g |_{  V^l_{(-1)^{l}} (\theta)}\Big]_{{\CC}^l_{(-1)^{l}} (\theta)} &  \text{ for all } (\theta-1)/2 < l \leq \theta-1 .     
\end{cases}
$$ 

\item When $ \zeta \,\in\, \O^3_\d $,
$$
\Big[g |_{ V_+^0 (\zeta) }\Big]_{{\CC}^0_+ (\zeta)} = 
             \begin{cases}                                                            
                 \Big[g |_{  V^l_{(-1)^{l}} (\zeta)}\Big]_{{\CC}^l_{(-1)^{l}} (\zeta)} &  \text{ for all } 0 \leq l < (\zeta-1)/2 \vspace{.15cm}\\            
      \Big[g |_{  V^{(\zeta-1)/2}_{-} (\zeta)}\Big]_{{\CC}^{(\zeta-1)/2}_{-} (\zeta)} \vspace{.15cm}  \\
                 \Big[g |_{  V^l_{(-1)^{l+1}} (\zeta)}\Big]_{{\CC}^l_{(-1)^{l+1} } (\zeta)} & \text{ for all } (\zeta-1)/2 < l \leq \zeta-1,
                 \end{cases}                 
$$
$$
  \Big[g |_{ V_-^0 (\zeta) }\Big]_{{\CC}^0_- (\zeta)} \,=
     \begin{cases}
      \Big[g |_{  V^l_{(-1)^{l+1}} (\zeta)}\Big]_{{\CC}^l_{(-1)^{l+1} } (\zeta)} &  \text{ for all } 0 \leq l < (\zeta-1)/2 \vspace{.15cm}\\
      \Big[g |_{  V^{(\zeta-1)/2}_{+} (\zeta)}\Big]_{{\CC}^{(\zeta-1)/2}_{+} (\zeta)} \vspace{.15cm}  \\
            \Big[g |_{  V^l_{(-1)^{l}} (\zeta)}\Big]_{{\CC}^l_{(-1)^{l}} (\zeta)} &  \text{ for all } (\zeta-1)/2 < l \leq \zeta-1.    
     \end{cases}
$$

\item For all $\eta \,\in\, \E_\d$ and  $0\leq l \leq \eta/2-1$ the inclusion $g(W^l(\eta))\,\subseteq\, W^l(\eta)$ holds.
 
\item For all $\eta \,\in\, \E_\d$, $0\,\leq\, l \,\leq\, \eta/2-1$, there exists $A(\eta),\, B(\eta)$ such that $A(\eta)+ \sqrt{-1} 
B(\eta) \,\in\, {\rm U}(t_{\eta/2})$ such that
$$
 \big[g|_{W^l(\eta)}\big]_{\DC^l(\eta)} \, = \big(\wp_{t_\eta/2, \C}(A(\eta)+ \sqrt{-1} B(\eta)) \big)_\blacktriangle ^2
 =\begin{pmatrix}
                                               A(\eta) & - B(\eta)\\
                                               B(\eta) &   A(\eta)\\
                                               &&  A(\eta) & - B(\eta)\\
                                               && B(\eta) &   A(\eta)
                                              \end{pmatrix} \,.
 $$
\end{enumerate}
\end{lemma}
 
\begin{proof}
Let $K'$ be the subgroup consisting of all $g \,\in\, \ZC_{{\rm SO} (p,q)}(X,\,H,\,Y)$ satisfying the conditions in
\eqref{max-cpt-odd-subspace-so-pq1}, \eqref{max-cpt-odd-subspace-so-pq} and \eqref{max-cpt-even-subspace-so-pq} below:
\begin{align}
 &  g ( V^l_+ (\theta))\, \subset \, V^l_+ (\theta) \, , \:  g ( V^l_- (\theta)) \subset \, V^l_-(\theta), \, \text{for all } \theta \in \O_\d,\,  0 \leq l < \theta   ; \label{max-cpt-odd-subspace-so-pq1}
 \\
  &\big[g |_{  V^l (\theta) }\big]_{{\CC}^l (\theta)} =   \big[g |_{ V^0 (\theta) }\big]_{{\CC}^0 (\theta)}  
\,,\qquad  \  \qquad \text{ for all } \theta\,\in\, \O_\d,\,  0 \leq l < \theta  \,; \label{max-cpt-odd-subspace-so-pq}
       \\
     & g |_{ L(\eta-1)}  \text{ commutes with }  { J}_{\BC^0(\eta)} \,, \quad \ \qquad \text{ for all } \eta \in \E_\d\,   \label{max-cpt-even-subspace-so-pq}.
   \end{align}
In view of Lemma \ref{reductive-part-comp-so-pq}, Remark \ref{structure-of-centralizer-sopq} and \cite[Lemma 4.34]{BCM} which is analogous 
to Lemma \ref{max-cpt-symplectic-gps}, it is straightforward that $K'$ is
actually a maximal compact subgroup of $\ZC_{{\rm SO} (p,q)}(X,\,H,\,Y)$. 
Thus to prove the lemma it is suffices to show that $K\,=\, K'$. Let $g\,\in\, \ZC_{{\rm SO} (p,q)}(X,\,H,\,Y)$. We omit the proof of the fact 
that $g$ satisfies \eqref{max-cpt-odd-subspace-so-pq1} and \eqref{max-cpt-odd-subspace-so-pq} if and only if $g$ satisfies (1), (2) and 
(3) in the statement of the lemma, as this follows from \cite[Lemma 4.19]{BCM} when $\theta \,\in\, \O_\d^1$ and $\zeta \,\in\, \O_\d^{3}$.
Let $$[g|_{L(\eta -1)}]_{\BC^0(\eta)} \,:=\, \begin{pmatrix}
A(\eta) & C(\eta)\\
B(\eta) & D(\eta)
\end{pmatrix}.$$
Next suppose that $g$ satisfies \eqref{max-cpt-even-subspace-so-pq}. Then it follows that $A(\eta)\,=\, D(\eta)$ and $B(\eta )\,=\, - C(\eta)$. 
Using \cite[Lemma 4.34]{BCM}  and Remark \ref{structure-of-centralizer-sopq}, we have $ A(\eta) + \sqrt{-1} B(\eta )\,\in\, {\rm U}(t_\eta/2)$. 
Now statement (4) of the lemma follows from the definition of $\DC^l(\eta)$. Also statement (5) of the lemma holds, as
$
 \big[g|_{W^l(\eta)}\big]_{\DC^l(\eta)}      =\begin{pmatrix}
                                               A(\eta) & - B(\eta)\\
                                               B(\eta) &   A(\eta)\\
                                                &&  A(\eta) & - B(\eta)\\
                                                 && B(\eta) &   A(\eta)
                                              \end{pmatrix}.$
 Lastly, we assume that $g$ satisfies statements (4) and (5) of  Lemma \ref{max-cpt-centralizer-sopq}. Then it follows that 
$$[g|_{L(\eta -1)}]_{\BC^0(\eta)}\, =\, \begin{pmatrix}
                                     A(\eta) & - B(\eta)\\
                                     B(\eta) &   A(\eta) 
                                     \end{pmatrix}.
$$
Now clearly \eqref{max-cpt-even-subspace-so-pq} holds for $g$ and this completes the proof of the lemma. 
\end{proof}

For $\eta\,\in\, \E_\d $, define
$$
\DC^l_+(\eta) \,:=\, (w_{1l}^\eta,\, \dots, \,w^\eta_{t_\eta l}) \quad \text{ and } \quad  \DC^l_-(\eta)
\,:=\, (w_{(t_\eta+1)\,l}^\eta,\, \dots,\, w^\eta_{2t_\eta l})\,.
$$
Set 
$$
\DC_+(\eta) \,:=\, \DC^0_+(\eta) \vee \cdots \vee \DC^{\eta/2-1}_+(\eta) \quad \text{and} \quad \DC_-(\eta)
\,:=\, \DC^0_-(\eta) \vee \cdots \vee \DC^{\eta/2-1}_-(\eta).
$$
When $\theta \,\in\, \O_\d$, define $\CC_+(\theta)$ and $\CC_-(\theta)$ as in \cite[(4.19), (4.20), (4.21), (4.22)]{BCM}. Set 
$$
\CC_+(\theta) \,:=\, \CC^0_+(\theta) \vee \cdots \vee \CC^{\theta-1}_+(\theta) \quad \text{and} \quad \CC_-(\theta)
\,:=\, \CC^0_-(\theta) \vee \cdots \vee \CC^{\theta -1}_-(\theta).
$$

Let $\alpha\, :=\, \# \E_\d $, $\beta \,:=\, \# \O^1_\d$ and $ \gamma \,:=\, \# \O^3_\d $.
We enumerate $\E_\d \,=\,\{ \eta_i \,\mid\, 1 \,\leq\, i \,\leq\, \alpha \}$ such that
$\eta_i \,< \,\eta_{i+1}$, and 
$\O^1_\d \,=\,\{ \theta_j \,\mid\, 1 \leq j \leq \beta \}$ such that
$\theta_j \,< \,\theta_{j+1}$; similarly enumerate $\O^3_\d \,=\,\{ \zeta_j \,\mid\, 1 \,\leq\, j
\,\leq\, \gamma \}$ such that $\zeta_j \,<\, \zeta_{j+1}$. Now define
$$
\EC_+ \,:=\, \DC_+ (\eta_1) \vee \cdots \vee \DC_+ (\eta_{\alpha})\, ; \ \  \,
\OC^1_+ \,:=\, \CC_+ (\theta_1) \vee \cdots \vee \CC_+ (\theta_{\beta})\, ; \ \ \,\OC^3_+ := \CC_+ (\zeta_1) \vee \cdots \vee \CC_+ (\zeta_{\gamma});
$$
$$
\EC_- \,:=\, \DC_- (\eta_1) \vee \cdots \vee \DC_- (\eta_{\alpha}) ; \ 
\OC^1_- \,:= \,\CC_-(\theta_1) \vee \cdots \vee \CC_-(\theta_{\beta})\, \text{ and } \,\OC^3_-
\,:=\, \CC_- (\zeta_1) \vee \cdots \vee \CC_-(\zeta_{\gamma}).
$$
Also we define
\begin{equation}\label{orthogonal-basis-final-sopq}
\HC_+ \,:= \,\EC_+ \vee \OC^1_+ \vee\OC^3_+ , \ \ \HC_- := \EC_- \vee \OC^1_- \vee\OC^3_- \ \text{ and } \
\HC \,:= \,\HC_+ \vee \HC_-. 
\end{equation}
It is clear that $ \HC$ is a standard orthogonal basis of $V$ such that $\HC_+\, =\, \{ v \,\in\, \HC \,\mid\, \langle v,\, v \rangle \,=\,1 \}$ and $\HC_- \,= \,\{ v \,\in\, \HC \,\mid\, \langle v,\, v \rangle \,=\,-1 \}$. In particular, $\# \HC_+ \,=\, p$ and $\# \HC_- \,=\,q$. 
Also, we have the following relations:
$$
 \sum_{i=1}^\alpha \frac{\eta_i}{2}t_{\eta_i} + \sum_{j=1}^\beta \big( \frac{\theta_j+1}{2}p_{\theta_j} +  \frac{\theta_j-1}{2}q_{\theta_j}\big)  + \sum_{k=1}^\gamma\big( \frac{\zeta_k-1}{2}p_{\zeta_k} +  \frac{\zeta_k+1}{2}q_{\zeta_k}\big) =p
$$
and 
$$
 \sum_{i=1}^\alpha \frac{\eta_i}{2}t_{\eta_i}  + \sum_{j=1}^\beta \big( \frac{\theta_j-1}{2}p_{\theta_j} +  \frac{\theta_j+1}{2}q_{\theta_j}\big)  + \sum_{k=1}^\gamma\big( \frac{\zeta_k+1}{2}p_{\zeta_k} +  \frac{\zeta_k-1}{2}q_{\zeta_k}\big) =q.
$$

The $\R$-algebra $$ \prod_{i= 1}^\alpha \big({\rm M}_{t_{\eta}/2}(\R) \times {\rm M}_{t_{\eta}/2}(\R) \big) \times
\prod_{j=1}^\beta \big( {\rm M}_{p_{\theta_j}} (\R) \times {\rm M}_{q_{\theta_j}} (\R) \big) \times \prod_{k=1}^\gamma
\big( {\rm M}_{p_{\zeta_k}} (\R) \times {\rm M}_{q_{\zeta_k}} (\R) \big)$$
is embedded in ${\rm M}_p(\R)$ and in ${\rm M}_q(\R)$ as follows:
\begin{align}\label{map-Dp-sopq}
\Db_{p} \,\colon\,  \prod_{i= 1}^\alpha \big({\rm M}_{t_{\eta}/2}(\R) \times {\rm M}_{t_{\eta}/2}(\R) \big) \times &
\prod_{j=1}^\beta \big( {\rm M}_{p_{\theta_j}} (\R) \times {\rm M}_{q_{\theta_j}} (\R) \big) \times \prod_{k=1}^\gamma
\big( {\rm M}_{p_{\zeta_k}} (\R) \times {\rm M}_{q_{\zeta_k}} (\R) \big)    \nonumber  \\
&\longrightarrow \,{\rm M}_{p}(\R)  
\end{align}
 \begin{align*}
 \big(A_{\eta_1}, &B_{\eta_1}, \dots , A_{\eta_\alpha}, B_{\eta_\alpha} \,;\,
 C_{\theta_1}, D_{\theta_1} , \dots,  C_{\theta_\beta}, D_{\theta_\beta}\,;\,
 E_{\zeta_1}, F_{\zeta_1}, \dots ,  E_{\zeta_\gamma}, F_{\zeta_\gamma}   \big) \,\longmapsto \\
  &  \bigoplus_{i=1}^\alpha  \wp_{t_\eta/2, \C}(A_i +\sqrt{-1} B_i)_\blacktriangle^{\eta/2}
   \oplus \bigoplus_{j=1}^\beta \Big( \big( C_{\theta_j} \oplus D_{\theta_j}\big)_\blacktriangle ^{\frac{\theta_j-1}{4}} \oplus C_{\theta_j} \oplus \big( C_{\theta_j} \oplus D_{\theta_j}\big)_\blacktriangle ^{\frac{\theta_j-1}{4}} \Big) \\ 
 \oplus & \bigoplus_{k=1}^\gamma \Big( \big( E_{\zeta_k}\oplus F_{\zeta_k} \big) _\blacktriangle ^{\frac{\zeta_k + 1}{4}}  \oplus \big( F_{\zeta_k}\oplus E_{\zeta_k} \big) _\blacktriangle ^{\frac{\zeta_k -3}{4}}  \oplus F_{\zeta_k} \Big)
  \end{align*}
and 
\begin{align}\label{map-Dq-sopq}
\Db_{q}\, \colon\, \prod_{i= 1}^\alpha \big({\rm M}_{t_{\eta}/2}(\R) \times {\rm M}_{t_{\eta}/2}(\R) \big) \times &
\prod_{j=1}^\beta \big( {\rm M}_{p_{\theta_j}} (\R) \times {\rm M}_{q_{\theta_j}} (\R) \big) \times \prod_{k=1}^\gamma \big( {\rm M}_{p_{\zeta_k}} (\R) \times {\rm M}_{q_{\zeta_k}} (\R) \big)   \nonumber   \\ 
& \longrightarrow \,{\rm M}_{q}(\R)
\end{align}
\begin{align*}
\big(A_{\eta_1}, &B_{\eta_1}, \dots , A_{\eta_\alpha}, B_{\eta_\alpha} \,;\,
C_{\theta_1}, D_{\theta_1} , \dots,  C_{\theta_\beta}, D_{\theta_\beta} \,;\,
E_{\zeta_1}, F_{\zeta_1}, \dots ,  E_{\zeta_\gamma}, F_{\zeta_\gamma} \big) \,\longmapsto \\
  &  \bigoplus_{i=1}^\alpha  \wp_{t_\eta/2, \C}(A_i +\sqrt{-1} B_i)_\blacktriangle^{\eta/2}
   \oplus \bigoplus_{j=1}^\beta \Big( \big( D_{\theta_j} \oplus C_{\theta_j}\big)_\blacktriangle ^{\frac{\theta_j-1}{4}} \oplus D_{\theta_j} \oplus \big( D_{\theta_j} \oplus C_{\theta_j}\big)_\blacktriangle ^{\frac{\theta_j-1}{4}} \Big) \\ 
 \oplus & \bigoplus_{k=1}^\gamma \Big( \big( F_{\zeta_k}\oplus E_{\zeta_k} \big) _\blacktriangle ^{\frac{\zeta_k + 1}{4}}  \oplus \big( E_{\zeta_k}\oplus F_{\zeta_k} \big) _\blacktriangle ^{\frac{\zeta_k -3}{4}}  \oplus E_{\zeta_k} \Big). 
 \end{align*}
Define two characters
\begin{align}\label{map-chi-p-sopq}
\bigchi_{p} \,\colon\,  \prod_{i=1}^\alpha  {\rm U}(t_{\eta_i}/2) \times
\prod_{j=1}^\beta \big( {\rm O}_{p_{\theta_j}} \times {\rm O}_{q_{\theta_j}}  \big) \times \prod_{k=1}^\gamma \big( {\rm O}_{p_{\zeta_k}}  \times {\rm O}_{q_{\zeta_k}}  \big) \longrightarrow \R\setminus\{0\}
\end{align}
\begin{align*}
\big(   A_{\eta_1} &,\,  \dots\,, A_{\eta_\alpha} ; C_{\theta_1}, D_{\theta_1} , \dots,  C_{\theta_\beta}, D_{\theta_\beta}; E_{\zeta_1}, F_{\zeta_1},  \dots ,  E_{\zeta_\gamma}, F_{\zeta_\gamma}  \big)  \\
&  \: \longmapsto \: \prod_{i=1}^\alpha \det  \wp_{t_{\eta_i}/2, \C} (A_{\eta_i} )^{\eta_i/2} \: \prod_{j=1}^\beta \det  C_{\theta_j} \:  \prod_{k=1}^\gamma \det  E_{\zeta_k}
\end{align*}
and
\begin{align}\label{map-chi-q-sopq}
\bigchi_{q} \, \colon\, \prod_{i=1}^\alpha {\rm U}(t_{\eta_i}/2)  \times \prod_{j=1}^\beta \big( {\rm O}_{p_{\theta_j}} \times {\rm O}_{q_{\theta_j}}  \big) \times \prod_{k=1}^\gamma \big( {\rm O}_{p_{\zeta_k}}  \times {\rm O}_{q_{\zeta_k}}  \big) \longrightarrow \R\setminus\{0\}
\end{align}
\begin{align*}
\big( A_{\eta_1}  &\,,\, \dots, \,A_{\eta_\alpha} \,;\, C_{\theta_1}, D_{\theta_1} , \dots, C_{\theta_\beta}, D_{\theta_\beta}; E_{\zeta_1}, F_{\zeta_1}, \dots ,  E_{\zeta_\gamma}, F_{\zeta_\gamma} \big) \\
& \longmapsto \:\,  \prod_{i=1}^\alpha \det \wp_{t_{\eta_i}/2, \C} (A_{\eta_i})^{\eta_i/2} \:  \prod_{j=1}^\beta \det D_{\theta_j} \: \prod_{k=1}^\gamma \det  F_{\zeta_k}\,.  
\end{align*}

Let $\Lambda_\HC \,\colon\, {\rm End}_\R \R^{n} \,\longrightarrow\, {\rm M}_n (\R)$ be the isomorphism of $\R$-algebras induced by the ordered basis $\HC$ as in \eqref{orthogonal-basis-final-sopq}. Let $M$ be the maximal compact subgroup of ${\rm SO}(p,q)$ which leaves invariant simultaneously the two subspaces spanned by $\HC_+$ and $\HC_-$. Clearly, $\Lambda_\HC (M) \,=\, {\rm S} ({\rm O} (p) \times {\rm O} (q))$.

\begin{theorem}\label{max-cpt-wrt-onb-sop} 
Let $X \,\in\, \NC_{\s\o(p,q)}$ and $\Psi_{{\rm SO}(p,q)^\circ}(\OC_X)\,=\,(\d, \,\sgn_{\OC_X})$.
Let $\alpha \,:=\, \# \E_\d$, $\beta \,:=\, \# \O^1_\d$ and $ \gamma \,:=\, \# \O^3_\d $. Let $\{X,\,H,\,Y\}\,\subset\, \s\o(p,q)$ be a
$\s\l_2(\R)$-triple, and let $(p_\theta,\, q_\theta)$ be the signature of the form $(\cdot, \cdot)_\theta$ for all $\theta \in \O_\d$. 
Let $K$ be the maximal compact subgroup of $\ZC_{{\rm SO} (p,q)} ( X, H, Y)$ as in Lemma \ref{max-cpt-centralizer-sopq}. 
Let  the maps $ \Db_p, \Db_q, \bigchi_p $ and $ \bigchi_q $ be defined as in \eqref{map-Dp-sopq}, \eqref{map-Dq-sopq}, \eqref{map-chi-p-sopq} and \eqref{map-chi-q-sopq}, respectively.  
Then  $\Lambda_\HC (K) \subset {\rm S}({\rm O}(p)\times {\rm O}(q))$ is given by
\begin{align*} 
\Lambda_\HC(K) = 
 \Bigg\{ \Db_p (g) \oplus \Db_q(g)  \Biggm| \!\!
                                         \begin{array}{c} \!\!
                                         g \in \prod_{i=1}^\alpha {\rm U}(t_{\eta_i}/2) \times                                        \prod_{j=1}^\beta \big( {\rm O}_{p_{\theta_j}} \times {\rm O}_{q_{\theta_j}} \big) 
                                         \times \prod_{k=1}^\gamma \big( {\rm O}_{p_{\zeta_k}} \times {\rm O}_{q_{\zeta_k}} \big) \vspace{.1cm}\! \\  \text{ and } \  ~ \bigchi_{p}(g) \bigchi_{q}(g) =1	
                                        \end{array} \!\! \! \Bigg\}.
\end{align*}
\end{theorem}

\begin{proof}
This follows by writing the matrices of the elements of the maximal compact subgroup $K$ in Lemma \ref{max-cpt-centralizer-sopq} with 
respect to the ordered basis $\HC$ as in \eqref{orthogonal-basis-final-sopq}.
\end{proof}

Since ${\rm SO}(p,q)^\circ$ is normal in ${\rm SO}(p,q)$, so is $\ZC_{{\rm SO}(p,q)^\circ} (X,\,H,\,Y)$ in $\ZC_{{\rm SO} (p,q)}
(X,\,H,\,Y)$. Recall that $K$ is a  maximal compact subgroup in $\ZC_{{\rm SO} (p,q)}(X,\,H,\,Y)$. Thus it follows that 
$$
\widetilde{K}\,:=\, K \cap \ZC_{ {\rm SO}(p,q)^\circ}(X,\,H,\,Y) \,=\, K\cap {\rm SO} (p,q)^\circ 
$$ 
is a maximal compact subgroup of $\ZC_{{\rm SO}(p,q)^\circ}(X,\,H,\,Y)$.
In the next result we obtain an explicit description of $\Lambda_\HC (\widetilde{K})$ in ${\rm SO} (p)\times {\rm SO} (q)$.

\begin{theorem}\label{max-cpt-0-wrt-onb-sopq} 
Let $X \,\in\, \NC_{\s\o(p,q)}$ and $\Psi_{{\rm SO}(p,q)^\circ}(\OC_X) \,=\,(\d,\, \sgn_{\OC_X})$.
Let $\alpha \,:=\, \# \E_\d,\,  \beta \,:=\, \# \O^1_\d$ and $ \gamma \,:=\, \# \O^3_\d $. Let $\{X,H,Y\}\,\subset\, \s\o(p,q)$
be a $\s\l_2(\R)$-triple, and let $(p_\theta,\, q_\theta)$ be the signature of the form $(\cdot, \,\cdot)_\theta$ for all $\theta\,\in\, \O_\d$. 
Let $\widetilde{K}$ be the maximal compact subgroup of $\ZC_{{\rm SO} (p,q)^\circ} ( X,\, H,\, Y)$ as in the preceding paragraph. 
Let  the maps $ \Db_p, \,\Db_q,\, \bigchi_p $ and $ \bigchi_q $ be defined as in \eqref{map-Dp-sopq}, \eqref{map-Dq-sopq},
\eqref{map-chi-p-sopq} and \eqref{map-chi-q-sopq}, respectively. 
Then  $\Lambda_\HC (\widetilde{K}) \,\subset\, {\rm SO}(p)\times {\rm SO}(q)$ is given by
\begin{align*}  
 \Bigg\{ \Db_p (g) \oplus \Db_q(g)  \Biggm| 
                                         \begin{array}{c}
                                         g \in \prod_{i=1}^\alpha {\rm U}(t_{\eta_i}/2) \times                                
                                       \prod_{j=1}^\beta \big( {\rm O}_{p_{\theta_j}} \times {\rm O}_{q_{\theta_j}} \big) 
                                         \times \prod_{k=1}^\gamma \big( {\rm O}_{p_{\zeta_k}} \times {\rm O}_{q_{\zeta_k}} \big) \vspace{.1cm} \\  \text{ and } \  ~ \bigchi_{p}(g) =1 , \ \bigchi_{q}(g) =1
                                        \end{array} \Bigg\}.
\end{align*}

The nilpotent orbit $\OC_X$ in $\s\o(p,q)$ is homotopic to ${\rm SO}(p)\times {\rm SO}(q))/ \Lambda_\HC(\widetilde{K})$.
\end{theorem}

\begin{proof}
Let $V_+$ and $V_-$ be the $\R$-span of $\HC_+$ and $\HC_-$, respectively. Let $M$ be the maximal compact subgroup in
${\rm SO}(p,q)$ which simultaneously leaves the subspaces  $V_+$ and $V_-$  invariant. It is clear that $M^\circ$ is a maximal
compact subgroup of ${\rm SO}(p,q)^\circ$. Hence,
$$
M^\circ\, =\, {\rm SO}(p,q)^\circ \cap M \,= \,\{g\in {\rm SO}(p,q) \,\,\mid\, \det g|_{V_+}\,=\,1,\,\ \det g|_{V_-}\,=\,1 \}.
$$
As $K\,\subset\, M$, we have that $K\cap {\rm SO}(p,q)^\circ \,=\, K\cap M^\circ$.  
The first part of the proposition now follows. 
For the second part we use Theorem \ref{mostow-corollary}.
\end{proof}

\subsection{Homotopy types of the nilpotent orbits in \texorpdfstring{${\s\p}(n, \C)$}{Lg}}\label{sec-sp-nc}

Let $n$ be a positive integer. The aim in this subsection is to write down the homotopy types of the nilpotent 
orbits in the simple Lie algebra ${\s\p}(n,\C)$ as compact homogeneous spaces. Throughout this subsection $\<>$ 
denotes the symplectic form on $\C^{2n}$ defined by $$\langle x,\, y \rangle \,:=\, x^t{\rm J}_{n} y,\, \ x,\,y \,\in\, 
\C^{2n}\, ,$$ where ${\rm J}_{n}$ is as in \eqref{defn-I-pq-J-n}. 

We first recall a suitable parametrization of the nilpotent orbits $\NC ({\rm Sp} (n,\C))$. 
Let $$\Psi_{{\rm SL}_n (\C)}\,:\,\NC ({\rm SL}_n(\C)) \,\longrightarrow\,\PC (n)$$ be the  parametrization of
$\NC ({\rm SL}_n (\C))$; see Section \ref{sec-sl-nc} for details.
Since ${\rm Sp} (n,\C) \,\subset\, {\rm SL}_{2n} (\C)$ (consequently as, the set of nilpotent elements
$\NC_{{\s\p}(n,\C)} \,\subset\, \NC_{\s\l_{2n}(\C)}$)
we have the inclusion map
$$\Theta_{{\rm Sp}(n,\C)}\,:\, \NC ({\rm Sp} (n,\C)) \,\longrightarrow\,\NC ( {\rm SL}_{2n} (\C)).$$ Let
$$
\Psi_{{\rm Sp}(n,\C)}\,:= \,\Psi_{{\rm SL}_{2n} (\C)}\circ
\Theta_{{\rm Sp}(n,\C)}\,\colon\, \NC ({\rm Sp} (n,\C))
\,\longrightarrow\, \PC (2n)
$$ 
be the composition.
Recall that $$\Psi_{{\rm Sp}(n,\C)} (\NC ({\rm Sp} (n,\C)))\,\subset\, \PC_{-1} (2n)$$ 
(this follows form \cite[Proposition A.6]{BCM}). Hence we have the following parametrizing map :
$$
\Psi_{{\rm Sp}(n,\C)}\,\colon\, \NC ({\rm Sp} (n,\C))
\,\longrightarrow\, \PC_{-1} (2n) \,.
$$

\begin{theorem}[{\cite[Theorem 5.1.3]{CoM}}] \label{sp-nC-parametrization}
The above map $\Psi_{{\rm Sp}(n,\C)}$ is bijective.
\end{theorem}
  
Let $0\,\not=\, X \,\in\, \NC _{\s\p (n,\C)}$ and $\{X,\,H,\,Y\}$ be a $\s\l_2(\R)$-triple in $\s\p (n,\C)$. Recall that the form $( \cdot ,\, \cdot )_d$ on $L(d-1)$ is symmetric when $d \in \E_\d$ and symplectic when $d \in \O_\d$. Let $(v^d_1,\, \dotsc \,, v^d_{t_d})$ be an $\C$-basis of $L(d-1)$ as in \cite[Proposition A.6]{BCM}. It follows form \cite[Proposition A.6]{BCM} that $(v^\eta_1,\, \dotsc \,, v^\eta_{t_\eta})$ is an orthonormal basis of $L(\eta - 1)$ for the form $(\cdot , \cdot)_\eta$ for all $\eta \in \E_\d$, i.e., 
 \begin{equation} \label{onb-eta-sp-nc}
  (v_j^\eta, v_j^\eta)_\eta \, = \, 1 \  \text{ for } 1 \leq j \leq t_\eta \ \ \text{ and } \ \  (v_j^\eta, v_i^\eta)_\eta \, = \, 0  \  \text{ for } j \neq i. 
 \end{equation}
For all $\theta \,\in \,\O_\d$, as $(\cdot,\, \cdot)_\theta$ is a symplectic form, we may
assume that $$( v^{\theta}_1,\, \dots,\, v^{\theta}_{t_{\theta}/2};\, v^{\theta}_{t_{\theta}/2+1},
\,\dots,\,  v^{\theta}_{t_{\theta}})$$ is a symplectic basis of $L(\theta - 1)$.  
This is equivalent to saying that, for all $\theta \,\in\, \O_\d$, 
\begin{equation}\label{symplectic-basis-theta-sp-n-R}
 (v_j^\theta, v^\theta_{ t_\theta/2 + j})_\theta = 1 \text{ for } 1 \le j \le t_\theta/2 \text{ and }  (v_j^\theta, v^\theta_{i})_\theta = 0 \text{ for all } i \neq j + t_\theta/2.
\end{equation}
Now fixing $\theta \,\in\, \O^1_\d $, for all $ 1\,\leq\, j \,\leq\, t_\theta$, define 
\begin{equation}\label{orthogonal-basis-V-theta-sp-n-R}
{ w}^{\theta}_{jl}:=
\begin{cases}
         \big(X^l v^{\theta}_j + X^{\theta-1-l} v^{\theta}_j \big)/{\sqrt{2}}
          &  \text{ if } l \text{ is  even and } 0 \leq l <  (\theta-1)/2 \vspace{0.1cm}
            \\
            \big(X^l v^{\theta}_j + X^{\theta-1-l} v^{\theta}_j \big)/{\sqrt{-2}}
            &  \text{ if } l \text{ is  odd and } 0 \leq l < (\theta-1)/2 \vspace{0.1cm}
              \\
           X^l v^{\theta}_j   & \text{ if } l =  (\theta-1)/2 \vspace{0.1cm}
            \\
              \big(X^{\theta-1-l} v^{\theta}_j - X^l v^{\theta}_j \big) /{\sqrt{-2}}    & \text{ if } l \text{ is  even and } (\theta+1)/2 \leq l \leq (\theta-1) \vspace{0.1cm}
                  \\
             \big(X^{\theta-1-l} v^{\theta}_j - X^l v^{\theta}_j \big)/{\sqrt{2}} & \text{ if } l \text{ is  odd and } (\theta+1)/2 \leq l \leq (\theta-1) . 
\end{cases} 
\end{equation}

For $\zeta \,\in\, \O^3_\d $, for all $ 1\,\leq\, j \,\leq\, t_\zeta$, define 
\begin{equation}\label{orthogonal-basis-V-zeta-sp-n-R}
{ w}^{\zeta}_{jl}:=
\begin{cases}
         \big(X^l v^{\zeta}_j + X^{\zeta-1-l} v^{\zeta}_j \big)/{\sqrt{2}}
          &  \text{ if } l \text{ is  even and } 0 \leq l <  (\zeta-1)/2 \vspace{0.1cm}
            \\
            \big(X^l v^{\zeta}_j + X^{\zeta-1-l} v^{\zeta}_j \big)/{\sqrt{-2}}
            &  \text{ if } l \text{ is  odd and } 0 \leq l < (\zeta-1)/2 \vspace{0.1cm}
              \\
           X^l v^{\zeta}_j \sqrt{-1}  & \text{ if } l =  (\zeta-1)/2 \vspace{0.1cm}
            \\
              \big(X^{\zeta-1-l} v^{\zeta}_j - X^l v^{\zeta}_j \big) /{\sqrt{-2}}    & \text{ if } l \text{ is  even and } (\zeta+1)/2 \leq l \leq (\zeta-1) \vspace{0.1cm}
                  \\
             \big(X^{\zeta-1-l} v^{\zeta}_j - X^l v^{\zeta}_j \big)/{\sqrt{2}} & \text{ if } l \text{ is  odd and } (\zeta+1)/2 \leq l \leq (\zeta-1) . 
\end{cases} 
\end{equation}

For $\theta \,\in \,\O_\d$, $0 \,\le\, l \,\le\, \theta-1$, set 
\begin{equation}\label{defn-V-theta-sp-n-R}
 V^l(\theta)\,:=\, \text{Span}_\R\{ w^{\theta}_{jl} \,\mid\, 1\,\le\, j \,\le\, t_\theta \}\, .
\end{equation}
Using \eqref{symplectic-basis-theta-sp-n-R} we 
observe that for each $\theta \,\in \,\O_\d$ the space $M (\theta-1)$ is a direct sum of the 
subspaces $V^l (\theta)$,\, $ 0 \,\leq\, l\, \leq\, \theta-1$, which are mutually orthogonal with respect to $\<>$. For
$\theta \,\in\, \O_\d$, define
$$
\CC^l(\theta) \,:= \,\big( {w}^{\theta}_{_{1 \, l}},   \, \dotsc \, ,   {w}^{\theta}_{_{t_{\theta}/2 \, l}}  \big) \vee  \big(  {w}^{\theta}_{_{(t_{\theta}/2 +1)\, l}},  \, \dotsc \, ,   {w}^{\theta}_{_{t_{\theta} \, l}}  \big)\,.
$$
Then using \eqref{symplectic-basis-theta-sp-n-R}, \eqref{orthogonal-basis-V-theta-sp-n-R} and \eqref{orthogonal-basis-V-zeta-sp-n-R}  it follows that $\CC^l(\theta)$ is a symplectic basis for $V^l(\theta)$.  Recall that $$\BC^l(d)\,=\, (X^lv^d_1,\, \dots,\,   X^lv^d_{t_d})$$
is the ordered basis of $X^l L(d-1)$ for $0\,\leq\, l \,\leq\, d -1,\, d  \,\in\, \N_\d$ as in \eqref{old-ordered-basis-part}.

\begin{lemma}\label{Z-sp-nC-XHY}
 Let $ X $ be a nilpotent element in $ \s\p(n,\C) $ and $ \{X,\,H,\,Y\} $ be a $ \s\l_2 (\R)$-triple  in $ \s\p(n,\C) $   containing $X$.
Then the  following holds:
$$
 \ZC_{ {\rm Sp} (n,\C) }(X,\,H,\,Y)\! =\! \left\{ g \in { {\rm Sp} (n,\C)} \middle\vert \! 
                                        \begin{array}{ccc}
                             g (V^l (\theta)) \subset \ V^l (\theta) \  \text{and }    \vspace{.14cm}\\
                             \big[g|_{ V^l(\theta)}\big]_{{\CC}^l(\theta)} = \big[g |_{ V^0 (\theta)}\big]_{{\CC}^0 (\theta)} \text{for all } \theta \in \O_\d, 0 \leq l < \theta ~; \vspace{.14cm}\\
                               g(X^l L(\eta-1)) \subset  X^l L(\eta-1)  \text{ and }    \vspace{.14cm}\\
                          \!   \big[g |_{ X^l L(\eta-1)} \big]_{{\BC}^l (\eta)}\! \!=\! \big[g |_{  L(\eta-1)}\big]_{{\BC}^0 (\eta)}  \text{for all } \eta \in \E_\d,   0 \leq l< \eta \!
                                       \end{array}
                                           \right\}\!.
                                           $$                                           
\end{lemma}

\begin{proof} 
The proof is similar to that of Lemma \ref{Z-so-nC-XHY}.
\end{proof}

\begin{remark}\label{structure-of-centralizer-spnc}
{\rm 
	We follow the notation as in Lemma \ref{Z-sp-nC-XHY} .
	Let $g \,\in\, \ZC_{ {\rm Sp} (n, \C)}(X,\, H,\, Y)$.  Let $\theta \,\in\, \O_\d$ and $\eta \,\in\, \E_\d$.   
	Then it follows from Lemma \ref{Z-sp-nC-XHY}  that $ g$ keeps the subspaces   $ V^0(\theta) $ and $ L(\eta -1) $  invariant. 
	Since the restriction of $  \<> $ is a symplectic form on $ V^0(\theta) $ we have $g|_{ V^0(\theta)} \,\in\,
	{\rm Sp}(V^0(\theta), \<>)$. 
	Further recall that the form $ (\cdot,\cdot)_\eta $, as defined in \eqref{new-form}, is symmetric  on $L(\eta-1)$, and $(gx,gy)_\eta = (x,y)_\eta$ for all $x,y\in L(\eta-1)$, see \cite[Remark A.7]{BCM}. 
	Thus $g |_{  L(\eta-1)} \,\in\, {\rm  O}(L(\eta-1),\, (\cdot,\, \cdot)_\eta)$.
}
\end{remark}

For $\eta \,\in\, \E_\d$, $0\,\le\, l \,\le\, \eta/2 -1$, set
\begin{equation}\label{defn-W-eta-sp-n-R}
  W^l(\eta)\,:=\,  X^l L (\eta-1) + X^{\eta-1-l} L(\eta-1)\, .
\end{equation} 
We re-arrange the ordered basis $\BC^l (\eta) \vee \BC^{\eta-1-l} (\eta)$ of
$W^l(\eta)$ as follows:
\begin{equation}\label{symplectic-basis-D-eta}
\DC^l(\eta):= \begin{cases}
               \big(X^l v_1,\, \dotsc \,, X^l v_{t_\eta}\big) \vee \big(X^{\eta-1-l} v_{1},\, \dotsc \, , X^{\eta-1-l} v_{t_\eta}\big)   & \text{ if }l \text{  is even} \vspace{.25cm}\\
             
             \big(X^{\eta-1-l} v_{1}, \, \dotsc \,, X^{\eta-1-l} v_{t_\eta}\big)  \vee \big(X^l v_{1}, \, \dotsc \,, X^l v_{t_\eta}\big) & \text{ if } l \text{ is odd}.               
 \end{cases}
 \end{equation}
Using \eqref{onb-eta-sp-nc} it can be easily verified that $\DC^l(\eta)$ is a symplectic basis with respect to $\<>$.
 Let $J_{\CC^l(\theta)}$ be the complex structure on $V^l(\theta)$ associated to the basis 
$\CC^l(\theta)$ for $\theta \,\in\, \O_\d$,\, $0 \,\leq\, l \,\leq\, \theta-1$, and let $J_{\DC^l(\eta)}$ be 
the complex structure on $W^l(\eta)$ associated to the basis $\DC^l(\eta)$ for $\eta \,\in\, \E_\d, 
0 \,\leq\, l \,\leq\, \eta-1$.

In the next lemma we specify a maximal compact subgroup $\ZC_{ {\rm Sp} (n,\C) } (X,\,H,\,Y)$ which will be used in Theorem 
\ref{max-cpt-sp-nC-wrt-basis}.  Recall that the $ \R$-algebra embedding $ \wp_{m, \H}$ is defined in \eqref{embedding-H2C}.
 
\begin{lemma}\label{max-cpt-reductive-part-sp-nC}
Let $K$ be the subgroup of $\ZC_{ {\rm Sp} (n,\C) } (X,\,H,\,Y)$ consisting of elements $g$ in \\
$\ZC_{ {\rm Sp} (n,\C) } (X,\,H,\,Y)$ satisfying the following conditions:
 \begin{enumerate}
\item For all  $\theta \,\in\, \O_\d $ and $0 \,\leq\, l \,\leq\, \theta -1$ the inclusion   $g ( V^l (\theta))\,\subset\,V^l (\theta)$ holds.
\item For all $ \theta \,\in \,\O_\d$, there exist  $A_\theta,\, B_\theta \,\in\,
{\rm M}_{t_\theta/2}(\C)$ with $A_\theta + \jb B_\theta \in {\rm Sp}(t_\theta/2) $ such that 
  $$
  \big[g|_{V^l(\theta)}\big]_{\CC^l(\theta)}\, =\,
    \wp_{t_\theta/2,\H}( A_\theta + \jb B_\theta)    $$ 
\item  For all $\eta \,\in\, \E_\d$ and $0 \,\leq\, l \,\leq\, \eta -1$ the inclusion $g ( X^l L(\eta-1)) \,\subset\, X^l L(\eta-1)$ holds.
\item  For all $\eta \,\in\, \E_\d$, there exist $C_\eta \,\in\, {\rm O}_{t_\eta}$ such that  
       $ \big[ g|_{X^l L(\eta-1)} \big]_{\BC^l(\eta)} \,=\, C_\eta$.
\end{enumerate}
Then $K$ is a maximal compact subgroup of $\ZC_{ {\rm Sp} (n,\C)}(X,\,H,\,Y)$.
\end{lemma}

\begin{proof}
Let $K'\, \subset\, \ZC_{ {\rm Sp} (n,\C) } (X,\,H,\,Y)$ be the subgroup consisting of all elements $g$ satisfying the conditions 
\eqref{commute-J-odd-spnc}, \eqref{g-on-orthogonal-even-subspace} and \eqref{commute-J-even-spnc} below :
\begin{align}
	& \overline{g}|_{ V^0(\theta)} J_{{\CC}^0 (\theta)} \, =\, J_{{\CC}^0 (\theta)} {g}|_{ V^0(\theta)}\,   , \hspace{1.5cm} \text{ for all } \theta \in \O_\d\,; \label{commute-J-odd-spnc} \\
&	g(W^l(\eta)) \subset W^l(\eta)  \,\text{ and } \,
	  \big[g |_{ W^l(\eta)} \big]_{{\DC}^l (\eta)}\! =  \big[g |_{ W^0(\eta)}\big]_{{\DC}^0 (\eta)} ,\, \text{ for all } \eta \in \E_\d,  
	0 \leq l< \eta \,;  \vspace{.3cm} \label{g-on-orthogonal-even-subspace}     \\  
	&  \overline{g}|_{ W^0(\eta)} J_{{\DC}^0 (\eta)} \, = \, J_{{\DC}^0 (\eta)} g|_{ W^0(\eta)} \,   , \hspace{1.5cm} \text{ for all } \eta \in \E_\d \, . \label{commute-J-even-spnc} 
\end{align}
Using Lemma \ref{Z-sp-nC-XHY} and Lemma \ref{max-cpt-symplectic-gps}(1) it is clear that
$K'$ is a maximal compact subgroup of $\ZC_{ {\rm Sp} (n,\C) }(X,\,H,\,Y)$. 
Hence to prove the lemma it suffices to show that $K \,= \,K'$. 
Let $g \,\in\,  {\rm Sp} (n,\C)$. Using Lemma \ref{max-cpt-symplectic-gps} (2) it is straightforward to check that $g$ satisfies
(1), (2) in the statement of the lemma if and only if $g$ satisfies \eqref{commute-J-odd-spnc}. Let $g\,\in\, \ZC_{ {\rm Sp} (n,\C)}(X,\,H,
\,Y)$ and  $$C_\eta\,:=\, \big[ g|_{ L(\eta-1)} \big]_{\BC^0(\eta)}\,=\, \big[ g|_{ X^lL(\eta-1)} \big]_{\BC^l(\eta)}.$$
Then we observe that $$[g|_{W^l(\eta)}]_{\DC^l(\eta)} \,=\, \begin{pmatrix}
 C_\eta\\
 & C_\eta
\end{pmatrix}.$$
Now suppose that $g \,\in\, {\rm Sp} (n,\C)$ and $g$ satisfies (3) and (4) in the statement
of the lemma. It is clear that \eqref{g-on-orthogonal-even-subspace} holds. Since $C_\eta
\,\in\, {\rm O}_{t_\eta}$, it follows that \eqref{commute-J-even-spnc} holds. 

Next assume that $g \,\in\, \ZC_{ {\rm Sp} (n,\C) } (X,\,H,\,Y)$ and satisfies  \eqref{g-on-orthogonal-even-subspace} and
\eqref{commute-J-even-spnc}. From \eqref{commute-J-even-spnc} it follows that $C_\eta\, =\, \overline{C_\eta}$.
Using Lemma \ref{max-cpt-symplectic-gps}(2), we have $C_\eta + \jb 0 \in {\rm Sp}(t_\eta)$. Thus $C_\eta\,\in\, {\rm O}_{t_\eta}$.
\end{proof}

We next introduce some notation which will be needed in Theorem \ref{max-cpt-sp-nC-wrt-basis}. Recall that the positive parts of the 
symplectic bases $\DC(\eta)$ and $\CC(\theta)$ are denoted by $\DC_+(\eta)$ and $\CC_+(\theta)$, respectively. Similarly the negative 
parts of $\DC(\eta)$ and $\CC(\theta)$ are denoted by $\DC_-(\eta)$ and $\CC_-(\theta)$, respectively.  For $\eta \,\in\, \E_\d$, set
$$
\DC_+ (\eta) \,:=\, \DC^0_+ (\eta) \vee \cdots \vee  \DC^{\eta/2-1}_+ (\eta) \ \text{ and } \ 
\DC_- (\eta) \,:=\, \DC^0_- (\eta) \vee \cdots \vee  \DC^{\eta/2 -1}_- (\eta).
$$
For $\theta \,\in\, \O_\d$, set
$$
\CC_+ (\theta) \,:=\, \CC^0_+ (\theta) \vee \cdots \vee  \CC^{\theta-1}_+ (\theta) \ \text{ and } \ 
\CC_- (\theta) \,:=\, \CC^0_- (\theta) \vee \cdots \vee  \CC^{\theta-1}_- (\theta).
$$
Let $\alpha\, :=\, \# \E_\d $, $\beta \,:=\, \# \O_\d$.
We enumerate $\E_\d \,=\,\{ \eta_i \,\mid\, 1 \,\leq\, i \,\leq\, \alpha \}$ such that
$\eta_i \,< \,\eta_{i+1}$, and 
$\O_\d \,=\,\{ \theta_j \,\mid\, 1 \leq j \leq \beta \}$ such that
$\theta_j \, < \,\theta_{j+1}$.
Now define
$$
\EC_+ \,:=\, \DC_+ (\eta_1) \vee \cdots \vee \DC_+ (\eta_{\alpha})\, ; \qquad
\OC_+ \,:=\, \CC_+ (\theta_1) \vee \cdots \vee \CC_+ (\theta_{\beta})\, .
$$
$$
\EC_-  \, := \, \DC_- (\eta_1) \vee \cdots \vee \DC_- (\eta_{\alpha}) ;  \text{ and } 
\OC_-\, :=\, \CC_-(\theta_1) \vee \cdots \vee \CC_-(\theta_{\beta}).
$$
Also we define
\begin{equation}\label{symplectic-basis-final}
\HC_+ \,:=\, \EC_+ \vee \OC_+ \, , \ \ \HC_- := \EC_- \vee \OC_-  \ \text{ and } \,
\HC \,:=\, \HC_+ \vee \HC_-. 
\end{equation}

As before, for a matrix $A\,=\, (a_{ij}) \,\in\, {\rm M}_r(\C)$, define $\overline{A}
\,:=\, (\overline{a}_{ij}) \,\in\, {\rm M}_r(\C)$.
Let
\begin{align}\label{map-D-SpnC}
\Db_{{\rm Sp}(n,\C)}\,\colon\, \prod_{i=1}^\alpha {\rm M}_{t_{\eta_i}} (\R)  \times \prod_{j=1}^\beta {\rm M}_{t_{\theta_j}/2} (\H)
\,\longrightarrow\, {\rm M}_n(\H)  
\end{align}
be the $\R$-algebra embedding defined by
\begin{align*}
\big(C_{\eta_1}, \,  \dotsc \, , C_{\eta_{\alpha}}\, ;\,  A_{\theta_1}\, , \dotsc\, , 
A_{\theta_\beta} \big) \, 
     \longmapsto \,  \bigoplus_{i=1}^\alpha  \big( C_{\eta_i}  \big)_\blacktriangle^{\eta_i /2} \oplus \bigoplus_{j=1}^\beta  \big( A_{\theta_j} \big)_\blacktriangle ^{\theta_j} \, .
\end{align*}

It is clear that the basis $\HC$ in \eqref{symplectic-basis-final} is a symplectic basis of $V$ with
respect to $\<>$. Let $\widetilde{\Lambda}_\HC \,\colon\, \{ x \,\in\, {\rm End}_\C \C^{2n}
\,\mid\, x J_\HC \,=\,J_\HC x  \} \,\longrightarrow\, {\rm M}_{n} (\H)$ be the isomorphism of
$\R$-algebras induced by the above symplectic basis $\HC$. Recall that $\widetilde{\Lambda}_\HC \,:\,
K_\HC \,\longrightarrow\, {\rm Sp} (n)$ is an isomorphism of Lie groups. 
 
\begin{theorem}\label{max-cpt-sp-nC-wrt-basis} 
Let $X \,\in\, \NC _{\s\p(n,\C)}$, $ X\,\ne\, 0$, and  $\Psi_{{\rm Sp} (n,\C)}(\OC_X) \,=\, \d$.
Let $\alpha \,:=\, \# \E_\d $ and $\beta \,:=\, \# \O_\d$. Let
$\{X,\,H,\,Y\}$ be a $\s\l_2(\R)$-triple in $\s\p (n,\C)$.
Let $K$ be the maximal compact subgroup of $\ZC_{{\rm Sp} (n,\C)} (X,\,H,\, Y)$ as in Lemma
\ref{max-cpt-reductive-part-sp-nC}, and 
the map $\Db_{{\rm Sp}(n,\C)}$ be defined as in \eqref{map-D-SpnC}. 
Then $\widetilde{\Lambda}_\HC(K) \,\subset\, {\rm Sp}(n)$ is given by
$$
\widetilde{\Lambda}_\HC(K)\,= \,\Big\{ \Db_{{\rm Sp}(n,\C)}(g) \, \, \bigm| \,\,  
                        g \in   \prod_{i=1}^\alpha  {\rm O}_{t_{\eta_i}}  \times \prod_{j=1}^\beta  {\rm Sp}(t_{\theta_j}/2)   \Big\}.
$$
Moreover, the nilpotent orbit $ \OC_X $ in $ \s\p(n,\C) $ is homotopic to $ {\rm Sp}(n) / \widetilde{ \Lambda}_\HC({K})$.
\end{theorem}

\begin{proof}
This follows by writing the matrices of the elements of the maximal compact subgroup $ K $ in Lemma
\ref{max-cpt-reductive-part-sp-nC} with respect to the ordered basis $\HC$ as in \eqref{symplectic-basis-final}.

The second part follows from Theorem \ref{mostow-corollary} and the well-known fact that any maximal compact subgroup of
${\rm Sp}(n,\C)$ is isomorphic to ${\rm Sp}(n)$.	
\end{proof}

\subsection{Homotopy types of the nilpotent orbits in \texorpdfstring{${\s\p}(p,q)$}{Lg}}\label{sec-sp-pq}

Let $n$ be a positive integer, and let $(p,\,q)$ be a pair of non-negative integers with $p + q \,=\,n$. In this 
subsection we write down the homotopy types of the nilpotent orbits in ${\s\p}(p,q)$ as compact homogeneous spaces. 
As we do not need to deal with compact groups, we will further assume that $p \,>\,0$ and $q\,>\,0$. Throughout 
this subsection $\<>$ denotes the Hermitian form on $\H^n$ defined by $\langle x,\, y \rangle \,:=\, 
\overline{x}^t{\rm I}_{p,q} y$, $x,\,y\, \in\, \H^n$, where ${\rm I}_{p,q}$ is as in \eqref{defn-I-pq-J-n}. We will 
follow notation as defined in \S \ref{sec-notation}.

First we will recall a parametrization of nilpotent orbits in $\s\p(p,q)$, see \cite[Section 4.8]{BCM}.
Let $$\Psi_{{\rm SL}_n (\H)} \,:\, \NC ({\rm SL}_n (\H)) \,\longrightarrow\, \PC (n)$$ be the
parametrization as in Theorem \ref{sl-H-parametrization}.
As ${\rm Sp} (p,q) \,\subset\, {\rm SL}_n (\H)$ (consequently, $\NC_{{\s\p}(p,q)} \,\subset\,
\NC_{\s\l_n(\H)}$)
we have the inclusion map $\Theta_{{\rm Sp}(p,q)} \,: \,\NC ({\rm Sp} (p,q)) \,
\longrightarrow\,\NC ( {\rm SL}_n (\H) )$. Let
$$\Psi'_{{\rm Sp}(p,q)}\,:=\, \Psi_{{\rm SL}_n (\H)} \circ \Theta_{{\rm Sp}(p,q)} \,
\colon\, \NC ({\rm Sp} (p,q))  \,\longrightarrow\, \PC (n)$$ be the composition.
Let $0\,\not=\, X \,\in\, {\s\p}(p,q)$ be a nilpotent element and $\OC_X$ be the corresponding nilpotent orbit in
${\s\p}(p,q)$. Let $\{X,\, H,\, Y\} \,\subset\, {\s\p}(p,q)$ be a $\s\l_2(\R)$-triple. We now use \cite[Proposition A.6, Remark A.8(3)]{BCM}.
Let $V := \H^n$ be the right $\H$-vector space of column vectors.
Let $\{d_1,\, \cdots,\, d_s\}$, with $d_1 \,<\, \cdots \,<\, d_s$, be a ordered finite subset of natural numbers that arise as $\R$-dimension of non-zero irreducible  $\text{Span}_\R \{ X,\,H,\,Y\}$-submodules of $V$.
Recall that $M(d-1)$ is defined to be the isotypical component of $V$ containing all irreducible $\text{Span}_\R \{ X,\,H,\,Y\}$-submodules of $V$ with highest weight $(d-1)$, and as in \eqref{definition-L-d-1}, we set $L(d-1)\,:= \,V_{Y,0} \cap M(d-1)$. Recall that the space $L(d_r-1)$ is a $\H$-subspace for $1\leq r \leq s$.
Let $t_{d_r} \,:=\, \dim_\H L(d_r-1)$ for 
$1\,\leq\, r \,\leq\, s$. Then 
$\d\,:=\, [d_1^{t_{d_1}},\, \cdots ,\,d_s^{t_{d_s}}]\,\in\, \PC(n)$, and  moreover,
$\Psi'_{{\rm Sp}(p,q)} (\OC_X) \,=\, \d$.

We next assign $\sgn_{\OC_X} \,\in\, \SC^{\rm even}_{\d}(p,q)$ to each $\OC_X \,\in\, \NC({\rm Sp}(p,q))$; see 
\eqref{S-d-pq-even} for the definition of $\SC^{\rm even}_{\d}(p,q)$.
For each $d \,\in\, \N_\d$ (see \eqref{Nd-Ed-Od} for the definition of $\N_\d$)  we will define a $t_d \times d$ matrix $(m^d_{ij})$ in $\Ab_d$ 
which depends only on the orbit $\OC_X$ containing $X$; see \eqref{A-d} for the definition of $\Ab_d$.
For this, recall that the form
$(\cdot,\,\cdot)_{d} \,\colon\, L(d-1) \times L(d-1) \,\longrightarrow\,
\H$ defined as in \eqref{new-form} is Hermitian or skew-Hermitian
according as $d$ is odd or even.
Denoting the signature of  $(\cdot,\, \cdot)_{\theta}$ by $(p_{\theta}, \,q_{\theta})$ 
when $\theta\,\in\, \O_\d$, we now define 
\begin{align*}
 m^\eta_{i1} & := +1 \qquad   \text{if } \  1 \leq i \leq t_{\eta}, \quad \eta \in \E_\d\,; \\
 m^\theta_{i1} &:= \begin{cases}
                   +1  & \text{ if } \,  1 \leq i \leq p_{\theta} \\
                   -1  & \text{ if } \, p_\theta < i \leq t_\theta  
                 \end{cases} ,\, \theta \in \O_\d\,;
               \end{align*}
and for $j \,>\,1$, define $(m^d_{ij})$ as in \eqref{def-sign-alternate} and
\eqref{def-sign-alternate-1}. 
Then the matrices $(m^d_{ij})$ clearly verify \eqref{yd-def2}. Set $\sgn_{\OC_X} \,:=\,
((m^{d_1}_{ij}),\, \cdots ,\, (m^{d_s}_{ij}))$.
It now follows from the last paragraph of \cite[Remark 5.21]{BCM}
that $\sgn_{\OC_X}\,\in\, \SC^{\rm even}_\d(p,q)$. Thus we have the map
$$
\Psi_{{\rm Sp} (p,q)}\,:\,\NC({\rm Sp}(p,q)) \,\longrightarrow\, \YC^{\rm even}(p,q)\, ,\ \
\OC_X \,\longmapsto\, \big(\Psi'_{{\rm Sp} (p,q)} (\OC_X),\, \sgn_{\OC_X}  \big)\, ;
$$
where $\YC^{\rm even}(p,q)$ is as in \eqref{yd-even-Y-pq}. The following theorem is standard, see \cite[Theorem 9.3.5]{CoM}, 
\cite[Theorem 4.38]{BCM}.
 
\begin{theorem}\label{sp-pq-parametrization}
The above map $\Psi_{{\rm Sp}(p,q)}$ is a bijection.
\end{theorem}

Let $0\,\not=\, X \,\in\, \NC _{\s\p (p,q)}$ and $\{X,\,H,\,Y\}$ a $\s\l_2(\R)$-triple in $\s\p (p,q)$. Let $\Psi_{{\rm Sp}(p,q)} (\OC_X) \,=\, \big( \d ,\, \sgn_{\OC_X} \big)$. Then $\Psi'_{{\rm Sp}(p,q)} (\OC_X) \,=\, \d$. 
Recall that $\sgn_{\OC_X}$ determines the signature of $(\cdot,\, \cdot)_\theta$ on $L(\theta -1)$ for all $\theta \,\in\, \O_\d$; let $(p_\theta, \,q_\theta)$ be the signature of $(\cdot,\, \cdot)_\theta$ on $L(\theta-1)$. Let $( v^{d}_1,\, \dots,\, v^d_{t_d} )$ be an ordered $\H$-basis of $L(d-1)$ as in \cite[Proposition A.6]{BCM}. It now follows from Proposition \cite[Proposition A.6 (3)(a)]{BCM} that $( v^d_1, \,\dots,\, v^d_{t_d})$ is an orthogonal basis of $L(d-1)$ for the form $(\cdot,\, \cdot)_d$ for all $d \,\in\, \N_\d$.
Since $(\cdot,\, \cdot)_\eta$ is skew-Hermitian for all $\eta \,\in \,\E_\d$, 
we may assume that for all $\eta \,\in\, \E_\d$ the orthogonal basis $( v^\eta_1,\, \dots,\, v^\eta_{t_\eta} )$ satisfies the following relations :  
\begin{equation}\label{orthonormal-basis-eta-sp-pq}
( v^{\eta}_j, v^\eta_j)_{\eta} = \jb \quad \text{ for all } \, 1 \leq j \leq t_\eta, \, \eta \in \E_\d.
\end{equation}
In view of the signature of $(\cdot,\, \cdot)_\theta$ we may also assume that
\begin{equation}\label{orthonormal-basis-theta-sp-pq}
(v^{\theta}_j, v^\theta_j)_{\theta}  =
 \begin{cases}
  +1  &  \text{ if } \;  1 \leq j \leq p_{\theta}   \\
  -1  &  \text{ if } \;  p_{\theta} < j \leq t_{\theta} 
 \end{cases}\,; \,\theta \,\in\, \O_\d\,.
\end{equation}
As a particular case of a more general construction given in \cite[Lemma A.11(1)]{BCM},  for $\eta\,\in\, \E_\d$ we define 
$$
   {w}_{jl}^{\eta}\,:=\,
\begin{cases}
 \big( {X^l v^{\eta}_j } \,+ \, X^{\eta-1-l} v^{\eta}_j \jb \big)/{\sqrt{2}}   & \text{ if } \ 0 \,\leq\, l \,<\, \eta/2   \vspace*{.2cm} \\
 \big( X^{\eta-1-l} v^{\eta}_j \, -\, X^l v^{\eta}_j \jb \big)/{\sqrt{2}}  & \text{ if }  \ \eta/2 \,\leq\, l\,\leq\, \eta -1\,.
\end{cases}
$$

Similarly, as in \cite[Lemma A.11(2)]{BCM}, for $\theta \,\in\, \O_\d$ we define,
$$
{  w}^{\theta}_{jl}\,:=\,
\begin{cases}
\big( X^l v^{\theta}_j \, + \, X^{\theta-1-l} v^{\theta}_j \big)/{\sqrt{2}} & \text{ if }\ 0 \,\leq\, l \,<\, (\theta-1)/2 \vspace*{.1cm}\\
  X^l v^{\theta}_j    & \text{ if }\  l \,= \,(\theta-1)/2 \vspace*{.1cm}\\
 \big( X^{\theta-1-l} v^{\theta}_j \, - \, X^l v^{\theta}_j \big)/{\sqrt{2}} &  \text{ if }\ (\theta-1)/2 \,<\, l \,\leq\, \theta -1.
\end{cases} 
$$

Let $\{w^{d}_{jl} \mid 1 \leq j \leq t_d, \, 0 \leq l \leq d-1 \}$ be the $\H$-basis of $M(d-1)$ constructed as above. For each $d \in \N_\d, \, 0 \leq l \leq d-1$ we set 
 $$
 V^l(d)\, := \, \text{Span}_\H \{w^d_{1l},\, \dotsc\,, w^d_{t_d l} \}. 
 $$
  The ordered basis $( w^d_{1l},\, \dots \,, w^d_{t_d l} )$ of $V^l(d)$ is denoted by $\CC^l(d) := (w^d_{1l},\, \dotsc\,, w^d_{t_d l} )$. Set
$$
V^l_+ (d) : = \text{ Span}_\H \{ v \mid  v \in  \CC^l (d),\langle v, v \rangle > 0 \}, \quad
V^l_- (d) : = \text{ Span}_\H \{ v \mid  v \in  \CC^l (d),\langle v, v \rangle < 0 \}.
$$
Now it is clear that for $\eta \in \E_\d$,
$$
V^l_+(\eta) \,: =\begin{cases}
                       V^l(\eta) &  \text{ if $l$ odd}\\
                       \phi     &  \text{ if $l$ even}
                      \end{cases}
\quad  \text{ and } \quad V^l_-(\eta) : =\begin{cases}
                          V^l(\eta) &  \text{ if $l$ even}\\
                           \phi     &  \text{ if $l$ odd} .
                           \end{cases}
$$
We next impose orderings on the sets $\{v \,\in\, \CC^l (d)\,\, \mid\,\, \langle v,\, v \rangle \,>\, 0 \}$,
$\{ v\,\in\, \CC^l (d)\,\mid\, \langle v, \,v \rangle \,<\, 0 \}$.
Define the ordered sets by 
$\CC^l_+ (\theta)$, $\CC^l_- (\theta)$, $\CC^l_+ (\zeta)$ and $\CC^l_- (\zeta)$ as in \cite[(4.19), (4.20), (4.21), (4.22)]{BCM}, respectively according as  $\theta \in \O^1_\d$ or $\zeta \in \O^3_\d$. 
Set 
$$
\CC^l_+(\eta) \,:=\,\begin{cases}
                       \CC^l(\eta) &  \text{ if $l$ odd}\\
                       \phi     &  \text{ if $l$ even}
                      \end{cases}
\quad  \text{ and } \quad \CC^l_-(\eta) : =\begin{cases}
                       \CC^l(\eta) &  \text{ if $l$ even}\\
                       \phi     &  \text{ if $l$ odd}\,.
                      \end{cases} 
$$
It is straightforward that $\CC^l_+ (d)$ and $ \CC^l_- (d)$ are indeed ordered bases of $V^l_+ (d)$ and $V^l_- (d)$, respectively.

Next we will write down  a suitable description of  reductive part of the centralizer of a nilpotent element in ${\s\p}(p,q) $. 

\begin{lemma}\label{reductive-part-comp-sp-pq}
Let $X$ be a nilpotent element in $\s\p(p,q)$ and $\{X,\,H,\,Y\}$ be a $\s\l_2(\R)$-triple   in $\s\p(p,q)$   containing $X$.
Then the following holds:
$$
\ZC_{ {\rm Sp} (p,q) }(X,H,Y)\! =\! \left\{ g \in {\rm Sp} (p,q) \middle\vert 
\begin{array}{ccc}
g (V^l (d)) \subset \ V^l (d) \  \text{and }    \vspace{.14cm}\\
\!\!	\big[g|_{ V^l(d)}\big]_{{\CC}^l(d)}\! = \big[g |_{ V^0 (d)}\big]_{{\CC}^0 (d)}  \text{ for  all } d \in \N_\d, 0 \leq l <d \!
\end{array}
\!\!   \right\}\!.
$$  
\end{lemma}

\begin{proof}
We omit the proof as it is similar to that of Lemma \ref{Z-so-nC-XHY}.
\end{proof}

\begin{remark}\label{structure-of-centralizer-sppq}
{\rm 
We follow the notation as in Lemma \ref{reductive-part-comp-sp-pq}. Let $g \,\in\, \ZC_{ {\rm Sp} (p,q)}(X,\, H,\, Y)$.  Let $\theta \,\in\, \O_\d$ and $\eta \,\in\, \E_\d$.   
Then it follows from Lemma \ref{reductive-part-comp-sp-pq} that $ g$ keeps the subspace $ V^0(\theta)$ invariant. 
Since the restriction of $\<> $ is a hermitian form on $ V^0(\theta) $ we have $g|_{ V^0(\theta)} \,\in\, {\rm SU}(V^0(\theta),\, \<>)$. 
Further recall that the form $ (\cdot,\,\cdot)_\eta $, as defined in \eqref{new-form}, is skew-hermitian  on $L(\eta-1)$. Also
$g$ keeps the subspace $ L(\eta -1)$ invariant  and  $(gx,gy)_\eta \,=\, (x,\,y)_\eta$ for all $x,\,y\,\in\, L(\eta-1)$, see
\cite[Lemma 4.4(3) and Remark A.7]{BCM}. 
Thus $g |_{L(\eta-1)} \,\in\, {\rm  SO}^*(L(\eta-1),\, (\cdot,\, \cdot)_\eta)$.
}	
 \end{remark}

Let $W$ be a right $\H$-vector space, $\<>'$ be a non-degenerate skew-Hermitian form on $W$. Let $\dim_\H W \,=\, m$ and $\BC'
\,:=\, (v_1, \dotsc, v_m)$ be a standard orthogonal basis of $W$ such that $\langle v_r,\,v_r \rangle' \,=\, \jb$ for all
$1\,\leq\, r \,\leq\, m$.
Let ${\rm J}_{\BC'} \,\colon\, W \,\longrightarrow\, W$ be defined by ${\rm J}_{\BC'}(\sum_r v_r z_r)
\,:=\, \sum_r v_r \jb z_r$ for all column vectors $(z_1, \,\dotsc,\, z_m)^t \,\in\, \H^m$.
The next lemma is a standard fact where we recall an explicit description of maximal compact subgroups in the group ${\rm SO}^*(W, \<>')$. 
We set
\begin{align*}
K_{\BC'} :=& \big\{ g \in {\rm SO}^* (W, \<>') \bigm|   g {\rm J}_{\BC'} = {\rm J}_{\BC'} g  \big\}.
\end{align*}
\begin{lemma}[{\cite[Lemma 4.28]{BCM}}] \label{max-cpt-so*2n}
Let $W, \<>'$ and $\BC'$  be as above. Then
\begin{enumerate}
 \item $K_{\BC'}$ is a maximal compact subgroup in ${\rm SO}^*(W, \<>')$.
 
 \item $K_{\BC'} = \big\{g \in {\rm SL}(W) \bigm| [g]_{\BC'} = A+\jb B  \text{ where } A,B \in {\rm M}_m(\R), A+\sqrt{-1}B \in {\rm U}(m) \big\} $.
\end{enumerate}
\end{lemma} 
  
 Recall that $\big\{x \in {\rm End}_\H W \mid x {\rm J}_{\BC'} = {\rm J}_{\BC'} x \big\} = \big\{ x \in {\rm End}_\H W \bigm| [x]_{\BC'} \in {\rm M}_m(\R) + \jb {\rm M}_m(\R)  \big\} $. 
 We now consider the $\R$-algebra isomorphism 
 \begin{equation}\label{R-algebra-isomorphism-so*}
  \Lambda'_{\BC'}  \colon \big\{x \in {\rm End}_\H W \mid x {\rm J}_{\BC'} = {\rm J}_{\BC'} x \big\} \longrightarrow {\rm M}_m(\C)
 \end{equation}
 by $\Lambda'_{\BC'}(x):= A + \sqrt{-1} B $ where $A,B \in {\rm M}_m(\R)$ are the unique elements such that $[x]_\BC = A + \jb B$.
In view of the above lemma it is clear that $\Lambda'_{\BC'}(K_{\BC'})={\rm U}(m)$ and hence $\Lambda'_{\BC'} \,\colon\, K_{\BC'}
\,\longrightarrow\, {\rm U}(m)$ is an isomorphism of Lie groups.

Let $\widetilde{W}$ be a right $\H$-vector space, $\widetilde{\<>}$ be a non-degenerate 
Hermitian form on $\widetilde{W}$ with signature $(\widetilde{p} ,\, \widetilde{q})$. Let 
$\widetilde{\BC} \,:=\, (v_1,\, \dots,\, v_{\widetilde p}, \,v_{\widetilde p +1 }, \,\dots,\, 
v_{\widetilde{p}+ \widetilde{q}})$ be a standard orthogonal basis of $W$ such that $$
 \widetilde{\langle v_r,\,v_r \rangle} \,= \begin{cases}
                              +1 & \text{ if }  1\,\leq\, r \,\leq \,\widetilde{p}\\                                                                                                                                                   
                              -1 & \text{ if } \widetilde{p} +1 \,\leq\, r \,\leq\, \widetilde{p}+
\widetilde{q} .                                                                                                                                                \end{cases}
$$
Let $\widetilde{W}_+ \,:=\,{\rm Span}_\H\{v_1,\, \dots, \,v_{\widetilde p}\} $ and $\widetilde{W}_-
\,:=\,{\rm Span}_\H\{v_{\widetilde{p} +1 },\, \dots,\, v_{\widetilde{p}+ \widetilde {q}}\}$.
The next lemma is a standard fact where we recall, without a proof, an explicit description of
maximal compact subgroups in the group ${\rm SU}(\widetilde{W},\, \widetilde{\<>})$. 
$$
\widetilde{K}_{\widetilde \BC} : = \{g \in {\rm SU}(\widetilde{W}, \widetilde{\<>}) \mid g(\widetilde{W}_+) \subset \widetilde{W}_+ \text{ and } g(\widetilde{W}_-) \subset \widetilde{W}_-  \}
$$

\begin{lemma}
 Let $\widetilde{W}, \,\widetilde{\<>},\, \widetilde{\BC}$ be as above. Then $\widetilde{K}_{\widetilde \BC} $ is a maximal compact subgroup of ${\rm SU}(\widetilde{W}, \widetilde{\<>})$. 
\end{lemma}

In the next lemma we specify a maximal compact subgroup $\ZC_{ {\rm Sp} (p,q) } (X,\,H,\,Y)$ which will be used in Theorem 
\ref{max-cpt-sp-pq-wrt-onb}.

\begin{lemma}\label{max-cpt-sppq}
Let $K$ be the subgroup of $\ZC_{ {\rm Sp} (p,q) } (X,\,H,\,Y)$ consisting of all $g$ in \\
$\ZC_{{\rm Sp}(p,q)}(X,\,H,\,Y)$ satisfying the following conditions:

\begin{enumerate}
 \item $g (V^l_+ (d)) \,\subset\, \, V^l_+ (d)$ and $g ( V^l_- (d)  ) \,\subset \, V^l_-(d)$,  for all  $ d \,\in \,\N_\d \, 
\text{ and } \, 0 \,\leq\, l \,\leq\, d -1$.

\item For all $\eta \,\in\, \E_\d $, there exists $A_\eta, B_\eta \,\in\, {\rm M}_{t_\eta}(\R) $ with $A_\eta + \sqrt{-1} B_\eta \,\in
\, {\rm U}(t_\eta)$ such that
 $$\begin{array}{cc}
A_\eta + \jb B_\eta =  \Big[g |_{ V_-^0 (\eta) }\Big]_{{\CC}^0_- (\eta)} = \Big[g |_{  V^l_{(-1)^{l+1}} (\eta)}\Big]_{{\CC}^l_{(-1)^{l+1} } (\eta)}   
      \end{array}; \text{ for all }\, 0 \leq l \leq \eta-1.
 $$
  \item 
  For all $\theta \,\in\, \O^1_\d $, there exist $C_\theta \in {\rm Sp}(p_\theta), \, D_\theta \in {\rm Sp}(q_\theta)$ such that
$$
C_\theta = \Big[g |_{ V_+^0 (\theta) }\Big]_{{\CC}^0_+ (\theta)} = 
             \begin{cases}                                                            
                 \Big[g |_{  V^l_{(-1)^{l}} (\theta)}\Big]_{{\CC}^l_{(-1)^{l}} (\theta)} & \text{ for all } 0 \leq l < (\theta-1)/2 \vspace{.15cm}\\
            \Big[g |_{  V^{(\theta-1)/2}_{+} (\theta)}\Big]_{{\CC}^{(\theta-1)/2}_{+} (\theta)} \vspace{.15cm}  \\
                 \Big[g |_{  V^l_{(-1)^{l+1}} (\theta)}\Big]_{{\CC}^l_{(-1)^{l+1} } (\theta)} &  \text{ for all } (\theta-1)/2 < l \leq \theta-1,
                 \end{cases}                 
$$
$$
 D_\theta = \Big[g |_{ V_-^0 (\theta) }\Big]_{{\CC}^0_- (\theta)} =
     \begin{cases}
      \Big[g |_{  V^l_{(-1)^{l+1}} (\theta)}\Big]_{{\CC}^l_{(-1)^{l+1} } (\theta)} &  \text{ for all } 0 \leq l < (\theta-1)/2 \vspace{.15cm}\\
       \Big[g |_{  V^{(\theta-1)/2}_{-} (\theta)}\Big]_{{\CC}^{(\theta-1)/2}_{-} (\theta)} \vspace{.15cm}  \\
      \Big[g |_{  V^l_{(-1)^{l}} (\theta)}\Big]_{{\CC}^l_{(-1)^{l}} (\theta)} & \text{ for all } (\theta-1)/2 < l \leq \theta-1 .     
     \end{cases}
$$
\item For all $ \zeta \in \O^3_\d $, there exist $C_\zeta \in {\rm Sp}(p_\zeta), \, D_\zeta \in {\rm Sp}(q_\zeta)$ such that
$$
C_\zeta = 
\Big[g |_{ V_+^0 (\zeta) }\Big]_{{\CC}^0_+ (\zeta)} = 
             \begin{cases}                                                            
                 \Big[g |_{  V^l_{(-1)^{l}} (\zeta)}\Big]_{{\CC}^l_{(-1)^{l}} (\zeta)} &  \text{ for all } 0 \leq l < (\zeta-1)/2 \vspace{.15cm}\\            
      \Big[g |_{  V^{(\zeta-1)/2}_{-} (\zeta)}\Big]_{{\CC}^{(\zeta-1)/2}_{-} (\zeta)} \vspace{.15cm}  \\
                 \Big[g |_{  V^l_{(-1)^{l+1}} (\zeta)}\Big]_{{\CC}^l_{(-1)^{l+1} } (\zeta)} & \text{ for all } (\zeta-1)/2 < l \leq \zeta-1,
                 \end{cases}                 
$$
$$
D_\zeta =  \Big[g |_{ V_-^0 (\zeta) }\Big]_{{\CC}^0_- (\zeta)} =
     \begin{cases}
      \Big[g |_{  V^l_{(-1)^{l+1}} (\zeta)}\Big]_{{\CC}^l_{(-1)^{l+1} } (\zeta)} &  \text{ for all } 0 \leq l < (\zeta-1)/2 \vspace{.15cm}\\
      \Big[g |_{  V^{(\zeta-1)/2}_{+} (\zeta)}\Big]_{{\CC}^{(\zeta-1)/2}_{+} (\zeta)} \vspace{.15cm}  \\
            \Big[g |_{  V^l_{(-1)^{l}} (\zeta)}\Big]_{{\CC}^l_{(-1)^{l}} (\zeta)} &  \text{ for all } (\zeta-1)/2 < l \leq \zeta-1.    
     \end{cases}
$$ 
\end{enumerate}
Then $K$ is a maximal compact subgroup of $\ZC_{ {\rm Sp}(p,q)}(X,H,Y)$.
\end{lemma}

\begin{proof}
The proof of the lemma is similar to that of Lemma \ref{max-cpt-centralizer-sopq}.
\end{proof}

We need some more  notation to state Theorem \ref{max-cpt-sp-pq-wrt-onb}. For $d \,\in\, \N_\d$, define
$$
\CC_+ (d) \,:=\, \CC^0_+ (d) \vee \cdots \vee  \CC^{d-1}_+ (d) \ \text{ and } \ 
\CC_- (d) \,:=\, \CC^0_- (d) \vee \cdots \vee  \CC^{d-1}_- (d).
$$
Let $\alpha \,:=\, \# \E_\d $,  $\beta \,:=\, \# \O^1_\d$ and $ \gamma \,:=\, \# \O^3_\d $.
We enumerate $\E_\d =\{ \eta_i \,\mid\, 1 \,\leq\, i \,\leq\, \alpha \}$ such that
$\eta_i \,<\, \eta_{i+1}$, 
$\O^1_\d \,=\,\{ \theta_j \,\mid\, 1 \,\leq\, j \,\leq\, \beta \}$ such that $\theta_j
\,<\, \theta_{j+1}$ and similarly
$\O^3_\d \,=\,\{ \zeta_j \,\mid\, 1 \,\leq \,j \,\leq\, \gamma \}$ such that
$\zeta_j \,<\, \zeta_{j+1}$. Now define
$$
\EC_+ := \CC_+ (\eta_1) \vee \cdots \vee \CC_+ (\eta_{\alpha})\, ; \ \  \,
\OC^1_+ := \CC_+ (\theta_1) \vee \cdots \vee \CC_+ (\theta_{\beta})\, ; \ \ \,\OC^3_+ := \CC_+ (\zeta_1) \vee \cdots \vee \CC_+ (\zeta_{\gamma});
$$
$$
\EC_- := \CC_- (\eta_1) \vee \cdots \vee \CC_- (\eta_{\alpha}) ; \ 
\OC^1_- := \CC_- (\theta_1) \vee \cdots \vee  \CC_- (\theta_{\beta})\, \text{ and } \,\OC^3_- := \CC_- (\zeta_1) \vee \cdots \vee  \CC_- (\zeta_{\gamma})  .
$$
Finally we define
\begin{equation}\label{orthogonal-basis-sp-pq-final}
 \HC_+ := \EC_+ \vee \OC^1_+ \vee\OC^3_+ , \ \ \HC_- := \EC_- \vee \OC^1_- \vee\OC^3_- \ \text{ and } \
\HC := \HC_+ \vee \HC_-.
\end{equation}
It is clear that $\HC$ is a standard orthogonal basis with
$\HC_+ \,=\, \{ v \,\in\, \HC \,\mid \langle v,\, v \rangle \,=\,1 \}$ and 
$\HC_- \,=\, \{ v \,\in\, \HC \,\mid\, \langle v,\, v \rangle \,=\,-1 \}$. In
particular, $\# \HC_+ \,=\, p$ and $\# \HC_- \,=\,q$.

For a complex matrix $A\in {\rm M}_m(\C)$, let ${\rm Re\,}A \in {\rm M}_m(\R)$ denote the real part of $A$ and ${\rm Im\,}A\in {\rm M}_m(\R)$ denote the imaginary part of $A$. Thus 
$A = {\rm Re\,}A + \sqrt{-1}\,{\rm Im\,}A $. The $\R$-algebra
$$
\prod_{i=1}^\alpha  {\rm M}_{t_{\eta_i}} (\C) \times \prod_{j=1}^\beta
\big( {\rm M}_{p_{\theta_j}} (\H)\times {\rm M}_{q_{\theta_j}} (\H) \big)
\times \prod_{k=1}^\gamma \big( {\rm M}_{p_{\zeta_k}} (\H) \times {\rm M}_{q_{\zeta_k}} (\H) \big)
$$
is embedded into ${\rm M}_p(\H)$ and ${\rm M}_q(\H)$ in the following two ways:
\begin{align}\label{map-Dp-sp-pq}
\Db_p \colon \prod_{i=1}^\alpha  &{\rm M}_{t_{\eta_i}} (\C)  \times  \prod_{j=1}^\beta  \big( {\rm M}_{p_{\theta_j}} (\H) \times {\rm M}_{q_{\theta_j}} (\H) \big) \times \prod_{k=1}^\gamma \big( {\rm M}_{p_{\zeta_k}} (\H) \times {\rm M}_{q_{\zeta_k}} (\H) \big) \, \longrightarrow \, {\rm M}_p(\H)  
\\
  \big(\,A_{\eta_1},\,\,  &\dotsc ,   A_{\eta_{\alpha}}\,; 
  \, C_{\theta_1},\, D_{\theta_1} ,\, \dotsc\, ,  C_{\theta_\beta},\, D_{\theta_\beta}\,;\, E_{\zeta_1},\,  F_{\zeta_1},\, \dotsc\, ,  E_{\zeta_\gamma},\, F_{\zeta_\gamma} \,  \big)  \nonumber  \\
 \longmapsto &\, \bigoplus_{i=1}^\alpha  \big( {\rm Re\,}A_{\eta_i} +\jb  \,{\rm Im\,}A_{\eta_i} \big)_\blacktriangle^{\eta_i /2}
 \oplus \bigoplus_{j=1}^\beta \Big( \big( C_{\theta_j} \oplus D_{\theta_j}\big)_\blacktriangle ^{\frac{\theta_j-1}{4}} \oplus C_{\theta_j} \oplus \big( C_{\theta_j} \oplus D_{\theta_j}\big)_\blacktriangle ^{\frac{\theta_j-1}{4}} \Big)    \nonumber \\
      \oplus &\, \bigoplus_{k=1}^\gamma \Big( \big( E_{\zeta_k}\oplus F_{\zeta_k} \big) _\blacktriangle ^{\frac{\zeta_k + 1}{4}}  \oplus \big( F_{\zeta_k}\oplus E_{\zeta_k} \big) _\blacktriangle ^{\frac{\zeta_k -3}{4}}  \oplus F_{\zeta_k} \Big) ,\,   \nonumber
\end{align}
and 
\begin{align}\label{map-Dq-sp-pq}
\Db_q \colon \prod_{i=1}^\alpha & {\rm M}_{t_{\eta_i}} (\C)  \times \prod_{j=1}^\beta   \big( {\rm M}_{p_{\theta_j}} (\H) \times {\rm M}_{q_{\theta_j}} (\H) \big) \times \prod_{k=1}^\gamma \big( {\rm M}_{p_{\zeta_k}} (\H) \times {\rm M}_{q_{\zeta_k}} (\H) \big) \, \longrightarrow \, {\rm M}_p(\H)  
\\
  \big(\,A_{\eta_1} \,\,,&  \dotsc ,   A_{\eta_{\alpha}}  \,; 
 \,  C_{\theta_1}, \,D_{\theta_1} ,\, \dotsc\, ,  C_{\theta_\beta}, \, D_{\theta_\beta}\,;\, E_{\zeta_1},\, F_{\zeta_1},\, \dotsc\, ,  E_{\zeta_\gamma}, \, F_{\zeta_\gamma} \,  \big)    \nonumber\\
 \longmapsto& \,\bigoplus_{i=1}^\alpha  \big( {\rm Re\,}A_{\eta_i} +\jb  \,{\rm Im\,}A_{\eta_i} \big)_\blacktriangle^{\eta_i /2}
 \oplus \bigoplus_{j=1}^\beta \Big( \big( D_{\theta_j} \oplus C_{\theta_j}\big)_\blacktriangle ^{\frac{\theta_j-1}{4}} \oplus D_{\theta_j} \oplus \big( D_{\theta_j} \oplus C_{\theta_j}\big)_\blacktriangle ^{\frac{\theta_j-1}{4}} \Big)    \nonumber\\
  \oplus & \bigoplus_{k=1}^\gamma \Big( \big( F_{\zeta_k}\oplus E_{\zeta_k} \big) _\blacktriangle ^{\frac{\zeta_k + 1}{4}}  \oplus \big( E_{\zeta_k}\oplus F_{\zeta_k} \big) _\blacktriangle ^{\frac{\zeta_k -3}{4}}  \oplus E_{\zeta_k} \Big)\nonumber .
\end{align}

 Let $\Lambda_\HC \,\colon\, {\rm End}_\H \H^{n} \,\longrightarrow\, {\rm M}_n (\H)$ be the isomorphism of $\R$-algebras induced by the ordered basis $\HC$ in \eqref{orthogonal-basis-sp-pq-final}. 
 Let $M$ be the maximal compact subgroup of ${\rm Sp}(p,q)$ which leaves invariant simultaneously the two subspace spanned by $\HC_+$ and $\HC_-$. Clearly, $\Lambda_\HC (M) = {\rm Sp} (p) \times {\rm Sp} (q)$. 

\begin{theorem}\label{max-cpt-sp-pq-wrt-onb} 
Let $X \in \NC _{\s\p (p,q)}$, $\Psi_{{\rm Sp} (p,q)}(\OC_X) \,=\, ( \d ,\, \sgn_{\OC_X} )$.
Let $\alpha \,:=\, \# \E_\d $,  $\beta \,:=\, \# \O^1_\d$ and $ \gamma \,:=\, \# \O^3_\d $.
Let $\{X,H,Y\}$ be a $\s\l_2(\R)$-triple in $\s\p (p,q)$ and 
$(p_\theta,\, q_\theta)$ be the signature of the form $(\cdot, \cdot)_\theta$, for all $\theta \,\in\, \O_\d$.
Let $K$ be the maximal compact subgroup of $\ZC_{{\rm Sp} (p,q)} ( X, H, Y)$ as in Lemma \ref{max-cpt-sppq}. 
Let  the maps $ \Db_p$ and $\Db_q$ be defined as in \eqref{map-Dp-sp-pq} and  \eqref{map-Dq-sp-pq}, respectively. 
Then  $\Lambda_\HC (K) \,\subset \, {\rm Sp} (p) \times {\rm Sp} (q)$ is given by 
\begin{align*} 
\Lambda_\HC(K)\,=\,
 \Bigg\{ \Db_p (g) \oplus \Db_q(g) \,  \Biggm| \, \begin{array}{c}
                       g \in \prod_{i=1}^\alpha {\rm U}(t_{\eta_i})  \times \prod_{j=1}^\beta \big( {\rm Sp}(p_{\theta_j}) \times {\rm Sp}(q_{\theta_j}) \big) \vspace{.12cm}\\
                     \times \prod_{k=1}^\gamma \big( {\rm Sp}(p_{\zeta_k}) \times {\rm Sp}(q_{\zeta_k})\big)
                                        \end{array} \Bigg\}.
\end{align*}
Moreover, the nilpotent orbit $\OC_X$ in $\s\p(p,q)$ is homotopic to $({\rm Sp}(p) \times{\rm Sp}(q))/\Lambda_\HC(K)$.
\end{theorem}
 
\begin{proof}
This follows by writing the matrices of the elements of the maximal compact subgroup $K$ in Lemma 
\ref{max-cpt-sppq} with respect to the ordered basis $\HC$ as in \eqref{orthogonal-basis-sp-pq-final}.

The second part follows from Theorem \ref{mostow-corollary} and the well-known fact that any maximal compact subgroup of ${\rm Sp}(p,q)$ is isomorphic to ${\rm Sp}(p)\times {\rm Sp}(q)$.
\end{proof}

\section*{Appendix}
 Let $ G $ be a real Lie group with Lie algebra $\g$. We further assume that the Lie algebra $ \g $ is simple and of classical type.
Here we give the descriptions of the reductive part of the centralizers in $G$ of nilpotent elements in $\g$ when $\g$ is
a subalgebra of a matrix algebra over $\H$.  
When $\g$ is isomorphic to a complex simple Lie algebra or one of the Lie algebras  $\g\l_n(\R)$ or $\u(p,q)$ or $\o(p,q)$ or $\s\p(n,\R)$ such descriptions 
are due to  Springer and Steinberg; see \cite[1.8,p.E-85; 2.25,p.E-95]{SS} and \cite[Theorem 6.1.3]{CoM}.
However, when the classical simple Lie algebras are matrix subalgebras with entries from $\H$, we are unable to locate such descriptions in the existing literature. Further, as mentioned in \cite{BCM}, the noncommutativity of $\H$ creates technical difficulties 
and does not allow direct 
extensions of the results of \cite{SS} to the case of simple Lie algebras involving $\H$. 
In \cite{BCM}  the description of the reductive part, as above, was needed in the case when $\g$ simple matrix Lie algebra involving $\H$. The reasons mentioned above together with the foregoing
requirements in \cite{BCM} led us to doing the computations.
The first key point in our computations is the well-known fact
(see \cite[Lemma 3.7.3, p. 50]{CoM})
 that the centralizer of a $\s \l_2(\R)$-triple containing a nilpotent element constitutes a reductive part of the centralizer of the nilpotent element itself.
In view of the  Jacobson–Morozov theorem and the above result  we use \cite[Lemma 4.4]{BCM} which was deduced   applying the basics of $\s\l_2(\R)$-representation theory. 
The above method was indicated  in  \cite[Section 3.1–3.3]{M} and in \cite[Section 9.3]{CoM}.  The  advantage
of this method lies in the uniform manner it deals with all the cases of simple classical Lie algebras involving $\R, \C$ and $\H$. 
We record these computations below and these results should be viewed as complementary to those
in \cite[1.8,p.E-85; 2.25,p.E-95]{SS} and in \cite[Theorem 6.1.3]{CoM}.

In what follows we will use the notation as defined in \S \ref{sec-notation}, and in particular, we will employ
the symbols $\N_\d$, $\E_\d$ and $\O_\d$ as given in \eqref{Nd-Ed-Od}.

\begin{proposition}\label{centralizer-sl2-triple-sl-n-H}
 	Let $X\in \s\l_n(\H) $ be a	non-zero nilpotent element, and $ \{X,H,Y\} $ be a $ \s\l_{2}(\R)$-triple in $\s\l_n(\H) $. Let the nilpotent orbit $ \OC_X$ corresponds to the partition $ \d\in \PC(n) $. Then
 	$$
 	\ZC_{{\rm SL}_n(\H)}(X,H,Y) \simeq  S\Big( \prod_{d \in \N_\d} {\rm GL}_{t_d}(\H)_\Delta^{d}\Big)	\,. 
 	$$
 \end{proposition}
 
\begin{proof}
 The proof follows from \cite[Lemma 4.4(2)]{BCM}.
\end{proof}

\begin{proposition}\label{centralizer-sl2-triple-so*}
Let $X\in \s\o^*(2n) $ be a	non-zero nilpotent element, and $ \{X,\,H,\,Y\} $ be a $ \s\l_{2}(\R) $-triple in
$\s\o^*(2n) $. Let the nilpotent orbit $ \OC_X$ corresponds to the signed Young diagram of size $n$.
Let $ p_\eta$ (respectively, $ q_\eta)$ be the number of $ +1 $  (respectively, $ -1 $) in the $ 1^{\rm st} $ column of the block of size $ t_\eta \times \eta $ for $ \eta\in \E_\d$.
Then
$$
\ZC_{{\rm SO}^*(2n)}(X,H,Y) \simeq \prod_{\theta \in \O_\d} {\rm SO^*}(2t_\theta) \times  \prod_{\eta \in \E_\d} {\rm Sp}(p_\eta, q_\eta)	\,. 
$$
\end{proposition}

\begin{proof}
Recall that for $d \in \N_\d$, the form $(\cdot,\cdot)_d$ as in \eqref{new-form} is Hermitian or skew-Hermitian according as $d$ is even or odd.
For the nilpotent element $X$, the signature of $(\cdot,\cdot)_\eta$ on $L(\eta-1)$ is $(p_\eta,q_\eta)$ for $\eta \in \E_\d$. Now the proof follows from \cite[Lemma 4.4(4)]{BCM}.
\end{proof}
 
\begin{proposition}\label{centralizer-sl2-triple-sp-pq}
Let $X\in \s\p(p,q) $ be a	non-zero nilpotent element, and $ \{X,H,Y\} $ be a $ \s\l_{2} (\R)$-triple in $\s\p(p,q) $. Let the nilpotent orbit $ \OC_X$ corresponds to the signed Young diagram of signature $(p,\,q)$, where
$ p_\theta$ (respectively, $ q_\theta)$ denotes the number of $ +1 $ (respectively, $ -1 $) in the $ 1^{\rm st} $ column of the block of size $t_\theta \times \theta $ for $ \theta\in \O_\d$. Then
	$$
	\ZC_{{\rm Sp}(p,q)}(X,H,Y) \simeq \prod_{\eta \in \E_\d} {\rm SO^*}(2t_\eta) \times  \prod_{\theta \in \O_\d} {\rm Sp}(p_\theta, q_\theta)	\,. 
	$$
\end{proposition}

\begin{proof} Recall that for $d \in \N_\d$, the form $(\cdot,\cdot)_d$ as in \eqref{new-form} is Hermitian or skew-Hermitian according as $d$ is odd or even.
 Now the proof follows from the fact that the signature of $(\cdot,\cdot)_\theta$ on $L(\theta-1)$ is $(p_\theta,q_\theta)$ for $\theta \in \O_\d$ and \cite[Lemma 4.4(4)]{BCM}.
\end{proof}
 
\section*{Acknowledgements}
Indranil Biswas is supported by a J. C. Bose Fellowship.
Pralay Chatterjee acknowledges support from the SERB-DST MATRICS project: MTR/2020/000007.
Chandan Maity is supported by an NBHM PDF during the course of this work.


\begin{thebibliography}{AAAA} 
\bibitem[Al]{A} A. Alexeevski, Component groups of the centralizers of unipotent elements in semisimple algebraic groups, in: {\it Lie groups and invariant theory}, Amer. Math. 
Soc. Transl. Ser. 2, 213, (Amer. Math. Soc., Providence, RI, 2005), pp. 15--31.

\bibitem[ABB]{ABB} H. Azad, E. van den Ban and I. Biswas, Symplectic geometry of 
semisimple orbits. {\it Indag. Math.} {\bf 19} (2008), no. 4, 507--533.

\bibitem[Bo]{Bo} A. Borel, {\it Linear algebraic groups}, Second Enlarged Edition, 
Graduate Texts in Mathematics, 126, Springer-Verlag, New York, 1991.

\bibitem[BC]{BC} I. Biswas and P. Chatterjee, On the exactness of 
Kostant-Kirillov form and the second cohomology of nilpotent orbits, {\it Internat. J. 
Math.} {\bf 23} (2012), no. 8, 1250086.

\bibitem[BCM]{BCM} I. Biswas, P. Chatterjee and C. Maity, The Second cohomology of nilpotent orbits in classical Lie algebras, {\it Kyoto J. Math.}
{\bf 60} (2020), no. 2,  717--799.

\bibitem[ChMa]{CM} P. Chatterjee and C. Maity, On the second cohomology of
nilpotent orbits in exceptional Lie algebras, {\it Bull. Sci.
Math.} {\bf 141} (2017), no. 1,  10--24.
  
\bibitem[CoMc]{CoM} D. H. Collingwood and W. M. McGovern, {\it Nilpotent orbits in 
semisimple Lie algebras}, Van Nostrand Reinhold Mathematics Series, Van Nostrand 
Reinhold Co., New York, 1993.

\bibitem[Cr]{Cr} P. Crooks, The torus-equivariant cohomology of nilpotent orbits, {\it J. Lie 
Theory}, {\bf 25} (2015), no. 4,  1073--1087.

\bibitem[Ho]{H} G. Hochschild, {\it The structure of Lie groups}, Holden-Day, Inc., 
San Francisco, 1965.

\bibitem[Ju]{Ju} D. Juteau, Cohomology of the minimal nilpotent orbit, 
{\it Transform. Gr.} {\bf 13} (2008),  no. 2, 355--387.

\bibitem[Ki]{Ki} D. R. King, The component groups of nilpotents in exceptional simple real Lie algebras, {\it Comm. Alg.} {\bf 20} (1992), no. 1, 219--284.

\bibitem[Mc]{M} W. M. McGovern, The adjoint representation and the adjoint 
action, in: {\it Algebraic quotients. Torus actions and cohomology. The adjoint 
representation and the adjoint action}, 159--238, Encyclopaedia Math. Sci., 131, Invariant Theory Algebr.Transform. Groups, II, 
Springer, Berlin, 2002.
 
\bibitem[Mo]{Mo} G. D. Mostow, On covariant fiberings of Klein spaces, {\it Amer. Jour. Math.} 
{\bf 77} (1955), 247--278.

\bibitem[So]{So} E. Sommers, A generalization of the Bala--Carter
theorem for nilpotent orbits, {\it Internat. Math. Res. Notices},
 (1998), no. 11, 539--562.
 
\bibitem[SpSt]{SS} T. A. Springer and R. Steinberg, 
Conjugacy classes, in : {\it 1970 Seminar on Algebraic Groups and Related Finite
Groups} (The Institute for Advanced Study, Princeton, N.J., 1968/69) pp. 167--266,
Lecture Notes in Mathematics, Vol. 131 Springer, Berlin.

\end{thebibliography}
\end{document}